\documentclass[12pt,reqno,a4paper]{amsart}
\newcommand\version{June 15th, 2026}

\usepackage{amsmath,amsfonts,amsthm,amssymb,amsxtra,enumerate,mathtools,mathabx}
\usepackage[normalem]{ulem}
\usepackage[utf8]{inputenc}
\usepackage[T1]{fontenc}
\usepackage{lmodern}
\usepackage{bbm}
\usepackage{mathrsfs}
\usepackage{enumerate}
\usepackage{comment}
\usepackage{placeins}
\usepackage{pgfplots}
\pgfplotsset{compat=newest}
\usepgfplotslibrary{fillbetween}
\usepackage{mathabx}
\usepackage{float}

\date{\version}

\newtheorem{theorem}{Theorem}[section]

\newtheorem{lemma}[theorem]{Lemma}
\newtheorem{corollary}[theorem]{Corollary}
\theoremstyle{definition}

\theoremstyle{remark}

\newtheorem{remark}[theorem]{Remark}

\numberwithin{equation}{section}

\newcommand{\1}{\mathbbm{1}}

\newcommand{\C}{\mathbb{C}}

\newcommand{\dd}{\, \mathrm{d}}

\renewcommand{\epsilon}{\varepsilon}
\newcommand{\e}{\mathrm{e}}
\newcommand{\ii}{\mathrm{i}}

\newcommand{\F}{\mathcal{F}}

\newcommand{\N}{\mathbb{N}}
\newcommand{\norm}[2][]{{\left\|#2\right\|_{#1}}} 
\newcommand{\abs}[2][]{{\left\vert#2\right\vert}} 

\renewcommand{\phi}{\varphi}
\newcommand{\R}{\mathbb{R}}

\newcommand{\sclp}[2][]{{\left\langle#2\right\rangle} _{#1}}

\newcommand{\Z}{\mathbb{Z}}

\newcommand{\Chat}{\widehat{\C}}

\newcommand{\pot}[1][q]{#1C_+\overline{#1}}
\newcommand{\ov}[1]{\overline{#1}}
\newcommand{\mn}{c} 

\DeclareMathOperator{\dom}{dom}
\DeclareMathOperator{\im}{Im}
\DeclareMathOperator{\ran}{ran}
\DeclareMathOperator{\re}{Re}

\DeclareMathOperator{\Tr}{Tr}

\newcommand{\CMop}[1][q]{L_{#1}}
\DeclareMathOperator{\ee}{\mathbf{\e}}

\let\oldtocsection=\tocsection
\let\oldtocsubsection=\tocsubsection
\let\oldtocsubsubsection=\tocsubsubsection
\renewcommand{\tocsection}[2]{\hspace{0em}\oldtocsection{#1}{#2}}
\renewcommand{\tocsubsection}[2]{\hspace{0.75cm}\oldtocsubsection{#1}{#2}}
\renewcommand{\tocsubsubsection}[2]{\hspace{2em}\oldtocsubsubsection{#1}{#2}}

\usepackage{hyperref}
\begin{document}
	
	\title[Jost solutions and direct scattering for the Calogero--Moser equation
	]
	{Jost solutions and direct scattering\\ for the continuum Calogero--Moser equation}

    \author[Rupert L. Frank]{Rupert L. Frank}
    \address[Rupert L. Frank]{Mathematisches Institut, Ludwig-Maximilans Universit\"at M\"unchen, Theresienstr.~39, 80333 M\"unchen, Germany, and Munich Center for Quantum Science and Technology, Schellingstr.~4, 80799 M\"unchen, Germany, and Mathematics 253-37, Caltech, Pasadena, CA 91125, USA}
	\email{r.frank@lmu.de}
 
    \author[Larry Read]{Larry Read}
	\address[Larry Read]{Laboratoire de mathématiques d’Orsay, CNRS, Université Paris-Saclay, 91405 Orsay, France}
	\email{larry.read@universite-paris-saclay.fr}

    


\begin{abstract}
     We propose an inverse scattering transform for the continuum Calogero--Moser equation. We give a rigorous treatment of the direct scattering problem by constructing the associated Jost solutions and introducing a distorted Fourier transform, as well as deriving trace formulas for the eigenvalues of the Lax operator.
 \end{abstract}
 
\maketitle
\setcounter{tocdepth}{1}
\tableofcontents

\phantomsection
\addcontentsline{toc}{part}{Overview}{}

\section{Introduction}  

Our objectives in this paper are to:
\begin{itemize}
    \item develop a direct scattering theory,
    \item lay the foundations of an inverse scattering theory, and
    \item derive trace formulas
\end{itemize}
for the operator
\begin{align*}
    L_q=-\ii \partial_x-\pot \quad \text{in}\ L^2_+(\R) \,.
\end{align*}
Here $L^2_+(\R)$ denotes the Hardy space on the real line, that is, the subspace of $L^2(\R)$ consisting of functions whose Fourier transforms are supported in $[0,\infty)$. The operator $C_+$ is the Cauchy--Szeg\H{o} projector, that is, the orthogonal projection from $L^2(\R)$ to $L^2_+(\R)$. The function $q$ is assumed to belong to $L^2_+(\R)$ and we do not distinguish in our notation between $q$ and the operator of multiplication by $q$.

Our interest in the operator $L_q$ comes from the continuum Calogero--Moser equation (aka Calogero--Moser derivative nonlinear Schr\"odinger equation), given by
\begin{equation}\label{eqn:CMeqn}
    \ii \partial_t q=-\partial_x^2 q+ 2\ii qC_+ \partial_x \left(\abs{q}^2\right).
\end{equation}
This equation was first introduced in \cite{abanov2009integrable} as a continuum limit of classical Calogero--Moser systems, which describe particles interacting pairwise through an inverse square potential. 

Equation \eqref{eqn:CMeqn} is completely integrable. Its detailed analysis was initiated by G\'erard and Lenzmann in \cite{gerard_calogero--moser_2023}. They described the Lax structure, constructed soliton and multi-soliton solutions, and established local well-posedness in $H^n_+(\R)\coloneq C_+H^n(\R)$ for integer $n\geq 1$, and global well-posedness in the same class under the mass constraint $\norm{q}_2^2\leq 2\pi$. 

Subsequently, in \cite{killip_scaling_2025} Killip, Laurens and Visan proved global well-posedness of \eqref{eqn:CMeqn} in the scaling critical space $L^2_+(\R)$, assuming the mass constraint $\norm{q}_2^2 < 2\pi$. This threshold is optimal \cite{hogan2024turbulent,kim25_construction}.

In the last few years, there has been a rather large number of papers devoted to the study of \eqref{eqn:CMeqn}, among others \cite{badreddine24_ontheglobal,badreddine25_traveling,badreddine24_zerodispersion,hogan2024turbulent,jeong24_quantized,kim25_construction,kim2024soliton,chen2025traveling,chen25_thedefocusing}.

The linear operator $L_q$ that we will study in this paper appears as the Lax operator corresponding to \eqref{eqn:CMeqn}. It plays a crucial role in \cite{gerard_calogero--moser_2023,killip_scaling_2025}.

\medskip

When saying that in this paper we develop a \emph{direct scattering theory} for this operator we mean, among other things, that we introduce `generalised eigenfunctions of the continuous spectrum' of the operator $L_q$. These are the analogues of the Jost functions that appear in the spectral analysis of the one-dimensional Schr\"odinger operator, which is the Lax operator corresponding to the KdV equation. Similar analogues of Jost functions also appear for Lax operators in structurally related equations, like the cubic nonlinear Schr\"odinger equation and the Benjamin--Ono equation. Using these generalised eigenfunctions together with the usual eigenfunctions, we construct a `distorted Fourier transform', which is a surjective partial isometry in $L^2_+(\R)$ that diagonalises the absolutely continuous part of $L_q$.

Besides these generalised eigenfunctions, we also introduce two scattering coefficients $\beta$ and $\Gamma$, defined on the continuous spectrum of $L_q$. The function $\Gamma$ is unimodular and represents the scattering matrix that connects one family of generalised eigenfunctions to another one. The coefficient $\beta$ is a characteristic of the continuous spectrum of $L_q$, on which we comment further below.

In addition to these data of the continuous spectrum, we also consider the eigenvalues $\lambda_j$ and certain corresponding constants $\gamma_j$ that we introduce. It is known \cite{gerard_calogero--moser_2023} that the eigenvalues may be negative or embedded in the continuous spectrum, but they are always simple and finite in number.

We propose that the function $\beta$ together with the finite sequences $\{\lambda_j\}_j$ and $\{\gamma_j\}_j$ constitute a complete set of spectral data of the operator $L_q$.

We argue formally that these spectral data evolve in a very simple way under the flow of the continuuum Calogero--Moser equation \eqref{eqn:CMeqn}. Indeed, we have
$$
\lambda_j(t) = \lambda_j(0) \,,
\qquad
\gamma_j(t) = -2\lambda_jt + \gamma_j(0) \,,
\qquad
\beta(\lambda,t) = \e^{-i\lambda^2 t} \beta(\lambda,0) \,.
$$
This argument is, for the moment, only formal, since it is not yet clear that the function $q(t)$ belongs to the class of functions for which we can develop our direct scattering theory. 

In any case, this simple evolution of the spectral data offers the possibility of solving the continuum Calogero--Moser equation \eqref{eqn:CMeqn} by inverse scattering theory, at least for a certain class of initial data. Needless to say, the solutions of other, structurally similar equations by inverse scattering theory is classical by now; see the foundational papers \cite{ggkm,lax,zaharov_faddeev,zakharov_shabat,akns,fokas_inverse_1983} and the textbooks \cite{ablowitz_segur,faddeev_takhtajan,ablowitz_clarkson} for a guide to the enormous literature.

\medskip

When saying that in this paper we lay the foundations of an \emph{inverse scattering theory} we mean, among other things, that we prove two formulas for reconstructing $q$ from the spectral data of $L_q$. The first formula, \eqref{eqn:qreconstruction}, involves the eigenfunctions as well as their generalised analogues, together with the scattering coefficient $\beta$. The second formula, \eqref{eqn:qreconstruct2}, involves a secondary family of Jost-ish functions, which is different, but related to the family of generalised eigenfunctions. This secondary family of functions has the advantage of being analytic off the spectrum of $L_q$, and recovering this family from the spectral data $\beta$, $\{\lambda_j\}_j$, $\{\beta_j\}_j$ leads to a certain nonlocal Hilbert--Riemann problem. The analysis of this problem remains open.

\medskip

When saying that in this paper we derive \emph{trace formulas}, we mean that we prove a sequence of identities that relate quantities coming from the discrete spectrum of $L_q$ (namely the eigenvalues $\lambda_j$), from its continuous spectrum (namely the scattering coefficient $\beta$) and from the coefficients of the operator (namely the function $q$). In the context of the Schr\"odinger operator, which is the Lax operator for the KdV equation, such identities go back to Zaharov and Faddeev \cite{zaharov_faddeev} and are well known. They have also been shown for the Lax operators of other integrable equations including the cubic nonlinear Schr\"odinger equation \cite{flaschka_newell} (see also \cite[Section I.7]{faddeev_takhtajan}) and the Benjamin--Ono equation \cite{kaup_inverse_1998}. 

Trace formulas have been useful both in linear and nonlinear problems. For a review concerning the linear case we refer to \cite{killip07_spectral}. Their usefulness in nonlinear problems comes from the fact that all three contributions to the trace formulas correspond to quantities that are conserved under the corresponding nonlinear flow. For recent applications of this idea, see, for instance, \cite{rybkin,killip18_low,koch_tataru}.

In the context of the continuum Calogero--Moser equation the conservation of the quantities coming from the discrete spectrum and the coefficients of the operator is known \cite{gerard_calogero--moser_2023}, but the definition of a contribution from the continuous spectrum and its conservation seem to be new. In particular, this quantity yields the `missing term' in the eigenvalue estimate of G\'erard and Lenzmann \cite{gerard_calogero--moser_2023}, see \eqref{eqn:Nfinbound} and Theorem \ref{thm:traceformula1}.

\medskip

This concludes our brief overview of our results in this paper. We describe them in more detail in Section \ref{sec:invspectrans} below. 

\medskip

Part of our work is motivated by similar results on the Benjamin--Ono equation. An inverse scattering transform for this equation was first proposed by Ablowitz and Fokas \cite{fokas_inverse_1983}, and later made rigorous by Coifman and Wickerhauser \cite{coifman_1990} for small initial data in a weighted $L^1$ class. Wu \cite{wu_simplicity_2016,wu_jost_2017} extended the direct problem to much larger weighted spaces without size restrictions; however, the corresponding inverse problem for such data remains unresolved, as far as we know.

Notably, the Lax operators associated with both the Benjamin–Ono and continuum Calogero--Moser systems are relatively compact perturbations of the same first-order differential operator acting on the same Hardy space. This structural similarity allows us in some places to argue analogously to Wu \cite{wu_simplicity_2016,wu_jost_2017}. We stress however that this analogy is more on a conceptual level, while on a technical level there are several subtle, but crucial differences. Among those is the existence of embedded eigenvalues in our case and the resulting difficulties when expanding the Jost functions and scattering coefficients thereabout. Other differences stem from the different form of the free resolvent kernel, which leads to better behaviour for small energies (in particular, in our case there are no zero-energy resonances) but worse behaviour at large energies. A technical difference is that Wu proves the existence of Jost(-ish) functions by solving an equation in a weighted $L^\infty$-space, while we solve it in $L^2$. This allows us to work with more modest assumptions on the potential.

Moreover, we spend the first part of our paper constructing the distorted Fourier transform. We hope that this will be useful for studying the nonlinear problem \eqref{eqn:CMeqn}; see for example how its analogue was applied recently for the Benjamin--Ono equation in \cite{gerardmiller}. We also prove the absence of singular continuous spectrum and derive trace formulas, which are topics that are not covered by Wu.


\section{Description of results} \label{sec:invspectrans}

Let $q\in L^2_+(\R)$. Then, as shown in \cite[Lemma 2.1]{killip_sharp_2023}, $L_q$ is a relatively compact perturbation of $L_0$ and therefore $L_q$ can be defined as a self-adjoint, lower semibounded operator in $L^2_+(\R)$ with operator domain $H^1_+(\R)$, having essential spectrum $[0,\infty)$. In \cite[Lemma 2.6]{killip_sharp_2023} it is shown that the difference of resolvents $(L_q+\kappa)^{-1} - (L_0+\kappa)^{-1}$ is trace class for all sufficiently large $\kappa$, which by the Birman--Krein theorem \cite{birmankrein} implies that the absolutely continuous spectrum of $L_q$ coincides with $[0,\infty)$.

As shown in \cite{gerard_calogero--moser_2023}, the point spectrum of the operator $\CMop$ consists of finitely many simple eigenvalues $\{\lambda_j\}_{j=1}^N$. We emphasise that these eigenvalues may be embedded in the continuous spectrum $[0,\infty)$. Indeed, if $\CMop$ has $\lambda$ as an eigenvalue with eigenfunction $\phi$, then for any $a\geq 0$ the operator $\CMop[\e^{\ii a x}q]$ has $\lambda+a$ as an eigenvalue with eigenfunction $\e^{\ii ax}\phi$.

The resolvent set, the spectrum and the point spectrum of $\CMop$ are denoted by $\rho(\CMop)$, $\sigma(\CMop)$ and $\sigma_{\rm p}(\CMop)$, respectively. We write $\lambda_j$ and $\phi_j$ for the eigenvalues and the corresponding (appropriately normalised) eigenfunctions.


\subsection*{Direct scattering I}

In Section \ref{sec:lap} we shall show that for any $\lambda\in[0,\infty)\setminus\sigma_\mathrm{p}(\CMop)$ and either choice $\pm$ of sign there is a bounded function $m_\e(\lambda\pm0\ii)$ on $\R$ that solves
\begin{align*}
    (-\ii\partial_x - qC_+\overline{q}) m_\e(\lambda\pm0\ii)=\lambda \, m_\e(\lambda\pm0\ii)
\end{align*}
and satisfies
$$
m_\e(x,\lambda\pm0\ii)-\e^{\ii\lambda x}\rightarrow 0
\qquad\text{as}\ x\rightarrow \mp \infty \,.
$$
We refer to the functions $m_\e(\lambda\pm0\ii)$ as Jost functions in analogy with the corresponding solutions of the Schr\"odinger equation; see, e.g., \cite[Chapter 4]{yafaev_mathematical_2010}.

A by-product of the arguments leading to the existence of Jost functions is a proof of the absence of singular continuous spectrum of the operator $\CMop$; see Theorem \ref{thm:nosc}. Moreover, in terms of $m_\e(\lambda-0\ii)$ we can define a surjective partial isometry that diagonalises the absolutely continuous part of $L_q$; see Theorem \ref{thm:meFT}. This operator is a `distorted Fourier transform' and is again reminiscent of corresponding operators in the theory of the Schr\"odinger equation.

Next, we introduce the scattering matrix $\Gamma(\lambda)$ for $\lambda\in[0,\infty)\setminus\mathcal N$. It can be expressed in terms of the Jost solutions, see \eqref{eqn:Gammadef}. In particular, $\Gamma$ relates one set of generalised eigenfunctions $m_\e(\lambda-0\ii)$ to the other $m_\e(\lambda+0\ii)$ via
$$
m_\e(\lambda+0\ii) = \Gamma(\lambda)\, m_\e(\lambda-0\ii) \,.
$$
The number $\Gamma(\lambda)$ is unimodular. The term `scattering matrix' is clarified at the end of Section \ref{sec:distortedFT} in the formalism of abstract mathematical scattering theory; see, e.g., \cite{yafaev_mathematical_1992}.


\subsection*{Direct scattering II}

Even if the proofs of the results that we have mentioned so far are often substantially different from the proofs in the Schr\"odinger case, there is a clear conceptual analogy between the results. When proceeding, however, we encounter a significant difference. While the Jost functions in the Schr\"odinger setting have an analytic continuation to the complement of the spectrum, no such continuation is possible for our $m_\e(\lambda\pm0\ii)$. 

However, as we shall show, there is a natural Jost-ish function that does have an analytic extension. It is defined through an \emph{inhomogeneous equation}. When looking for such a function we were motivated by the existence of a similar solution in the setting of the Benjamin--Ono equation \cite{kaup_inverse_1998,wu_jost_2017}, but it should be emphasised that the inhomogeneity in our case is different from that in the Benjamin--Ono case.

In order to define this function we need to impose the condition $q\in L^1(\R)$ in addition to our standing assumption $q\in L^2_+(\R)$. We introduce the function $m_0(k)$ for any $k\in(\C\setminus\R_+) \cup(\R_++0\ii)\cup(\R_+-0\ii)$ (here $\R_+=(0,\infty)$) with $k\not\in((\sigma_\mathrm p(\CMop)\cup\mathcal N)+0\ii)\cup((\sigma_\mathrm p(\CMop)\cup\mathcal N)-0\ii)$, which is a bounded function on $\R$ that solves
$$
(-i\partial_x - qC_+ \ov{q})m_0(k) = k m_0(k) + q
$$
and satisfies
$$
m_0(x,k) \to 0
\qquad
\begin{cases}
            \text{ as } \abs{x}\rightarrow \infty , & \text{if}\ k\in \C\backslash(0,\infty) \,, \\
            \text{ as }x\rightarrow \mp \infty, & \text{if}\ k=\lambda\pm 0\ii\in \R_+\pm 0\ii \,.
\end{cases}
$$
We prove the existence of this function in Section \ref{sec:jostinhom}. Moreover, for $k\in\C\setminus\sigma(\CMop)$ we have
$$
m_0(k) = (L_q-k)^{-1} q \,.
$$

There are important relations between the Jost functions $m_\e$ and $m_0$. Lemmas \ref{lem:scatobjsbeta} and \ref{lem:emederiv} show that for any $\lambda\in[0,\infty)\setminus\sigma_\mathrm{p}(\CMop)$ we have
\begin{align}\label{eq:m0merelintro}
    m_0(\lambda+0\ii)-m_0(\lambda-0\ii)=\beta(\lambda) \, m_\e(\lambda-0\ii),
\end{align}
and 
\begin{align}\label{eq:mediffeqintro}
    \ee(\lambda)\partial_\lambda(\ov{\ee(\lambda)}m_\e(\lambda-0\ii))=-\frac{\ov{\beta(\lambda)}}{2\pi\ii} \, m_0(\lambda-0\ii),
\end{align}
where $\ee(x,\lambda)\coloneqq\e^{\ii \lambda x}$. 

We can introduce a further finite set of numbers $\{\gamma_j\}_{j=1}^N$ associated to the eigenvalues $\{\lambda_j\}_{j=1}^N$. In a neighbourhood of each eigenvalue $\lambda_j$, Lemma \ref{lem:laurentexpanimprov} yields the Laurent expansion
\begin{align*}
        m_0(k)=-\frac{\ii \phi_j}{k-\lambda_j}+(\gamma_j+x)\phi_j + o_{\lambda\rightarrow \lambda_j}(1),
\end{align*}
for some $\gamma_j\in \C$, where $\phi_j$ is an appropriately normalised eigenfunction corresponding to the eigenvalue $\lambda_j$ and where $g$ is bounded near $\lambda_j$. We emphasise that this holds even for embedded eigenvalues. 

In Theorem \ref{thm:embeddexpan} we also establish the low-energy asymptotics
\begin{align*}
    \lim_{\lambda\rightarrow 0_+}\ov{\e(\lambda)}m_\e(\lambda-0\ii)=1.
\end{align*}
This allows us to solve the above differential equation \eqref{eq:mediffeqintro} for $\ov{\e(\lambda)}m_\e(\lambda-0\ii)$ and to obtain
\begin{align}\label{eq:meformulaintro}
    \ov{\ee(\lambda)}m_\e(\lambda-0\ii)=1- \frac{1}{2\pi\ii}\int_0^\lambda \ov{\beta(\lambda') \, \ee(\lambda')} \, m_0(\lambda'-0\ii)\dd\lambda' \,.
\end{align}

The integral is understood as an improper $L^\infty(\R)$-valued integral. It converges since \eqref{eq:mediffeqintro} holds on each connected component of $(0,\infty)\backslash\sigma_p(\CMop)$ while Theorem \ref{thm:embeddexpan} shows that $\lambda\mapsto\ov{\e(\lambda)}m_\e(\lambda-0\ii)$ extends to a $C(\R_+;L^\infty(\R))$ function. Summing the resulting expressions over the finite number of these connected components yields the above. 


\subsection*{Inverse scattering}
The scattering data we propose consist of:
\begin{itemize}
    \item the eigenvalues $\{\lambda_j\}_{j=1}^N$,
    \item the coefficients $\{\gamma_j\}_{j=1}^N$, and
    \item the scattering coefficient $\beta(\lambda)$ for $\lambda\in [0,\infty)$,
\end{itemize}
and the inverse scattering problem consists in recovering $q$ from the numbers $\{\lambda_j\}_{j=1}$, $\{\gamma_j\}_{j=1}$ and the function $\beta$.

To do so, we suggest to find the Jost function $m_0$ through the following properties:
\begin{enumerate}
    \item[(a)] The Laurent expansion of $m_0$ near each $\lambda_j$,
    \begin{align*}
        m_0(k)=-\frac{\ii \phi_j}{k-\lambda_j}+(\gamma_j+x)\phi_j + o_{\lambda\rightarrow \lambda_j}(1) \,.
    \end{align*}
    \item[(b)] The jump condition for $\lambda>0$,
    \begin{align*}
        m_0(\lambda+0\ii)-m_0(\lambda-0\ii)=\beta(\lambda)\left(\ee(\lambda)-\int_0^\lambda \frac{\ee(\lambda-\mu)\ov{\beta(\mu)}}{2\pi\ii}m_0(\mu-0\ii)\dd\mu\right).
    \end{align*}
\end{enumerate}
The latter relation comes from combining \eqref{eq:m0merelintro} and \eqref{eq:meformulaintro}. Once $m_0(k)$ has been found, $q$ can then be determined by the relation
\begin{align*}
    q(x)=- \lim_{\abs{k}\rightarrow\infty}k m_0(x,k),
\end{align*}
which is shown in Lemma \ref{lem:m0asmpqcts}. The problem of finding $m_0$ given the properties (a) and (b) is a nonlocal Riemann--Hilbert problem, similarly to that proposed by Fokas and Ablowitz \cite{fokas_inverse_1983} in connection with the Benjamin--Ono equation. Solving this Riemann--Hilbert problem remains an open question.

We emphasise that the solution of the inverse problem would give a way of solving the continuous Calogero--Moser equation \eqref{eqn:CMeqn} in view of the simple time evolution of the spectral data. This is further discussed in Section \ref{sec:timeevolscat}.


\subsection*{Trace formulas}

The scattering coefficient $\beta(\lambda)$, defined for $\lambda\in[0,\infty)$, gives rise to a new family of conserved quantities for the continuum Calogero--Moser equation \eqref{eqn:CMeqn}. Let $q\in \mathcal{S}_+(\R) \coloneqq \mathcal{S}(\R)\cap L^2_+(\R)$, and let $\{\lambda_j\}_{j=1}^N$ denote the eigenvalues of $\CMop$. Then, for each $n\in\N_0$, we find that  
    \begin{align*}
        \frac{1}{2\pi}\int_0^\infty \abs{\beta(\lambda)}^2\lambda^n\dd\lambda + 2\pi\sum_{j=1}^N \lambda_j^n =-\int_{-\infty}^\infty\overline{q(x)}\mn_{n+1}(x)\dd x
    \end{align*}
with 
$$
c_n\coloneqq(\CMop)^nq \,.
$$
Since the eigenvalues $\{\lambda_j\}_j$ are conserved under the Calogero--Moser flow, and since
\begin{align*}
    \sclp{c_n,q}=\sclp{(\CMop)^nq,q}
\end{align*}
define the conserved quantities of \eqref{eqn:CMeqn} found in \cite{gerard_calogero--moser_2023}, it follows that 
\begin{align*}
    \int_0^\infty \abs{\beta(\lambda)}^2\lambda^n\dd\lambda
\end{align*}
are themselves conserved quantities of the flow. In fact, our discussion in Section \ref{sec:timeevolscat} suggests that $|\beta(\lambda)|^2$ is pointwise conserved under the flow.

In addition to these trace formulas, which are of Zaharov--Faddeev-type, we also prove trace formulas of Birman--Krein-type, which are intimately related to the spectral shift function; see Theorem \ref{thm:birmankrein}.



\newpage

\part{Spectral theory}

\section{Eigenfunctions}

Throughout this section, unless specified otherwise, we assume that
$$
q\in L^2_+(\R) \,.
$$
We will be concerned with eigenfunctions of the operator $L_q$ and, more generally, with solutions of the corresponding eigenvalue equation that vanish at infinity. Our first result is that the qualitative vanishing property already implies the quantitative $L^2$-property. This result will be important in the next section when discussing the limiting absorption principle.

\begin{theorem}\label{ef}
    Let $\lambda\in\R$ and let $\psi\in L^\infty(\R)$ be locally absolutely continuous with
    $$
    -\ii\psi' - qC_+\overline q\psi = \lambda\psi \qquad\text{in}\ \R
    $$
    and $\psi(x)\to 0$ as $|x|\to\infty$. Then
    $\psi\in L^2(\R)$ and
    \begin{align}\label{eq:psidecay}
        \psi(x)=o(|x|^{-1/2})\qquad \text{as}\quad |x|\rightarrow \infty. 
    \end{align}
\end{theorem}

\begin{proof}
    \emph{Step 1.} 
    Starting from 
    \begin{align*}
         (\e^{-\ii\lambda x}\psi)'= \ii e^{-\ii\lambda x}qC_+\ov{q}\psi,
    \end{align*}
    using that $qC_+\ov{q}\psi\in L^1(\R)$ and $\psi\rightarrow 0$ at infinity, we get 
    \begin{align*}
        \psi(x)= \1_{\{x\leq 0\}}\ii \int^x_{-\infty}\e^{\ii \lambda (x-y)}qC_+\ov{q}\psi \dd y-\1_{\{x>0\}}\ii \int_{x}^\infty \e^{\ii \lambda (x-y)}qC_+\ov{q}\psi\dd y.
    \end{align*}

    For $R>0$, taking 
    \begin{align*}
        H(R):=\int_{|y|>R}|qg|\dd y, \qquad g:=C_+\ov{q}\psi, 
    \end{align*}
    it follows that 
    \begin{align}\label{eqn:psiH}
        |\psi(x)|\leq H(|x|). 
    \end{align}

    In the following we denote
    \begin{align*}
        Q(R):=\|q\|_{L^2(|y|>R)}, \quad F(R):=\|\psi\|_{L^2(|y|<R)}. 
    \end{align*}
    We begin by estimating $H(R)$, splitting $g=g_{0,R}+g_{1,R}$ with
    \begin{align*}
        g_{0,R}:=C_+\1_{\{|y|>R/2\}}\ov{q}\psi,\qquad g_{1,R}:=C_+\1_{\{y<R/2\}}\ov{q}\psi.
    \end{align*}
    For the first part, using the $L^2$ boundedness of $C_+$, we have
    \begin{align*}
       \|q g_{0,R}\|_{L^1(|y|>R)}\leq Q(R)\|q \psi\|_{L^2(|y|>R/2)}&\leq Q(R)Q(R/2)\|\psi\|_{L^\infty(|y|>R/2)}\\&\leq Q(R/2)^2 H(R/2),
    \end{align*}
    where we have used \eqref{eqn:psiH} and the monotonicity of $Q$ and $H$ in the last step.

    For the second part we use the property of $C_+$ that for $f\in L^2(
    \R)$ and a.e.~$x\in\R$ we have
    \begin{align*}
        (C_+f)(x)= \lim_{\varepsilon \rightarrow 0_+}\frac{1}{2\pi \ii}\int_{\R}\frac{f(y)}{y-x-\ii \varepsilon}\dd y.
    \end{align*}
    Thus, for $|x|>R$, since $|y|<R/2<|x|/2$ and so $|y-x|>|x|/2$, we get 
    \begin{align*}
        g_{1,R}(x)&=\frac{1}{2\pi \ii}\int_{|y|<R/2}\frac{q(y)\psi(y)}{y-x}\dd y
    \end{align*}
    and 
    \begin{align*}
        |g_{1,R}(x)|&\leq \frac{1}{\pi|x|}\int_{|y|<R/2}|q(y)\psi(y)|\dd y\leq \frac{\|q\|_{L^2(\R)}}{\pi|x|}F(R/2). 
    \end{align*}
    Introducing 
    \begin{align*}
        A(R):=\int_{|y|>R}\frac{|q(y)|}{|y|}\dd y,
    \end{align*}
    we deduce that 
    \begin{align*}
        \|qg_{1,R}\|_{L^1(|y|>R)}\leq C F(R/2)A(R)
    \end{align*}
    with $C:=\pi^{-1}\|q\|_{L^2(\R)}$.

    Putting the estimates on $qg_{0,R}$ and $qg_{1,R}$ together, we have 
    \begin{align}\label{eqn:Hbound}
        H(R) & \leq Q(R) Q(R/2) H(R/2) + C F(R/2) A(R) \notag \\
        & \leq Q(R/2)^2 H(R/2)+CF(R/2)A(R). 
    \end{align}
    To show that $\psi\in L^2(\R)$ we will use that, in view of \eqref{eqn:psiH}, for any $R_0>0$, 
    \begin{align}\label{eqn:psiHL2}
        \int_\R|\psi(x)|^2\dd x\leq 2R_0\|\psi\|_{L^\infty(\R)}^2+2\int_{R_0}^\infty H(R)^2\dd R,
    \end{align}
    so that the statement follows if we can show the integral on the right is finite. 

    For any $0<R_0<L$, it follows from \eqref{eqn:Hbound} that 
    \begin{align}\label{eqn:H2Int}
        \int_{R_0}^LH(R)^2\dd R\leq 2\int_{R_0}^L Q(R/2)^4 H(R/2)^2\dd R+2C^2\int_{R_0}^LF(R/2)^2A(R)^2\dd R.
    \end{align}
    Since $Q(R)\rightarrow 0$ as $R\rightarrow \infty$, we can choose $R_0$ such that for any $R\geq R_0$ we have 
    \begin{align*}
        2\int_{R_0}^LQ(R/2)^4 H(R/2)^2\dd R \leq C_{R_0} +\frac{1}{2}\int_{R_0}^LH(R)^2\dd R. 
    \end{align*}
    where
    $$
    C_{R_0} := 2\int_{R_0}^{2R_0} Q(R/2)^4 H(R/2)^2\dd R \,.
    $$
    Note that $C_{R_0}<\infty$ since $Q$ and $H$ are bounded.

    As a consequence, \eqref{eqn:H2Int} implies
    \begin{align*}
        \int_{R_0}^LH(R)^2\dd R \leq 2 C_{R_0} + 4 C^2 \int_{R_0}^LF(R/2)^2A(R)^2\dd R.
    \end{align*}
    
    Using \eqref{eqn:psiH}, and  
    \begin{align*}
        F(R/2)^2&
        \leq F(R_0)^2+ 2 \int_{R_0}^R H(T)^2\dd T,
    \end{align*}
    we get 
    \begin{align*}
        \int_{R_0}^LH(R)^2\dd R & \leq 2C_{R_0} + 4C^2 F(R_0)^2 \int_{R_0}^L A(R)^2\dd R \\
        & \quad + 8 C^2 \int_{R_0}^L A(R)^2 \left(\int_{R_0}^R H(T)^2\dd T\right)\dd R.
    \end{align*}
    Applying Gr\"onwall's inequality gives 
    \begin{align*}
        \int_{R_0}^L H(R)^2 \dd R\leq 
        \left( 2C_{R_0} + 4C^2 F(R_0)^2 \int_{R_0}^L A(R)^2\dd R \right)
        e^{8 C^2 \int_{R_0}^L A(R)^2\dd R}.
    \end{align*}
    If $A\in L^2(R_0,\infty)$ then taking $L\rightarrow\infty$ in the above gives $H\in L^2(R_0,\infty)$ and hence $\psi\in L^2(\R)$ by \eqref{eqn:psiHL2}. 

    Since 
    \begin{align*}
        \|A\|_{L^2(R_0,\infty)}&=\left(\int_{R_0}^\infty\left(\int_{|x|>R}\frac{|q(x)|}{|x|}\dd x\right)^2\dd R \right)^{1/2}\\
        &=\left(\int_{R_0}^\infty\left(\int_{|y|>1}|q(R y)|\frac{\dd y}{|y|}\right)^2\dd R \right)^{1/2}\\
        &= \int_{|y|>1}\left(\int_{R_0}^\infty|q(R y)|^2\dd R\right)^{1/2}\frac{\dd y}{|y|}
    \end{align*}
    where we have used Minkowski in the last step. Rescaling in the integrand gives 
    \begin{align*}
        \|A\|_{L^2(R_0,\infty)}&=\int_{1}^\infty\left(\int_{R_0}^\infty|q(R y)|^2\dd R\right)^{1/2}\frac{\dd y}{|y|}+\int_{1}^\infty\left(\int_{R_0}^\infty|q(-R y)|^2\dd R\right)^{1/2}\frac{\dd y}{|y|}\\
        &\leq \int_{1}^\infty\left(\int_{R_0}^\infty|q(R)|^2\frac{\dd R}{|y|}\right)^{1/2}\frac{\dd y}{|y|}+\int_{1}^\infty\left(\int_{R_0}^\infty|q(-R)|^2\frac{\dd R}{|y|} \right)^{1/2}\frac{\dd y}{|y|}\\
        &\leq \|q\|_{L^2(\R)} \int_{1}^\infty\frac{\dd y}{y^{3/2}}\leq 2\|q\|_{L^2(\R)}.
    \end{align*}
    This concludes the proof of $\psi\in L^2(\R)$.

    \medskip

    \emph{Step 2.}
    To show the asymptotic decay we note that $F(R/2)\leq \|\psi\|_{L^2(\R)}$, hence from \eqref{eqn:Hbound} we have
    \begin{align*}
        H(R)\leq Q(R)Q(R/2)H(R/2)+C \|\psi\|_{L^2(\R)} A(R) \,.
    \end{align*}
    As
    \begin{align*}
        A(R)=\int_{|y|>R}\frac{|q(y)|}{|y|}\dd y\leq (R/2)^{-1/2}Q(R),
    \end{align*}
    we obtain
    $$
    Q(R)^{-1} R^{1/2} H(R) \leq \left( 2^{1/2} Q(R/2)^2 \right) Q(R/2)^{-1} (R/2)^{1/2} H(R/2) + C'
    $$
    with $C' := 2^{1/2} C \|\psi\|_{L^2(\R)}$.

    We can choose $R_0'$ such that for all $R\geq R_0'$ we have $2^{1/2} Q(R/2)^2\leq1/2$. Thus, for all $R\geq 2R_0'$.
    $$
    Q(R)^{-1} R^{1/2} H(R) \leq \frac12 Q(R/2)^{-1} (R/2)^{1/2} H(R/2) + C' \,.
    $$
    Taking $R_n := 2^n R_0'$ and $a_n:=Q(R_n)^{-1}R_n^{1/2}H(R_n)$ for $n\geq 0$ we get
    $$
    a_{n+1} \leq \frac12 a_n + C' 
    \qquad\text{for all}\ n\geq 0 \,.
    $$
    It follows that, for any $N\geq 1$,
    $$
    a_N \leq C'\sum_{n=0}^{N-1}2^{-n}+2^{-N}a_0\lesssim 1. 
    $$
    Thus, we have shown that
    $$
    H(R_N) \lesssim Q(R_N) R_N^{-1/2} \,.
    $$
    By \eqref{eqn:psiH} and monotonicity of $H$, we conclude that for $|x|>R_0$
    $$
    |\psi(x)| \leq H(|x|) \lesssim |x|^{-1/2} \|q\|_{L^2(|y|>|x|/2)},
    $$
    This completes the proof of the theorem.
\end{proof}

The following lemma will be useful in Section \ref{sec:expansionevs}. It improves the decay of eigenfunctions under the additional assumption that $q$ is integrable.

\begin{lemma}\label{efl1}
    Let $q\in L^2_+(\R)\cap L^1(\R)$ and let $\psi$ be as in Theorem \ref{ef}. Then $\psi\in L^1(\R)$ and 
    \begin{align*}
        \psi(x)=o(|x|^{-1})\qquad \text{as}\quad |x|\rightarrow\infty. 
    \end{align*}
\end{lemma}

\begin{proof}
    From Theorem \ref{ef} we have $\psi\in L^2(\R)$. From the proof we have 
    \begin{align*}
        \int_{|x|<L} |\psi(x)|\dd x\leq 2R_0\|\psi\|_{L^\infty(\R)}+2\int_{R_0}^L H(R)\dd R
    \end{align*}
    and 
    \begin{align*}
        \int_{R_0}^L H(R)\dd R\leq \int_{R_0}^L Q(R/2)^2H(R/2)\dd R+C \int_{R_0}^L F(R/2) A(R)\dd R. 
    \end{align*}
    Choosing $R_0$ large enough, we can absorb the first term on the left to get
    \begin{align*}
        \int_{R_0}^L H(R)\dd R\leq 2 C_{R_0}' + 2C \int_{R_0}^L F(R/2) A(R)\dd R
    \end{align*}
    with $C_{R_0} := \int_{R_0}^{2R_0} Q(R/2)^2H(R/2)\dd R$. Note that $F(R/2)\leq\|\psi\|_{L^2(\R)}$. Thus, by taking $L\rightarrow \infty$, $\psi\in L^1(\R)$ will follow if $A\in L^1(R_0,\infty)$.

    Using Fubini--Tonelli, it follows that  
    \begin{align*}
        \|A\|_{L^1(R_0,\infty)}&=\int_{R_0}^\infty \int_{|x|>R}\frac{|q(x)|}{|x|}\dd x\dd R\\
        &\leq \int_{|x|>R_0}\frac{|q(x)|}{|x|}\left(\int_{R_0}^{|x|} \dd R\right)\dd x\leq \int_{\R} |q(x)|\dd x.
    \end{align*}
    This shows that $\psi\in L^1(\R)$. 
    
    The asymptotic decay follows from an analogous argument as in Step 2 of the proof of Theorem \ref{ef}, using 
    \begin{equation*}
        A(R)\leq R^{-1}\|q\|_{L^1(|x|>R)}.
    \end{equation*}
    Indeed, we can find $R_0'$ such that for $|x|\geq R_0'$, 
    \begin{equation*}
        |\psi(x)|\lesssim |x|^{-1}\max\{\|q\|_{L^1(|y|>|x|/2)},\|q\|_{L^2(|y|>|x|/2)}\}.\qedhere
    \end{equation*}
\end{proof}

As a brief aside, in the next lemma we provide an alternative proof of the following result, due to G\'erard and Lenzmann \cite{gerard_calogero--moser_2023}. The statement will be important in the proof of the trace identities in Section~\ref{sec:traceform}. 

\begin{lemma}\label{lem:qeigenfuncinprod}
	Let $q\in L^2_+(\R)$ and let $\psi$ be an eigenfunction of $L_q$. Then
	\begin{equation}\label{eqn:qeigenfuncinprod}
        \abs{\int_{\R} \overline{q(x)}\psi(x)\dd x}^2=2\pi \int_\R\abs{\psi(x)}^2\dd x. 
    \end{equation}
\end{lemma}

\begin{proof}
	We multiply the equation
	$$
	-\ii\psi' - q C_+(\overline q\psi) = \lambda \psi
	$$
	by $x\overline\psi$ and take the imaginary part to find
	$$
	- \frac12 x (|\psi|^2)' - x\im(\overline\psi q C_+(\overline q \psi)) = 0 \,.
	$$
	Let $L>0$ be a parameter that will tend to infinity at the end. Integrating the previous identity over $(-L,L)$ yields
	$$
	- \frac12 x |\psi|^2 \Big|_{-L}^L  + \frac12 \int_{-L}^L |\psi|^2\dd x - \int_{-L}^L x \im(\overline\psi q C_+(\overline q \psi))\dd x = 0 \,.
	$$
	As $L\to\infty$, the second term on the left side tends to $\frac12\|\psi\|_2^2$ and, by \eqref{eq:psidecay}, the first term vanishes. We write the third term as
	\begin{align*}
	\int_{-L}^L x \im(\overline\psi q C_+(\overline q \psi))\dd x
	& = \int_{-L}^L x \im(\overline\psi q C_+(\1_{\{|y|< L\}} \overline q \psi))\dd x \\
	& \quad + \int_{-L}^L x \im(\overline\psi q C_+(\1_{\{|y|\geq L\}} \overline q \psi))\dd x \,.    
	\end{align*}
	We note that
	\begin{align*}
		\int_{-L}^L x \im(\overline\psi q C_+(\1_{\{|y|\leq L\}} \overline q \psi))\dd x
		& = \im \frac1{2\pi\ii} \iint_{(-L,L)\times(-L,L)} \frac{x}{y-x} \overline{\psi(x)} q(x) \overline{q(y)} \psi(y)\dd x\dd y \\
		& = \frac1{4\pi} \left| \int_{-L}^L \overline{q(x)} \psi(x)\dd x \right|^2.
	\end{align*}
	As $L\to\infty$, this converges to
	$$
	\frac1{4\pi} \left| \int_\R \overline{q(x)} \psi(x)\dd x \right|^2.
	$$
	This shows that the identity in the lemma is equivalent to
	$$
	\lim_{L\to\infty} \int_{-L}^L x \im(\overline\psi q C_+(\1_{\{|y|\geq L\}} \overline q \psi))\dd x = 0 \,.
	$$
	We decompose
	$$
	\int_{-L}^L x \im(\overline\psi q C_+(\1_{\{|y|\geq L\}} \overline q \psi))\,dx = I_1 + I_2
	$$
	with
	\begin{align*}
		I_1 & := \int_{L/2<|x|<L} x \im(\overline\psi q C_+(\1_{\{|y|\geq L\}} \overline q \psi))\dd x \,,\\
		I_2 & := \int_{|x|<L/2} x \im(\overline\psi q C_+(\1_{\{|y| \geq L\}} \overline q \psi))\dd x \,,
	\end{align*}
	and show that both terms tend to zero as $L\to\infty$.
	
	To bound $I_1$, we use the $L^2(\R)$ boundedness of $C_+$ to see that
	\begin{align*}
		|I_1| & \leq \|\1_{\{L/2<|x|<L\}} x \psi q\|_{L^2(\R)}\|\1_{\{|y| \geq L\}}\ov{q}\psi\|_{L^2(\R)}\\
        &\leq L \|\psi\|_{L^\infty(|x|>L/2)}^2\|q\|_{L^2(\R)}^2 \,.
	\end{align*}
	The right side tends to zero from \eqref{eq:psidecay}.
	
	For $I_2$ we use that for $|x|<L/2$ and $|y|\geq L$ we have $|y-x|\geq L/2$, so that $|x/(y-x)|\leq 1$ and 
    \begin{align*}
        |I_2| &\leq \frac1{2\pi} \int_{|y|>L}\int_{|x|<L/2}  |\overline{\psi(x)} q(x)| |\overline{q(y)} \psi(y)|\dd x\dd y\\
        &= \frac{1}{2\pi}\int_{|x|<L/2}|\psi(x)q(x)|\dd x\int_{|y|>L}|\psi(y)q(y)|\dd y. 
    \end{align*}
    Since $q\psi\in L^1(\R)$, one factor is bounded while the other goes to zero. 
\end{proof}

As noted in \cite{gerard_calogero--moser_2023}, as a consequence of the identity \eqref{eqn:qeigenfuncinprod} we learn that $\CMop$ has only finitely many eigenvalues, namely
\begin{align}\label{eqn:Nfinbound}
    N(\CMop)=\sum_{j=1}^N \norm{\phi_j}_2^2=\frac{1}{2\pi}\sum_{j=1}^N\abs{\langle \phi_j,q\rangle}^2\leq \frac{1}{2\pi}\int_{\R}\abs{q}^2\dd x,
\end{align}
where $\phi_j$ are the normalised eigenfunctions of $\CMop$, and Bessel's inequality is used in the last step.

Moreover, as noted in \cite{gerard_calogero--moser_2023}, the identity \eqref{eqn:qeigenfuncinprod} implies that all eigenvalues, including any embedded ones, must be simple. Indeed, suppose that $\phi$ and $\psi$ are orthogonal eigenfunctions corresponding to the same eigenvalue. Then, by \eqref{eqn:qeigenfuncinprod} we can choose nonzero constants $c, C\in \C$ such that 
\begin{align*}
    \sclp{q,c\phi}=\sclp{q,C\psi}\neq 0.
\end{align*}
Define $\eta\coloneqq c\phi-C\psi$, which is again an eigenfunction corresponding to the same eigenvalue. But then
\begin{align*}
    \sclp{q,\eta}=0,
\end{align*}
contradicting the identity in \eqref{eqn:qeigenfuncinprod}. Hence all eigenvalues of $\CMop$ are simple.


\section{The limiting absorption principle}\label{sec:lap}

We continue to assume throughout this section that
$$
q\in L^2_+(\R) \,.
$$
The aim of this section is to study the resolvent of $L_q$ and, in particular, to prove a limiting absorption principle concerning the boundary values of the resolvent on the spectrum.

The perturbed and unperturbed resolvents,
$$
R(k):=(L_q-k)^{-1}
\qquad\text{and}\qquad
R_0(k):=(L_0-k)^{-1} \,,
$$
are initially defined for $k$ in the resolvent sets $\rho(L_q)$ and $\rho(L_0)$, respectively. To motivate the arguments that follow we recall the resolvent identity in the form
\begin{equation}
	\label{eq:resolventformula}
	R(k) = R_0(k) + R_0(k) q C_+ (1- C_+ \overline{q} R_0(k) q C_+)^{-1} C_+ \overline q R_0(k) \,.
\end{equation}
The precise statement is that for $k\in\C\setminus[0,\infty)=\rho(L_0)$, we have $k\in\rho(L_q)$ if and only if the operator $1- C_+ \overline{q} R_0(k) q C_+$ is invertible, and in this case the above formula holds; see \cite[Section 0.3.]{yafaev_mathematical_2010}. 

It will be relatively easy to extend $R_0(k)$ to the cut and then we will use \eqref{eq:resolventformula} to also extend $R(k)$ to the cut, at least away from a `small' exceptional set. By the cut we mean the positive half-axis and, as usual, the extensions from above and from below are not necessarily the same.

As a point of notation, we recall that $\R_+=(0,\infty)$ and define
\[
\Chat\coloneqq\bigl(\C\setminus\R_+\bigr)\cup(\R_++0\ii)\cup(\R_+ -0\ii)\,.
\]
We equip $\Chat$ with the coarsest topology such that the inclusion $\C\setminus\R_+\hookrightarrow \Chat$ is continuous, and that for each $\lambda>0$ the boundary points $\lambda\pm 0\ii$ admit neighbourhoods consisting of all $k\in \Chat$ with $\abs{k-\lambda}<\varepsilon$ and $\pm \im k\geq 0$, for sufficiently small $\varepsilon>0$. At the origin, neighbourhoods of $0\in \Chat$ consist of all $z\in \Chat$ with $\abs{z}<\varepsilon$. For $\lambda\leq 0$, the notation $\lambda\pm0\ii$ is understood to  simply mean $\lambda$. Moreover, for a set $E\subset\R$ let
$$
E^\diamond := (E+0\ii) \cup (E-0\ii)
$$
denote its embedding into $\Chat$. 


\subsection*{The free resolvent}

For $k\in\C\setminus\R$ the operator $R_0(k)$ acts as a convolution operator,
\begin{equation}
    \label{eq:freeresolvent}
    (R_0(k)f)(x) = \int_\R G_k(x-y) f(y)\dd y
    \qquad\text{for}\ x\in\R \ \text{and}\ f\in L^2_+(\R) \,.
\end{equation}
The convolution kernel $G_k$ is defined more generally for all $k\in \Chat$ as follows. For $k\in \Chat\backslash(-\infty,0]$ we have
\begin{align}\label{eqn:Gk}
    G_{k}(x)\coloneqq\pm\ii \e^{\ii k x }\1_{\R_\pm}(x), \quad k=\lambda \pm\ii\mu \quad\text{with}\quad\lambda\in \R, \ \mu\geq 0.
\end{align}
When $k \in (-\infty,0]$ we use the convention that $G_k$ can denote either one of the two choices $k=\lambda\pm0\ii$. While this may seem ambiguous, it will not create any confusion for us, since we will only convolve $G_k$ with functions $f\in L^1(\R)$ whose Fourier transform is supported on $[0,\infty)$ and for such functions we have
\begin{align}\label{eqn:pmambig}
    \ii \int_{-\infty}^x \e^{\ii\lambda (x-y)}f(y)\dd y=-\ii \int_{x}^\infty \e^{\ii\lambda (x-y)}f(y)\dd y \quad \text{ for all}\quad x\in\R. 
\end{align}

Equation \eqref{eq:freeresolvent}, that is, the fact that $R_0(k)$ acts as convolution with $G_k$ can be shown either applying a Fourier transform and noting that
\begin{align*}
    \frac{1}{2\pi} \int_\R \frac{\e^{\ii\xi(x-y)}}{\xi - k} \dd\xi = G_k(x-y)
\end{align*}
or by solving the differential equation $-\ii\psi' = k\psi +f$.

We remark that in the Benjamin--Ono case in \cite{wu_jost_2017} a different form of the free resolvent is used, namely convolution with
\begin{align*}
    \widetilde{G}_k(x)\coloneqq\frac{1}{2\pi}\int_0^\infty\frac{\e^{\ii \xi x}}{\xi-k}\dd \xi.
\end{align*}
We find our choice easier to work with since, in particular, the kernel $G_k$ is bounded.

In the following lemma we summarise properties of the free resolvent. With $\widehat f$ denoting the Fourier transform of $f$ (see Lemma \ref{lem:emederiv} for the precise normalisation, which is irrelevant here) we shall use the notation
$$
L^1_+(\R):= \{ f\in L^1(\R) :\  \widehat f(\lambda) = 0 \ \text{for all}\ \lambda \in (-\infty,0]\} \,.
$$

\begin{lemma}\label{lem:qRqcompact}
Let $q \in L^2_+(\R)$. Then:
\begin{itemize}
    \item[(a)] 
    Convolution with $G_k$ defines a bounded operator from $L^1_+(\R)$ to $L^\infty(\R)$ for all $k \in \Chat$ and for $k\in\C\setminus\R$ this operator coincides with $R_0(k)$ on $L^2_+(\R)\cap L^1_+(\R)$. In what follows, $R_0(k)$ will also denote the operator from $L^1_+(\R)$ to $L^\infty(\R)$.
    \item[(b)] For any $f,g\in L^1_+(\R)$, the map 
    \[
        k\mapsto \sclp{R_0(k)f,g}
    \]
    is continuous on $\Chat$.
    \item[(c)] The operator 
    \[
        C_+\overline{q} R_0(k) q \, C_+
    \]
    is Hilbert--Schmidt.
    \item[(d)] The map
    \[
        k \mapsto C_+ \overline{q}R_0(k) q  C_+
    \]
    is analytic from $\C \setminus [0,\infty)$ into the Hilbert--Schmidt class $\mathfrak{S}_2$, 
    and extends continuously to $\Chat$.
\end{itemize}
\end{lemma}

We remark that throughout, the symbol $\sclp{\cdot,\cdot}$ denotes either the $L^2$ inner product or the natural dual pairings $(L^1,L^\infty)$ and $(L^\infty,L^1)$; no distinction will be made between these uses. Also, we use the convention that $\sclp{\cdot,\cdot}$ is linear in the first and anti-linear in the second argument.

\begin{proof}  
        {\emph{Step 1.}} 
        The fact that convolution with $G_k$ defines a bounded operator from $L^1_+$ to $L^\infty$ follows immediately from the fact that $G_k\in L^\infty(\R)$ for all $k\in\Chat$. The fact that for $k\in\C\setminus\R$ this coincides with $R_0(k)$ on $L^2_+\cap L^1_+$ follows from \eqref{eq:freeresolvent}.

        \medskip

        \emph{Step 2.}
        Let us show the following fact: If $k\in\Chat$ and $(k_n)\subset\Chat$ satisfy $k_n\to k$ in $\Chat$, then for any $f\in L^1_+(\R)$ we have $R_0(k_n)f\to R_0(k)f$ pointwise and boundedly. Once this is shown, it follows by dominated convergence that $(R_0(k_n)f,g)\to (R_0(k)f,g)$ for any $g\in L^1_+(\R)$.
	
	We write $k=\lambda+\mu\ii$ and $k_n = \lambda_n+\mu_n\ii$. We may assume that $\mu=\im k\geq 0$, the opposite case being similar. When $\lambda>0$, we know that $\mu_n\geq 0$ for all sufficiently large $n$. When $\lambda\leq 0$ we make the additional assumption that $\mu_n\geq 0$ for all sufficiently large $n$. This does not represent any loss of generality, because we can handle the case where $\mu_n<0$ with the case where $\mu\leq 0$ using \eqref{eqn:pmambig}. 
	
	As a consequence, we have
	$$
	(R_0(k_n)f - R_0(k)f)(x) = \ii \int_{-\infty}^x \e^{\ii\lambda(x-y)} \e^{-\mu(x-y)}
	\left( \e^{\ii(\lambda_n-\lambda)(x-y)} \e^{-(\mu_n-\mu)(x-y)} - 1 \right) f(y)\dd y \,.
	$$
	When $\mu>0$, we bound
	$$
	\left| \e^{\ii(\lambda_n-\lambda)(x-y)} \e^{-(\mu_n-\mu)(x-y)} - 1 \right| \leq 1+ \e^{|\mu_n-\mu|(x-y)} \leq 2 \e^{|\mu_n-\mu|(x-y)} \,.
	$$
	In particular, when $n$ is so large that $|\mu_n-\mu|\leq\mu$, then
	$$
	| 	(R_0(k_n)f - R_0(k)f)(x) | \leq 2 \|f\|_1 \,.
	$$
	This bound and pointwise convergence of the integrand implies that $(R_0(k_n)f - R_0(k)f)(x)\to 0$ for each $x\in\R$.
	
	Now let $\mu=0$. Using $\mu_n\geq 0$ we can bound
	$$
	\left| \e^{\ii(\lambda_n-\lambda)(x-y)} \e^{-\mu_n(x-y)} - 1 \right| \leq 1+ \e^{-\mu_n(x-y)} \leq 2 \,.
	$$
	Using this bound we can argue as before.
        
        \medskip
        
        \emph{Step 3.} Let us note that the operator $C_+\ov{q}R_0(k)qC_+$ is bounded on $L^2_+$. Indeed, $C_+$ is a bounded operator on $L^2$, multiplication by $q\in L^2_+$ is bounded from $L^2_+$ to $L^1_+$, by part (a) $R_0(k)$ is bounded from $L^1_+$ to $L^\infty$ and multiplication by $q$ is bounded from $L^\infty$ to $L^2$.

        In particular, we have for all $k\in \Chat$ and $\phi\in L^2_+$        \begin{align*}
            C_+\ov{q}R_0(k)qC_+\phi=C_+\ov{q}G_k\ast q\phi \,.
        \end{align*}
        For the Hilbert--Schmidt norm, we obtain 
        \begin{align*}
            \|C_+\ov{q}R_0(k)qC_+\|_{\mathfrak{S}^2(L^2_+)}^2 & \leq \|\ov{q}G_k\ast (q\,\cdot)\|_{\mathfrak{S}^2(L^2)}^2 \\
            &\leq \int_{-\infty}^\infty\int_{-\infty}^{\infty} \abs{q(x)}^2\abs{G_k(x-y)}^2\abs{q(y)}^2\dd x\dd y \\
            & \leq \norm{G_k}_{L^\infty}^2\norm{q}_2^4 \\
            & \leq \norm{q}_2^4 \,.
        \end{align*}
        Note, in particular, that this Hilbert--Schmidt norm is uniformly bounded in $k$.

\medskip
        
        {\emph{Step 4.}} Let us show continuity of $k\mapsto C_+ \ov{q}R_0(k)q C_+$ at a point $k=\lambda\pm\ii 0$. (Continuity at other points will follow from analyticity established in the next step.) Let $h\in \Chat$ with $\pm\im h\geq 0$, then
        \begin{align*}
             \norm{C_+(\ov{q}R_0(k+h)q-\ov{q}R_0(k)q)C_+}_{\mathfrak{S}^2}^2&\leq \|\ov{q}G_{k+h}\ast (q\,\cdot)-\ov{q}G_{k}\ast(q\,\cdot)\|_{\mathfrak{S}_2(L^2)}^2.
        \end{align*}
        As $h\to 0$, the pointwise convergence of $G_{k+h}$ to $G_k$ together with dominated convergence implies that the right side tends to zero. This proves the asserted continuity at $\lambda>0$. It also proves continuity at $\lambda\leq 0$ by the same argument as in Step 2.

\medskip

    {\emph{Step 5.}} Finally, let us prove the analyticity of $k\mapsto C_+\ov{q}R_0(k)qC_+$ in $\C\setminus\R$, for the sake of concreteness in the upper half-plane. For $\im k>0$ and $h\in\C$ with $\abs{h}$ small we note that for each $x\in \R$ 
    \begin{align*}
        D_hG_k(x)\coloneqq \frac{1}{h}(G_{k+h}(x)-G_k(x))=\ii \e^{\ii kx}\1_{\R_+}(x) \frac1h (\e^{\ii hx}-1)
    \end{align*}
    converges pointwise to $-x \e^{\ii kx}\1_{\R_+}(x)$ as $\abs{h}\rightarrow 0$. Moreover, using Taylor expansion,
    \begin{align*}
         \abs{D_hG_k(x) + x \1_{\R_+}(x)\e^{\ii kx}} & = \1_{\R_+}(x)\e^{ -\im k x}\abs{\frac{\e^{\ii hx}-1}{h}-\ii x} \\
         & \leq \frac{1}{2} \e^{|\im(hx)|} \1_{\R_+}(x)\abs{h} x^2\e^{ -\im k x},
    \end{align*}
    which is uniformly bounded in $x$ (if, say, $|h|\leq \im k/2$). Thus, $D_hG_k$ has a limit in $L^\infty(\R)$ and $C_+\ov{q}R_0(k)qC_+$ is analytic on $k\in \C\backslash\R$. 
    
    Note that since in the previous step we have established continuity of $C_+\ov{q}R_0(k)qC_+$ as $k$ approaches the real axis, we obtain analyticity in all of $\C\backslash[0,\infty)$.
\end{proof}


\subsection*{The exceptional set}

According to Lemma \ref{lem:qRqcompact} the map $k\mapsto C_+ \overline{q} R_0(k) q C_+$ is continuous in Hilbert--Schmidt norm in $\Chat$ and, in particular, the operators $C_+ \overline{q} R_0(\lambda\pm0\ii) q C_+$ are well defined for $\lambda\in[0,\infty)$. Let 
$$
\mathcal N_\pm := \left\{ \lambda\in [0,\infty) :\ \ker(1- C_+ \overline{q} R_0(\lambda\pm0\ii) q C_+) \neq \{0\} \right\}.
$$
It follows from the analytic Fredholm alternative \cite[Theorem 1.8.3]{yafaev_mathematical_1992} that both sets $\mathcal N_\pm$ are closed and have zero measure.

Our next goal is to prove that $\mathcal N_+=\mathcal N_-= \sigma_\mathrm{p}(L_q)\cap [0,\infty)$. To do so, we use the following lemma that describes solutions of the inhomogeneous equation.

For $k\in\Chat$ we let $\mathcal S(k)$ denote the set of all $\psi\in L^\infty(\R)$ that are locally absolutely continuous and satisfy
$$
-\ii\psi' - qC_+\overline q\psi = k\psi \,,
$$
together with the asymptotics
$$
\begin{cases}
	\psi(x) \to 0	\qquad\text{as}\ |x|\to\infty  & \text{if}\ k\in\C\setminus[0,\infty) \,, \\
	\psi(x) \to 0 \qquad\text{as}\ x\to\mp\infty & \text{if}\ k=\lambda\pm0\ii \,.
\end{cases}
$$

\begin{lemma}\label{lem:LaxeigenqRqfixpt}
	Let $k\in\Chat$. There is a bijective correspondence between the sets
	$$
	\ker(1-C_+\overline q R_0(k)q C_+)
		\qquad\text{and}\qquad
	\mathcal S(k) \,.
	$$
	If $k\in\C\setminus\R^\diamond$, both sets are empty, and if $k=\lambda\pm0\ii\in\R^\diamond$, the following holds:
	\begin{itemize}
		\item[(a)] If $g\in\ker(1-C_+\overline q R_0(\lambda\pm0\ii)q C_+)$, then $\psi:=R_0(\lambda\pm0\ii)qC_+g \in\mathcal S(\lambda\pm0\ii)$, $\psi(x)\to 0$ as $|x|\to\infty$ and
		\begin{equation}
		    \label{eqn:fixedptvanish}
            \int_\R \e^{-\ii\lambda y} q(y) (C_+\overline q\psi)(y)\dd y = 0 \,.
		\end{equation}
		\item[(b)] If $\psi\in\mathcal S(\lambda\pm0\ii)$, then $g:=C_+\overline q\psi \in \ker(1-C_+\overline q R_0(\lambda\pm0\ii)q C_+)$.
	\end{itemize}
	Moreover, the maps $g\mapsto\psi$ and $\psi\mapsto g$ are inverse to each other.
\end{lemma}

\begin{proof}
	We only prove the assertion when $\im k\geq 0$, the opposite case being similar.

\medskip
    
	\emph{Step 1.} Let $g\in\ker(1-C_+\overline q R_0(k)q C_+)$. When $k\in\C\setminus\R$ it follows from the precise form of the resolvent identity that $g=0$. Thus, in what follows we assume that $k=\lambda+\ii 0$ with $\lambda\in\R$. Since $qC_+g\in L^1_+$, and since $R_0(k)$ maps $L^1_+$ into $L^\infty$, we have $\psi\in L^\infty$. Moreover, the formula
	$$
	\psi(x) = \ii \int_{-\infty}^x \e^{\ii\lambda(x-y)} q(y) (C_+g)(y)\dd y
	$$
	implies that $\psi$ is locally absolutely continuous and satisfies the claimed equation. (Here we use $g=C_+\overline q\psi$ by the equation for $g$.) Since $q C_+g\in L^1(\R)$, the formula for $\psi$ implies, by dominated convergence,
	$$
	\psi(x) \to 0
	\qquad\text{as}\ x\to-\infty 
	$$
	and
	$$
	\e^{-\ii\lambda x} \psi(x) \to \ii \int_\R \e^{-\ii\lambda y} q(y) (C_+g)(y)\dd y = \ii \int_\R \e^{-\ii\lambda y}q(y) (C_+ \overline q\psi)(y)\dd y \ \text{as}\ x\to+\infty \,. 
	$$
	
	We multiply the equation for $\psi$ by $\overline\psi$ and integrate it over a bounded interval $(a,b)$. Using $2\re\overline\psi\psi' = (|\psi|^2)'$, we find
	$$
	- \frac \ii2 (|\psi(b)|^2 - |\psi(a)|^2) + \im \int_a^b \overline\psi \psi'\dd x - \int_a^b \overline\psi q C_+\overline q\psi\dd x = \lambda \int_a^b |\psi|^2\dd x \,.
	$$
	Taking the imaginary part gives
	$$
	- \frac 12 (|\psi(b)|^2 - |\psi(a)|^2) - \im \int_a^b \overline\psi q C_+\overline q\psi\dd x = 0 \,.
	$$
	We note that $\overline\psi q C_+\overline q\psi$ belongs to $L^1$. Therefore we can take the limit $a\to-\infty$. Since we have already shown that $\psi(x)\to 0$ as $x\to-\infty$, we infer that
	$$
	-\frac12 |\psi(b)|^2 - \im \int_{-\infty}^b  \overline\psi q C_+\overline q\psi\dd x = 0 \,.
	$$
	We now take the limit $b\to\infty$. Since $C_+$ is selfadjoint, the integral on the left side vanishes as $b\to\infty$ and we deduce that $\psi(b)\to 0$ as $b\to\infty$. This proves the assertion about the limit at $+\infty$ and, in view of what we proved before about the limit of $\overline{\e(\lambda)}\psi$, the claimed vanishing of the integral.
	
		
	\medskip
	
	\emph{Step 2.} We now assume that $\psi\in\mathcal S(k)$. 
	
	We first consider the case $\im k>0$ and aim at showing that $\psi=0$. As before, we multiply the equation for $\psi$ by $\overline\psi$ and integrate over $(a,b)$. Taking the imaginary part gives
	$$
	-\frac12 (|\psi(b)|^2 - |\psi(a)|^2) - \im \int_a^b \overline\psi q C_+ \overline q \psi\dd x = \im k \int_a^b |\psi|^2\dd x \,.
	$$
	Using the assumed boundary conditions for $\psi$ together with the fact that $\overline\psi q C_+ \overline q \psi$ belongs to $L^1$, we can let $a\to-\infty$ and $b\to+\infty$ to deduce that
	$$
	0 = \im k \int_\R |\psi|^2\dd x \,.
	$$
	Since $\im k> 0$, we conclude that $\psi=0$, as claimed.
	
	Now let $k=\lambda+0\ii $ with $\lambda\in\R$. Note that $\psi\in L^\infty$ implies $g\in L^2_+(\R)$. Moreover, the equation for $\psi$ implies that
	$$
	(\overline{\e(\lambda)} \psi)'(x) = \overline{\e(\lambda)} (\psi' - \ii \lambda \psi) = \ii \overline{\e(\lambda)} q C_+ \overline q \psi = \ii \overline{\e(\lambda)} q g \,.
	$$
	Since the right side belongs to $L^1$ and since $\e^{-\ii\lambda x} \psi(x)\to 0$ as $x\to-\infty$, we deduce that
	$$
	\e^{-\ii\lambda x} \psi(x) = \ii \int_{-\infty}^x \e^{-\ii\lambda y} q(y) g(y)\dd y \,,
	$$ 
	that is,
	$$
	\psi = R_0(\lambda+0\ii) q g \,.
	$$
	Multiplying this equation by $\overline{q}$ and applying $C_+$, we arrive at $g = C_+\overline qR_0(\lambda+0\ii)qg$, as claimed.

	\medskip
	
	\emph{Step 3.} Finally, we prove that the maps $g\mapsto\psi$ in (a) and $\psi\mapsto g$ in (b) are inverses to each other. If $g\in\ker(1-C_+\overline q R_0(\lambda\pm 0\ii)q C_+)$ and $\psi:=R_0(\lambda\pm 0\ii) qC_+g$, then clearly $C_+ \overline q \psi =
	C_+ \overline q R_0(\lambda\pm 0\ii)q C_+ g = g$, proving one direction.
	
	For the other direction, let $\psi\in S(\lambda\pm 0\ii)$ and $g:=C_+\overline q\psi$. Then $\tilde\psi := R_0(\lambda\pm 0\ii)qC_+ g= R_0(\lambda\pm 0\ii)qC_+\overline q\psi$ satisfies $-\ii\tilde\psi' - qC_+\overline q\psi = \lambda\tilde\psi$ and $\tilde\psi(x)\to 0$ as $x\to\mp\infty$. Thus, $\chi:=\tilde\psi-\psi$ satisfies $-\ii\chi' = \lambda\chi$ and $\chi(x)\to 0$ as $x\to\mp\infty$. This implies $\chi=0$, that is $\tilde\psi=\psi$, as claimed. 
\end{proof}

\begin{corollary}\label{cornpm}
    We have
	$$
	\mathcal N_+ = \mathcal N_- = \sigma_{\rm p}(L_q)\cap\overline{\R_+} \,.
	$$
\end{corollary}

\begin{proof}
	Let $\lambda\in\mathcal N_+$ and $0\neq g\in\ker (1-C_+\overline{q}R_0(\lambda+ 0\ii)qC_+)$. Then, by part (a) of Lemma \ref{lem:LaxeigenqRqfixpt} the function $\psi:=R_0(\lambda+0\ii)qC_+g$ belongs to $\mathcal S(\lambda+0\ii)$ and satisfies $\psi(x)\to 0$ as $|x|\to\infty$. According to Theorem \ref{ef}, this implies $\psi\in L^2(\R)$. Since, again by Lemma \ref{lem:LaxeigenqRqfixpt}, $\psi\neq 0$, we conclude that $\lambda\in \sigma_{\rm p}(L_q)$. The proof for $\lambda\in\mathcal N_-$ is similar.
	
	Conversely, if $\lambda\in\sigma_{\rm p}(L_q)\cap\overline{\R_+}$, then there is a $\psi\neq 0$ satisfying $-i\psi' - qC_+ \overline q\psi = \lambda\psi$. Moreover, from \cite{killip_scaling_2025} we know that $\psi\in\dom L_q = H^1_+(\R)\coloneqq H^1(\R)\cap L^2_+(\R)$. Since all functions in $H^1_+(\R)$ belong to $L^\infty$ and tend to zero at both infinities, we are in the situation of part (b) of Lemma \ref{lem:LaxeigenqRqfixpt} and infer that $C_+\overline q \psi$ belongs to both $\ker (1-C_+\overline{q}R_0(\lambda\pm 0\ii)qC_+)$. Since $\psi\neq 0$, the equation for $\psi$ together with the boundary conditions imply that $C_+\overline q\psi\neq 0$, so $\lambda\in\mathcal N_+\cap\mathcal N_-$.	
\end{proof}


\subsection*{The perturbed resolvent} 

We will show that the perturbed resolvent $R(k)$ exists as a bounded operator from $L^1_+(\R)$ to $L^\infty(\R)$ for all $k\in\Chat\setminus(\sigma_{\mathrm p}(\CMop))^\diamond$ and has suitable continuity properties.

\begin{lemma}\label{lem:pertresolvcont}
    Let $q \in L^2_+(\R)$. Then:
    \begin{itemize}
        \item[(a)] For $k\in\C\setminus\sigma(\CMop)$ the resolvent $R(k)$ extends to a bounded operator from $L^1_+(\R)$ to $L^\infty(\R)$. Moreover, this operator is well defined for all $k\in\Chat\setminus(\sigma_{\mathrm p}(\CMop))^\diamond$ as a bounded operator from $L^1_+(\R)$ to $L^\infty(\R)$ and the resolvent identity \eqref{eq:resolventformula} holds for such $k$.
        \item[(b)] For any $f,g\in L^1_+(\R)$, the map 
    \[
        k\mapsto \sclp{R(k)f,g}
    \]
    is continuous on $\Chat$.   
    \end{itemize}
\end{lemma}

\begin{proof}
    \emph{Step 1.} For all $k\in\Chat\setminus(\sigma_{\mathrm p}(\CMop))^\diamond$ we \emph{define} $R(k)$ as a bounded operator from $L^1_+$ to $L^\infty$ through the resolvent identity \eqref{eq:resolventformula}. Indeed, the first term, $R_0(k)$, is bounded from $L^1_+$ to $L^\infty$ for all $k\in\Chat$ by Lemma \ref{lem:qRqcompact}. For the second term we note that, by Corollary \ref{cornpm}, for $k\in \Chat\setminus(\sigma_{\mathrm p}(\CMop))^\diamond$ the operator $1- C_+ \overline qR_0(k) qC_+$ is boundedly invertible on $L^2_+$. Moreover, $C_+\overline q R_0(k)$ maps $L^1_+$ to $L^2_+$ and, by duality, $R_0(k)qC_+$ maps $L^2_+$ to $L^\infty$. This proves boundedness of $R(k)$ from $L^1_+$ to $L^\infty$ for all $k\in\Chat\setminus(\sigma_{\mathrm p}(\CMop))^\diamond$.

    \medskip

    \emph{Step 2.} To prove the claimed continuity, we write, using \eqref{eq:resolventformula},
	$$
	\sclp{R(k)f,g} = \sclp{R_0(k)f,g} - \langle(1-C_+ \overline q R_0(k)q C_+)^{-1} C_+\overline qR_0(k)f,C_+\overline q R_0(\overline{k})g\rangle \,.
	$$
	According to Lemma \ref{lem:qRqcompact}, the first term on the right side is continuous. The fact that was shown in Step 2 of the proof of Lemma \ref{lem:qRqcompact} also implies, by dominated convergence, that $k\mapsto \overline q R_0(k)f$ is continuous in $\Chat$ with values in $L^2(\R)$. Since $(1-C_+ \overline q R_0(k)q C_+)^{-1}$ is also continuous by Lemma \ref{lem:qRqcompact}, we obtain the claimed continuity of $\sclp{R(k)f,g}$.
\end{proof}

The following lemma shows that the extended resolvent indeed provides the unique solution to a certain equation.

\begin{lemma}\label{lem:pertresolvent}
	Let $q\in L^2_+(\R)$ and $f\in L^1_+(\R)$.
	\begin{itemize}
		\item[(a)] If $k\in\C\setminus\sigma(L_q)$, then there is a unique $\psi\in L^\infty(\R)$ that is locally absolutely continuous and satisfies
		$$
		\begin{cases}
			& -\ii\psi' - q C_+\overline q\psi = k \psi + f \,,\\
			& \psi(x) \to 0 \ \text{as}\ |x|\to\infty \,.
		\end{cases}
		$$
		It is given by $\psi=R(k)f$.
		\item[(b)] If $\lambda\in\overline{\R_+}\setminus\sigma_{\mathrm p}(\CMop)$, then for each choice of sign there is a unique $\psi\in L^\infty(\R)$ that is locally absolutely continuous and satisfies
		$$
		\begin{cases}
			& -\ii\psi' - q C_+\overline q\psi = \lambda \psi + f \,,\\
			& \psi(x) \to 0 \ \text{as}\ x\to\mp\infty \,.
		\end{cases}
		$$
		It is given by $\psi=R(\lambda\pm 0\ii)f$.
	\end{itemize}
\end{lemma}

\begin{proof}
	We already know that $R(k)$ is well defined for $k\in\Chat\setminus\sigma_\mathrm{p}(L_q)^\diamond$. It is easy to see that $R(k)f$ in case (a) and $R(\lambda\pm 0\ii)f$ in case (b) satisfies the claimed equation and boundary conditions. 
	
	It remains to prove uniqueness. Clearly it suffices to show that the corresponding homogeneous equation, together with the corresponding boundary conditions, has only the trivial solution. This is a consequence of Lemma \ref{lem:LaxeigenqRqfixpt}. Indeed, this lemma implies immediately that $\mathcal S(k)=\emptyset$ if $\im k\neq 0$. When $k=\lambda\pm 0\ii$, the lemma implies that $\mathcal S(k)\neq\emptyset$ if and only if $\ker(1-C_+ \overline q R_0(k)qC_+)\neq\emptyset$. When $\lambda<0$ (so that $R_0(\lambda)$ is a bounded operator on $L^2_+(\R)$), it is elementary to see that this happens if and only if $\lambda\in\sigma_\mathrm{p}(L_q)$, which is excluded in part (a). When $\lambda\geq 0$ this happens, by definition, if and only if $\lambda\in\sigma_{\mathrm p}(\CMop)$, which is excluded in part (b). Therefore, under the assumptions of the lemma, there is no nontrivial solution of the homogeneous equation, proving uniqueness of the solution of the inhomogeneous equation.
\end{proof}


\subsection*{The generalised eigenfunctions}

For $\lambda\in[0,\infty)\setminus\sigma_\mathrm{p}(L_q)$ and either choice of sign, we define the \emph{homogeneous Jost solution} by
$$
\boxed{
m_{\rm e}(\lambda\pm 0\ii) := \e(\lambda) + R(\lambda\pm 0\ii) q C_+\overline q\e(\lambda) \,.
}
$$
Note that the definition makes sense as an element of $L^\infty(\R)$ since $qC_+\overline q \e(\lambda)\in L^1_+(\R)$ and $R(\lambda\pm 0\ii)$ maps $L^1(\R)_+$ to $L^\infty(\R)$.

Using the resolvent identity \eqref{eq:resolventformula} we arrive at the integral equation
\begin{equation}\label{eqn:meinteq}
m_{\rm e}(\lambda\pm 0\ii) = \e(\lambda) + R_0(\lambda\pm 0\ii) q C_+\overline q m_\e(\lambda\pm 0\ii) \,,
\end{equation}
which is the analogue of the Lippman--Schwinger equation in scattering theory for Schr\"odinger operators.

Let us show that $m_{\rm e}$ can be characterised as the unique solution of a homogeneous equation with certain asymptotics.

\begin{lemma}\label{lem:uniqueme}
    Let $q\in L^2_+(\R)$ and $\lambda\in [0,\infty) \backslash \sigma_\mathrm{p}(L_q)$. If $\psi\in L^\infty(\R)$ is locally absolutely continuous, solves 
    \begin{align*}
        -\ii\psi'-\pot \psi=\lambda \psi,
    \end{align*}
    and satisfies $\psi(x)-\e^{\ii \lambda x}\rightarrow 0$ as one of $x\rightarrow \mp \infty$, then $\psi$ is given uniquely by $m_\e(\lambda\pm0\ii)$.
\end{lemma}

\begin{proof}
    The function $\psi:=m_{\rm e}(\lambda\pm 0\ii)-\e(\lambda)\in L^\infty$ satisfies the equation and asymptotics in part (b) of Lemma \ref{lem:pertresolvent} with $f=q C_+ \ov{q}\e(\lambda)\in L^1_+$ and is therefore unique.
\end{proof}

\subsection*{The scattering matrix}

For $\lambda\in[0,\infty)\setminus\sigma_\mathrm{p}(L_q)$ we define
\begin{equation}\label{eqn:Gammadef}
\boxed{
\Gamma(\lambda):= 1 + \ii \int_\R \e^{-\ii\lambda y} q(y) C_+ \overline q m_{\rm e}(\lambda+0\ii)\dd y \,.
}
\end{equation}

\begin{lemma}\label{lem:scatobjsGamma}
    Let $q\in L^2_{+}(\R)$. Then $\Gamma \in C([0,\infty)\setminus\sigma_\mathrm{p}(L_q))$ and for all $\lambda\in [0,\infty)\setminus\sigma_\mathrm{p}(L_q)$,
    \begin{equation}\label{eqn:gammaoneident}
        |\Gamma(\lambda)|=1
    \end{equation}
    and
    \begin{equation}
        m_\e(x,\lambda+0\ii)=\Gamma(\lambda) \, m_\e(x,\lambda-0\ii) \,. \label{eqn:megammarel}
    \end{equation}
\end{lemma}

\begin{proof}
    To see the continuity, we insert the definition of $m_{\rm e}(\lambda+0\ii)$ into the definition of $\Gamma$ and find
	$$
	\Gamma(\lambda) = 1 + \ii \langle C_+ \overline q \e(\lambda), \overline q \e(\lambda)\rangle + \ii \langle R(\lambda+0\ii)q C_+ \overline q \e(\lambda), q C_+ \overline q \e(\lambda) \rangle \,.
	$$
	Since $\lambda\mapsto\overline q \e(\lambda)$ is continuous with values in $L^2$ (by dominated convergence), the second term on the right side is continuous and $\lambda\mapsto q C_+ \overline q \e(\lambda)$ is continuous with values in $L^1$. Since $R(\lambda+0\ii)$ is uniformly bounded on compact subsets of $[0,\infty)\setminus\sigma_\mathrm{p}(L_q)$, we can deduce the continuity of the third term from the continuity statement of Lemma~\ref{lem:pertresolvcont}.
    
    The proof of \eqref{eqn:gammaoneident} uses the following, valid for $\lambda>0$,
    \begin{align}\label{eqn:Ginprodfg}
        \langle G_{\lambda\pm 0\ii}\ast f,g\rangle -\langle f,G_{\lambda\pm 0\ii}\ast g\rangle=\pm \ii \langle f,\e(\lambda)\rangle \overline{\langle g,\e(\lambda)\rangle },
    \end{align}
    which can be seen easily from 
    \begin{align}\label{eqn:Gpmdiff}
        G_{\lambda+0\ii}(x)-G_{\lambda-0\ii}(x)=\ii \, \e(x,\lambda);
    \end{align}
    see also \cite[Lemma 4.3]{wu_jost_2017}.
    
   Using this and the definition of $m_\e$, we compute 
    \begin{align*}
        \ii \abs{\sclp{qC_+\overline{q}m_\e(\lambda+0\ii),\e(\lambda)}}^2=&\sclp{G_{\lambda+0\ii}\ast \pot m_\e(\lambda+0\ii),\pot m_\e(\lambda+0\ii)}\\&-\sclp{\pot m_\e(\lambda+0\ii),G_{\lambda+0\ii}\ast \pot m_\e(\lambda+0\ii)}\\
        =&\sclp{m_\e(\lambda+0\ii)-\e(\lambda),\pot m_\e(\lambda+0\ii)}\\&-\sclp{\pot m_\e(\lambda+0\ii),m_\e(\lambda+0\ii)-\e(\lambda)}\\
        =&2\ii \im \sclp{\pot m_\e(\lambda+0\ii),\e(\lambda)},
    \end{align*}
    where we used that $\pot$ is self-adjoint in the last step. Using that $\Gamma(\lambda)=1+\ii \sclp{qC_+\overline{q}m_\e(\lambda+0\ii),\e(\lambda)}$, we find 
    \begin{align*}
        -2\ii \re(1-\Gamma(\lambda))+\ii \abs{1-\Gamma(\lambda)}^2=0 \,,
    \end{align*}
    from which it follows that $|\Gamma(\lambda)|=1$.

    To show \eqref{eqn:megammarel}, we start from the integral equation for $m_\e(\lambda+0\ii)$ and apply \eqref{eqn:Gpmdiff}, 
    \begin{align*}
        m_\e(\lambda+0\ii)&=\e(\lambda)+G_{\lambda+0\ii}\ast(qC_+\overline{q}m_\e(\lambda+0\ii))\\
        &=\e(\lambda)+(\ii \, \e(\lambda)+G_{\lambda-0\ii})\ast(qC_+\overline{q}m_\e(\lambda+0\ii))\\
        &=\Gamma(\lambda)\e(\lambda)+R_0(\lambda-0\ii)\pot m_\e(\lambda+0\ii) \,.
    \end{align*} 
    By rearranging and solving  for $m_\e(\lambda+0\ii)$, we find the desired identity. 
\end{proof}

We note in passing that
\begin{equation}
    \label{eq:lowenergy1}
    m_\e(0\pm 0\ii) = 1
\qquad\text{and}\qquad
\Gamma(0) = 1
\qquad\text{if}\ 0\not\in\sigma_\mathrm{p}(L_q) \,.
\end{equation}
Indeed, this follows from the definition of $m_\e$ and $\Gamma$, using $C_+\ov{q}=0$.


\section{The distorted Fourier transform}\label{sec:distortedFT}

In this section we will show that the homogeneous Jost solutions $m_\e(\lambda-0\ii)$ constitute a complete set of generalised eigenfunctions of the absolutely continuous spectrum of $\CMop$. More precisely, in terms of these functions we will define a surjective partial isometry with the absolutely continuous spectral subspace as its initial space that diagonalises the operator $\CMop$. Thus, the functions $m_\e(\lambda-0\ii)$ play a similar role for the operator $\CMop$ as the exponentials $\e^{\ii\lambda x}$ play for the operator $L_0$, and the transform that we are going to define is the analogue of the Fourier transform.

We begin by stating a by-product of the proof of the diagonalisation.

\begin{theorem}\label{thm:nosc}
    Let $q\in L^2_{+}(\R)$. Then the spectrum of $L_q$ is purely absolutely continuous on $\R_+\setminus\sigma_\mathrm{p}(L_q)$.
\end{theorem}

Since the set $\sigma_\mathrm{p}(L_q)$ is at most countable (indeed, as discussed after Lemma \ref{lem:qeigenfuncinprod}, it is finite) and since a nontrivial singular continuous measure cannot be supported on a countable set, we obtain the following consequence:

\begin{corollary}\label{cor:nosc}
    Let $q\in L^2_{+}(\R)$. Then the singular continuous spectrum of $\CMop$ is empty.
\end{corollary}

We now turn to the details of the diagonalisation. For every $f\in L^1_+(\R)$ and every $\lambda\in [0,\infty)\setminus\sigma_\mathrm{p}(L_q)$, we define
\begin{align}\label{eqn:Phif}
    (\Phi f)(\lambda)\coloneqq \frac{1}{\sqrt{2\pi}} \int_{\R} f(y)\ov{m_\e(y,\lambda-0\ii)}\dd y.
\end{align} 
This is well defined since $m_\e(\lambda-0\ii)\in L^\infty(\R)$ for every $\lambda\in[0,\infty)\setminus\sigma_\mathrm{p}(L_q)$.

The following is the main result of this section. We let $P_{\mathrm{ac}}(\CMop)$ denote the projection onto the absolutely continuous subspace of $L_q$.

\begin{theorem}\label{thm:meFT}
   	Let $q\in L^2_{+}(\R)$. Then the map $\Phi$ extends to a bounded linear operator $\Phi\colon L^2_+(\R)\rightarrow L^2(\R_+)$, with
    \begin{align}\label{eq:unitarity}
   		\Phi^\ast \Phi = P_{\mathrm{ac}}(\CMop), \qquad \Phi \Phi^\ast =1. 
   	\end{align}
    Moreover,
    \begin{equation}
        \label{eq:diag1}
            \dom L_q = \left\{ f\in L^2_+(\R) :\ \int_0^\infty \lambda^2 |(\Phi f)(\lambda)|^2 \,\dd\lambda < \infty \right\},
    \end{equation}
    and for $f$ from this space
    \begin{align}
        \label{eq:diag2}
   		(\Phi \CMop f)(\lambda)=\lambda \, (\Phi f)(\lambda)
        \qquad\text{for a.e.}\ \lambda\in\R_+ \,.
   	\end{align}
\end{theorem}

As a consequence, for any $f\in L^2_+(\R)$, we have the decomposition
\begin{align}\label{eqn:fspectraldecomp}
    f(x)= \sum_{j=1}^N\frac{\sclp{f,\phi_j}}{\norm{\phi_j}_2^2}\phi_j(x)+\frac{1}{\sqrt{2\pi}}\int_0^\infty (\Phi f)(\lambda)m_\e(x,\lambda-0\ii)\dd \lambda,
\end{align}
where $\{\phi_j\}_{j=1}^N$ are the eigenfunctions of $\CMop$. This formula is understood in the $L^2_+(\R)$-sense.

The following result is the main tool in the proof of Theorems \ref{thm:nosc} and \ref{thm:meFT}. We let $\1_\Lambda$ denote the characteristic function of a Borel set $\Lambda$, so that, by the functional calculus, the operator $\1_{\Lambda}(L_q)$ is the spectral projection corresponding to this set.

\begin{lemma}\label{lem:projection}
    Let $q\in L^2_{+}(\R)$ and let $\Lambda\subset[0,\infty)$ be a bounded interval with $\overline\Lambda\cap\sigma_\mathrm{p}(L_q)=\emptyset$. Then, for all $f,g\in L^2_+(\R)\cap L^1_+(\R)$,
    $$
    \sclp{\1_{\Lambda}(L_q) f, g} = \int_\Lambda (\Phi f)(\lambda) \overline{(\Phi g)(\lambda)}\,\dd\lambda \,.
    $$
\end{lemma}

\begin{proof}
    \emph{Step 1.} Let $\lambda\in[0,\infty)\setminus\sigma_\mathrm{p}(L_q)$ and $f\in L^1_+(\R)$. We will show that
    \begin{align*}
        R(\lambda+0\ii)f-R(\lambda-0\ii)f=\ii\sqrt{2\pi} \, (\Phi f)(\lambda) \, m_\e(\lambda-0\ii).
    \end{align*}

    Indeed, the functions $R(\lambda\pm0\ii) f$ satisfy the integral equations
    $$
    R(\lambda\pm0\ii) f = G_{\lambda\pm0\ii}* f + G_{\lambda\pm0\ii}* (qC_+ \overline{q} R(\lambda\pm0\ii) f) \,.
    $$
    Subtracting these two equations and using \eqref{eqn:Gpmdiff}, we find
    \begin{align*}
        R(\lambda+0\ii) f - R(\lambda-0\ii) f
        & = (G_{\lambda+0\ii} - G_{\lambda-0\ii})* f
        + (G_{\lambda+0\ii} - G_{\lambda-0\ii})* (qC_+ \overline{q} R(\lambda+0\ii) f) \\
        & \quad + G_{\lambda-0\ii}*(qC_+ \overline{q} (R(\lambda+0\ii)f - R(\lambda-0\ii) f)) \\
        & = c(\lambda) \e(\lambda) + G_{\lambda-0\ii}*(qC_+ \overline{q} (R(\lambda+0\ii)f - R(\lambda-0\ii) f))
    \end{align*}
    with
    $$
    c(\lambda) = \ii \int_\R \left( f(x) + qC_+\overline q R(\lambda+0\ii) f \right) \e^{-\ii\lambda x}\,\dd x \,.
    $$
    By the definition of $m_\e(\lambda\pm\ii 0)$ in \eqref{eqn:meinteq}, it follows that
    $$
    R(\lambda+0\ii) f - R(\lambda-0\ii) f = c(\lambda)\, m_\e(\lambda-0\ii) \,.
    $$
    This will give the claimed identity, provided we can prove that
    $$
    c(\lambda) = \ii\sqrt{2\pi} \, (\Phi f)(\lambda) \,.
    $$

    Observe that, using the equations for $R(\lambda\pm 0\ii)f$ and $m_\e(\lambda-0\ii)$,
    \begin{align*}
        \ii \sclp{\pot R(\lambda+0\ii)f,\e(\lambda)} \, &\ov{\sclp{\pot m_\e(\lambda+0\ii),\e(\lambda)}}\\=&\sclp{G_{\lambda+0\ii}\ast \pot R(\lambda+0\ii)f,\pot m_\e(\lambda+0\ii)}\\&-\sclp{\pot R(\lambda+0\ii)f,G_{\lambda+0\ii}\ast \pot m_\e(\lambda+0\ii)}\\
        =&\sclp{R(\lambda+0\ii)f-G_{\lambda+0\ii}\ast f,\pot m_\e(\lambda+0\ii)}\\&-\sclp{\pot R(\lambda+0\ii)f,m_\e(\lambda+0\ii)-\e(\lambda)}\\
        =&-\sclp{G_{\lambda+0\ii}\ast f,\pot m_\e(\lambda+0\ii)}+\sclp{\pot R(\lambda+0\ii)f,\e(\lambda)},
    \end{align*}
    where we used the self-adjointness of $\pot$ in the last step. Rearranging gives
    \begin{align}\label{eqn:sctreleqn1}
        \langle G_{\lambda+0\ii}\ast f, q C_+ \overline{q} m_\e(\lambda+0\ii)\rangle- \ov{\Gamma(\lambda)}\langle q C_+ \overline{q}R(\lambda+0\ii)f,\e(\lambda)\rangle =0.
    \end{align}
    Meanwhile, using \eqref{eqn:Ginprodfg} again and the equation for $m_\e$, 
    \begin{align*}
         \ii \langle f, \e(\lambda)\rangle \,& \overline{\langle qC_+\overline{q}m_\e(\lambda+0\ii),\e(\lambda) \rangle} \\
        = & \langle G_{\lambda+0\ii}\ast f, q C_+ \overline{q} m_\e(\lambda+0\ii)\rangle 
        - \langle f, G_{\lambda+0\ii}\ast qC_+\overline{q}m_\e(\lambda+0\ii) \\
        = & \langle G_{\lambda+0\ii}\ast f, q C_+ \overline{q} m_\e(\lambda+0\ii)\rangle - \langle f,m_\e(\lambda+0\ii)\rangle + \langle f,\e(\lambda)\rangle \,.
    \end{align*}
   Using \eqref{eqn:megammarel} we can rewrite this as
   \begin{align}\label{eqn:sctreleqn2}
        \langle G_{\lambda+0\ii}\ast f, q C_+ \overline{q} m_\e(\lambda+0\ii)\rangle=-\ov{\Gamma(\lambda)}\sclp{f,\e(\lambda)}+\ov{\Gamma(\lambda)}\langle f,m_\e(\lambda-0\ii)\rangle. 
    \end{align}
    Putting \eqref{eqn:sctreleqn1} and \eqref{eqn:sctreleqn2} together, gives
    \begin{align*}
        -\overline{\Gamma(\lambda)}\langle f,\e(\lambda)\rangle +\overline{\Gamma(\lambda)}\langle f,m_\e(\lambda-0\ii)\rangle-\overline{\Gamma(\lambda)} \langle q C_+ \overline{q} R(\lambda+0\ii)f,\e(\lambda)\rangle =0
    \end{align*}
    and thus, since $\abs{\Gamma(\lambda)}=1\neq 0$ by \eqref{eqn:gammaoneident},we have  
    \begin{align*}
       \langle f,m_\e(\lambda-0\ii)\rangle- \langle qC_+ \overline{q} R(\lambda+0\ii)f+f,\e(\lambda)\rangle=0,
    \end{align*}
    or equivalently, using \eqref{eqn:Phif},
    \begin{align*}
        \sqrt{2\pi}\Phi(\lambda)+\ii c(\lambda)=0,
    \end{align*}
    which is the claimed identity.

    \medskip

    \emph{Step 2.} Let $\Lambda\subset[0,\infty)$ be a bounded interval with $\overline{\Lambda}\cap \sigma_\mathrm{p}(L_q)=\emptyset$ and let $f,g\in L^2_{+}(\R)\cap L^1_+(\R)$. According to Lemma \ref{lem:pertresolvcont}, $\sclp{R(k)f,g}$ is continuous with respect to $k$. We obtain, using  Step 1, that
    $$
    \frac1{2\pi\ii} \left( \sclp{R(\lambda+\ii\epsilon) f,g} - \sclp{R(\lambda-\ii\epsilon) f,g} \right)
    \to (\Phi f)(\lambda)\, \overline{(\Phi g)(\lambda)} \,,
    $$
    and these asymptotics are uniform for $\lambda\in\overline\Lambda$. 

    Meanwhile, by Stone's formula we have
    $$
    \frac{1}{2\pi\ii} \int_\Lambda \left( \sclp{R(\lambda+\ii\epsilon) f,g} - \sclp{R(\lambda-\ii\epsilon) f,g} \right)\dd\lambda
    \to \frac12 \left( \sclp{\1_{\overline\Lambda}(L_q) f, g} + \langle{\1_{\overset\circ\Lambda}(L_q) f, g}\rangle \right).
    $$
    Since the endpoints of $\Lambda$ are, by assumption, not in $\sigma_\mathrm{p}(L_q)$ and therefore by Corollary \ref{cornpm} not eigenvalues of $L_q$, the right side is equal to $\sclp{\1_{\Lambda}(L_q) f, g}$.

    Noting that $\sclp{R(\lambda\pm\ii\epsilon) f,g}$ is uniformly bounded in $(\lambda,\epsilon)\in\Lambda\times(0,1]$, we obtain the claimed formula by dominated convergence.
\end{proof}

\begin{proof}[Proof of Theorem \ref{thm:nosc}]
    As we have mentioned before, the set $\sigma_\mathrm{p}(L_q)$ is closed and therefore the open set $\R_+\setminus\sigma_\mathrm{p}(L_q)$ is the countable union of disjoint open intervals, $\R_+\setminus\sigma_\mathrm{p}(L_q)=\bigcup_n I_n$. Fix one of these intervals $I_n$ and let $K$ be a compact subinterval of $I_n$. For each $f\in L^2_+(\R)\cap L^1_+(\R)$, the function $|\Phi f|^2$ is bounded on $K$ and therefore, by Lemma \ref{lem:projection}, the spectral measure $\Lambda\mapsto (\1_{\Lambda}(L_q)f,f)$ is absolutely continuous on $K$, that is, $\1_K(L_q) f$ belongs to the absolutely continuous subspace of $L_q$. Since $L^2_+(\R)\cap L^1(\R)$ is dense in $L^2_+(\R)$, we deduce that $\ran \1_K(L_q)$ is contained in the absolutely continuous subspace of $L_q$, that is, the spectum of $L_q$ is purely absolutely continuous on $K$. Since $K$ is an arbitrary compact subinterval of $I_n$, it follows that the spectrum of $L_q$ is purely absolutely continuous on $I_n$. This implies that the spectrum is purely absolutely continuous on $\bigcup_n I_n =\R_+\setminus\sigma_\mathrm{p}(L_q)$, as claimed.
\end{proof}

\begin{proof}[Proof of Theorem \ref{thm:meFT}]
    \emph{Step 1.}
    We show that $\Phi$ can be extended to a bounded operator from $L^2_+(\R)$ to $L^2(\R_+)$ that satisfies the first equality in \eqref{eq:unitarity}.
    
    Let $f\in L^2_+(\R)\cap L^1_+(\R)$. We write $\R_+\setminus\sigma_\mathrm{p}(L_q)$ as a countable union of disjoint open intervals. Approximating each one of these intervals by compact intervals and using monotone convergence, we deduce from Lemma \ref{lem:projection} that
    $$
    \sclp{\1_{\R_+\setminus\sigma_\mathrm{p}(L_q)} f,f} = \int_{\R_+} |(\Phi f)(\lambda)|^2\,\dd\lambda \,.
    $$
    On the right side we write $\R_+$ as integration domain rather than $\R_+\setminus\sigma_\mathrm{p}(L_q)$, the value of $(\Phi f)(\lambda)$ for $\lambda$ in the null set $\sigma_\mathrm{p}(L_q)$ being irrelevant. Note that since 
    $$
    \sclp{\1_{\R_+\setminus\sigma_\mathrm{p}(L_q)} f,f}\leq \|f\|_2^2 \,,
    $$
    we infer that $\Phi f$ is square-integrable. Moreover, by Theorem \ref{thm:nosc} we have $\1_{\R_+\setminus\sigma_\mathrm{p}(L_q)} = P_\mathrm{ac}(L_q)$, so
    $$
    \sclp{P_\mathrm{ac}(L_q) f,f} = \int_{\R_+} |(\Phi f)(\lambda)|^2\,\dd\lambda \,.
    $$
    By density of $L^2_+(\R)\cap L^1(\R)$ in $L^2_+(\R)$, we infer that $\Phi$ extends to a bounded linear operator from $L^2_+(\R)$ to $L^2(\R_+)$ satisfying the first equality in \eqref{eq:unitarity}.

    \medskip

    \emph{Step 2.} We shall prove `one half' of \eqref{eq:diag1} and \eqref{eq:diag2}. Specifically, we shall show the inclusion $\subset$ in \eqref{eq:diag1} and the equality in \eqref{eq:diag2} for $f\in\dom L_q$.
    
    First, for $f\in\mathcal S_+(\R)\coloneqq \mathcal{S}(\R)\cap L^2_+(\R)$ and $\lambda\in\R_+\setminus\sigma_\mathrm{p}(L_q)$ we can use the fact that $m_\e(\lambda-0\ii)$ satisfies the equation $L_q m_\e(\lambda-0\ii) = \lambda m_\e(\lambda-0\ii)$ in the sense of tempered distributions to deduce that
    $$
    (\Phi L_q f)(\lambda) = \lambda (\Phi f)(\lambda) \,.
    $$
    Since the operator norm of $L_q$ is equivalent to the $H^1(\R)$-norm and since $\mathcal S_+(\R)$ is dense in $H^1_+(\R)$, we easily deduce that, if $f\in\dom L_q$, then $\int_{\R_+} \lambda^2 |(\Phi f)(\lambda)|^2\,\dd\lambda<\infty$ and $(\Phi L_qf)(\lambda) = \lambda (\Phi f)(\lambda)$ for almost all $\lambda\in\R_+$.

    \medskip

    \emph{Step 3.}
    Let us show the second equality in \eqref{eq:unitarity}. Taking into account the first equality there, which we have already proved, it suffices to show that $\ran\Phi$ is dense in $L^2(\R_+)$. To do this, suppose that there is a $h\in L^2(\R_+)$ that is orthogonal to $\ran\Phi$. Thus
    \begin{align*}
        \sclp{\Phi \1_\Lambda(\CMop) f,h}=0
    \end{align*}
    for every bounded interval $\Lambda\subset\R_+$ with $\overline\Lambda\cap\sigma_\mathrm{p}(L_q)=\emptyset$ and every $f\in L^2_+(\R)\cap L^1(\R)$. According to Step 2 and Lemma \ref{lem:intertwine} below this means that for all such $\Lambda$ and $f$ we have
    \begin{align*}
        0=\int_\Lambda \int_{-\infty}^\infty f(y) \ov{m_\e(y,\lambda-0\ii)}\dd y\ov{h(\lambda)}\dd \lambda
        =\int_{-\infty}^\infty f(y)\ov{\int_{\Lambda} h(\lambda)m_\e(y,\lambda-0\ii)\dd \lambda }\dd y \,.
    \end{align*}
    The interchange of integrals is allowed since $m_\e(\lambda-0\ii)$ is bounded for $\lambda$ in the intervals $\Lambda$ under consideration.
    
    Since $f$ is arbitrary, we deduce that for all such $\Lambda$ we have for almost every $x\in\R$
    \begin{align*}
        \int_{\Lambda} h(\lambda)m_\e(x,\lambda-0\ii)\dd x=0 \,.
    \end{align*}
    Restricting to $\Lambda$ with rational endpoints, say, we may assume that the full measure set of $x$'s is independent of $\Lambda$. Fixing now $x$ in this full measure set we obtain, by shrinking $\Lambda$ and applying the Lebesgue differentiation theorem, that
    $$
    h(\lambda)m_\e(x,\lambda-0\ii)=0
    \qquad\text{for a.e.}\ \lambda\in\R_+ \,.
    $$
    We also used the fact that $\sigma_\mathrm{p}(L_q)$ is a null set.
    
    We shall use the asymptotics of $m_\e(x,\lambda-0\ii)$ as $x\to\mp\infty$, see Lemma \ref{lem:uniqueme}. From the integral equation \eqref{eqn:meinteq} we easily see that these asymptotics are uniform for $\lambda$ in compact subsets of $\R_+\setminus\sigma_\mathrm{p}(L_q)$. Thus, fixing such a subset $K$, we find $x$ large enough and belonging to our full measure set such that $m_\e(x,\lambda-0\ii)\neq 0$ for all $\lambda\in K$. We deduce that $h(\lambda)=0$ for a.e.~$\lambda\in K$. Since $K$ is arbitrary, we find $h\equiv 0$. Thus, $\ran\Phi$ is dense in $L^2(\R_+)$, as we wanted to prove.

    \medskip

    \emph{Step 4.} We finally turn to the proof of the `other half' of \eqref{eq:diag1} and \eqref{eq:diag2}. Specifically, we shall show the inclusion $\supset$ in \eqref{eq:diag1}.

    To do so, let $f\in L^2_+(\R)$ with $\lambda\Phi f\in L^2(\R_+)$ and set $g:=\Phi^*(\lambda\Phi f)\in L^2(\R_+)$. Then for any $h\in L^2_+(\R)$
    $$
    \langle g, h \rangle = \int_{\R_+} \lambda (\Phi f)(\lambda) \overline{(\Phi h)(\lambda)}\,\dd\lambda \,.
    $$
    Under the stronger assumption $h\in\dom L_q$ we have by the first equality in \eqref{eq:unitarity}
    \begin{align*}
        \langle f, L_q h\rangle & = \int_{\R_+} (\Phi f)(\lambda) \overline{(\Phi L_q h)(\lambda)}\,\dd\lambda + \langle f,(1-P_\mathrm{ac}(L_q))L_q h \rangle \\
        & = \int_{\R_+} \lambda (\Phi f)(\lambda) \overline{(\Phi h)(\lambda)}\,\dd\lambda + \langle f,(1-P_\mathrm{ac}(L_q)) L_q h \rangle \,.
    \end{align*}
    Here in the last equality we used Step 2.

    In Theorem \ref{thm:simpasymp} below we will show that the set $\sigma_\mathrm{p}(L_q)$ is bounded. As a consequence, we have $\widetilde g:= L_q(1-P_\mathrm{ac}(L_q))f\in L^2_+(\R)$.

    To summarise, we have shown that for $h\in\dom L_q$ we have
    $$
    \langle f, L_q h\rangle = \langle g + \widetilde g, h \rangle \,.
    $$
    This proves that $f\in\dom L_q^*$ and, since $L_q$ is self-adjoint, that $f\in\dom L_q$. This concludes the proof.    
\end{proof}

In the previous proof we used the following technical lemma.

\begin{lemma}\label{lem:intertwine}
    Let $\mathcal H$ and $\mathcal G$ be Hilbert spaces, let $A$ and $B$ be self-adjoint operator in $\mathcal H$ and $\mathcal G$, respectively, and let $C:\mathcal H\to\mathcal G$ be a bounded operator. Assume that for every $f\in\dom A$, $Cf\in\dom B$ and
    \begin{equation}
        \label{eq:intertwine}
            C A f = B Cf \,.
    \end{equation}
    Then for any Borel set $\Lambda\subset\R$,
    \begin{equation}
        \label{eq:intertwine2}
            C \1_\Lambda (A) = \1_\Lambda(B) C \,.
    \end{equation}
\end{lemma}

The point of this lemma is that we do \emph{not} assume that the identity $CA=BC$ holds in the sense of (unbounded) operators. That is, we do \emph{not} assume that for all $f\in\mathcal H$ with $Cf\in\dom B$ we have $f\in\dom A$ and \eqref{eq:intertwine} holds. The weak assumption that we impose is dictated by our application.

\begin{proof}
    Let $g\in \mathcal H$ and $z\in\C_+$. Then $(A-z)^{-1} g\in\dom A$ and therefore, by assumption \eqref{eq:intertwine} with $f=(A-z)^{-1} g$, we have
    $$
    C g = (B-z) C (A-z)^{-1} g \,.
    $$
    Applying the bounded operator $(B-z)^{-1}$ we obtain the operator identity
    $$
    (B-z)^{-1} C = C (A-z)^{-1} \,.
    $$
    By the spectral theorem, this implies, for any $f\in\mathcal H$, $\phi\in\mathcal G$
    $$
    \int_\R \frac{d(\1_{(-\infty,\lambda)}(B)Cf,\phi)}{\lambda-z} =
    \int_\R \frac{d(\1_{(-\infty,\lambda)}(A)f,C^*\phi)}{\lambda-z} \,.
    $$
    Thus, the Stieltjes transform of the finite, signed measure 
    $$
    d\left( (\1_{(-\infty,\lambda)}(B)Cf,\phi) - (\1_{(-\infty,\lambda)}(A)f,C^*\phi)\right)
    $$ 
    vanishes on $\C_+$. This implies that the measure vanishes, that is, for any Borel set $\Lambda\subset\R$ we have
    $$
    (\1_{\Lambda}(B)Cf,\phi) - (\1_{\Lambda}(A)f,C^*\phi) = 0 \,.
    $$
    Since $f$ and $\phi$ are arbitrary, this implies \eqref{eq:intertwine2}.
\end{proof}

\subsection*{Wave operators and scattering matrix}

We end this section with a brief aside by putting the results we have proved so far into the framework of mathematical scattering theory as presented, for instance, in the textbooks of Yafaev  \cite{yafaev_mathematical_1992,yafaev_mathematical_2010}. Since we will not make use of the results elsewhere in this paper, we will omit proofs.

There are various ways to see that both wave operators
$$
W_\pm(L_q,L_0) := s-\lim_{t\to\pm\infty} \exp(\ii t L_q) \exp(-\ii tL_0)
$$
exist. This follows by the Birman--Krein theorem \cite{birmankrein} from the fact that the difference of resolvents $(L_q+\kappa)^{-1}-(L_0+\kappa)^{-1}$ is trace class for all sufficiently large $\kappa$. The latter fact is shown in \cite{killip_sharp_2023} under the sole assumption $q\in L^2_+(\R)$.

We define operators $\Phi_\pm:L^2_+(\R)\to L^2(\R_+)$ by
$$
(\Phi_+ f)(\lambda):= (\Phi f)(\lambda) \,,
\qquad
(\Phi_- f)(\lambda) := \overline{\Gamma(\lambda)} (\Phi f)(\lambda) \,.
$$
Note that, by \eqref{eqn:megammarel}, $\Phi_-$ is the same as $\Phi$, except that $m_\e(\lambda-\ii 0)$ is replaced by $m_\e(\lambda+\ii 0)$ in its definition.

Then, proceeding exactly as in \cite[Subsection 6.6.2]{yafaev_mathematical_2010}, we find that
$$
W_\pm(L_q,L_0) = \Phi_\pm^* \mathcal F \,,
$$
where $\mathcal F$ is the Fourier transform. The first equality in \eqref{eq:unitarity} implies that the range of $W_\pm(L_q,L_0)$ is the absolutely continuous subspace of $L_q$, that is, the wave operators are complete.

As a consequence of the above formulas for the wave operators, the scattering operator
$$
S(L_q,L_0) := W_+(L_q,L_0)^* W_-(L_q,L_0)
$$
is given by
$$
S(L_q,L_0) = \mathcal F^* \Phi_+ \Phi_-^* \mathcal F.
$$
Thus, in view of the second identity in \eqref{eq:unitarity},
$$
(\mathcal F S(L_q,L_0)f)(\lambda) = \Gamma(\lambda) \, (\mathcal F f)(\lambda) \,.
$$
That is, in the Fourier representation, the scattering operator acts by multiplication by $\Gamma$. In this sense, the number $\Gamma(\lambda)$ is the scattering matrix at energy $\lambda$. Note that the equation $|\Gamma(\lambda)|=1$, see \eqref{eqn:gammaoneident}, corresponds to unitarity of the scattering matrix.

\newpage

\part{Direct scattering theory} 

In this second part we introduce the inhomogeneous Jost solutions that, together with the generalised eigenfunctions, will form the basis of the direct scattering theory for \eqref{eqn:CMeqn}. In addition to $q\in L^2_+(\R)$, we will require that $q\in L^1_+(\R)$. 


\section{Inhomogeneous Jost solution}\label{sec:jostinhom}

We define the inhomogeneous Jost solution $m_0(k)\in L^\infty(\R)$, for $k\in \Chat\setminus(\sigma_\mathrm{p}(\CMop))^\diamond$, by 
\begin{align*}
    \boxed{
    m_0(k):=R(k)q.
    }
\end{align*}
By Lemma \ref{lem:pertresolvcont} this is well defined and satisfies the property that $k\mapsto \sclp{m_0(k),g}$ is a continuous map on $\Chat$, for any $g\in L^1_+(\R)$. 

From the resolvent formula, we can deduce the integral equation
\begin{align}\label{eqn:m0inteq}
    m_0(k)=R_0(k)q+R_0(k)\pot m_0(k) \,,
\end{align}
which will be helpful later. 


\subsection*{The scattering coefficient}
For $\lambda\in [0,\infty)\backslash\sigma_\mathrm{p}(\CMop)$, we define 
    \begin{align}
        \boxed{
        \beta(\lambda):= \ii \int_{\R} q(y)\left([C_+\overline{q}m_0(\lambda+0\ii)](y)+1\right)\e^{-\ii \lambda y}\dd y \,.
        }
        \label{eqn:betadef}
    \end{align}

In the following lemma, we show that $\beta$ corresponds to $\sqrt{2\pi}\ii\Phi(q)$, where $\Phi$ is given by \eqref{eqn:Phif} in Section \ref{sec:distortedFT}. Using the decomposition \eqref{eqn:fspectraldecomp} in the special case of $f=q$, the results of the last section amount to the reconstruction formula 
\begin{align}\label{eqn:qreconstruction}
    q(x)= \sum_{j=1}^N\frac{\sclp{q,\phi_j}}{\norm{\phi_j}_2^2}\phi_j(x)+\frac{\ii}{2\pi}\int_0^\infty \beta (\lambda)m_\e(x,\lambda-0\ii)\dd \lambda.
\end{align}
An analogous representation was non-rigorously derived for the Benjamin--Ono equation in \cite{fokas_inverse_1983}.
    
\begin{lemma}\label{lem:scatobjsbeta}
    Let $q\in L^2_+(\R)\cap L^1_+(\R)$. Then $\beta\in C([0,\infty)\backslash\sigma_\mathrm{p}(\CMop))$ and, for any $\lambda\in [0,\infty)\backslash\sigma_\mathrm{p}(\CMop)$,
    \begin{align}
        m_0(x,\lambda+0\ii)-m_0(x,\lambda-0\ii)&=\beta(\lambda)m_\e(x,\lambda-0\ii) \,,
        \label{eqn:m0mebetarel} \\
        \ii \overline{\beta(\lambda)}&=\int_{\R}\overline{q(y)}m_\e(y,\lambda-0\ii)\dd y \,,
        \label{eqn:fbetaident}
        \intertext{and}
        \abs{\beta(\lambda)}^2&=2\im\int_{\R}\overline{q(y)}m_0(y,\lambda+0\ii)\dd y \,. \label{eqn:betasquaredident}
    \end{align}
\end{lemma}

\begin{proof}
    The continuity of $\beta$ follows writing 
    \begin{align*}
        \beta(\lambda) = \ii\sclp{q,\e(\lambda)}
        +\sclp{R(\lambda+0\ii)q, q C_+\ov{q}\e(\lambda)}
    \end{align*}
    and arguing similarly to Lemma \ref{lem:scatobjsGamma}, in particular, using the continuity of $\lambda\mapsto\ov{q}\e(\lambda)$ in $L^2$ and the continuity of Fourier transforms. 

    For \eqref{eqn:m0mebetarel}, we use the integral equation \eqref{eqn:m0inteq} for $m_0$ and apply the relation \eqref{eqn:Gpmdiff} to find
    \begin{align*}
        m_0(\lambda+0\ii)-m_0(\lambda-0\ii) & = \left(G_{\lambda+0\ii}-G_{\lambda-0\ii}\right)\ast q+(G_{\lambda+0\ii}-G_{\lambda-0\ii})\ast (qC_+\overline{q}m_0(\lambda+0\ii))\\
        & \quad +G_{\lambda-0\ii}\ast(qC_+\overline{q}(m_0(\lambda+0\ii)-m_0(\lambda-0\ii)))\\
        & = \ii \, \e(\lambda)\ast q+\ii \e(\lambda)\ast (qC_+\overline{q}m_0(\lambda+0\ii))\\
        & \quad +G_{\lambda-0\ii}\ast(qC_+\overline{q}(m_0(\lambda+0\ii)-m_0(\lambda-0\ii))) \,.
    \end{align*}
    The identity follows by rearranging for $m_0(\lambda+0\ii)-m_0(\lambda-0\ii)$ and comparing with the definitions of $m_\e(\lambda-0\ii)$ and $\beta(\lambda)$.

    The statement \eqref{eqn:fbetaident} follows directly from Step 1 of the proof of Lemma \ref{lem:projection}, using that $\beta(\lambda)=c(\lambda)$ with $f=q$. 

    For \eqref{eqn:betasquaredident}, using \eqref{eqn:Ginprodfg} and the integral equation \eqref{eqn:m0inteq} for $m_0$ we compute 
    \begin{align*}
        \ii \abs{\beta(\lambda)}^2=&\ii \abs{\sclp{\pot m_0(\lambda+0\ii)+q,\e(\lambda)}}^2\\=&\sclp{G_{\lambda+0\ii}\ast (\pot m_0(\lambda+0\ii)+q),\pot m_0(\lambda+0\ii)+q}\\&-\sclp{\pot m_0(\lambda+0\ii)+q,G_{\lambda+0\ii}\ast (\pot m_0(\lambda+0\ii)+q)}\\
        =&\sclp{m_0(\lambda+0\ii),\pot m_0(\lambda+0\ii)+q}-\sclp{\pot m_0(\lambda+0\ii)+q,m_0(\lambda+0\ii)}
        \\=&\sclp{m_0(\lambda+0\ii),q}-\sclp{q,m_0(\lambda+0\ii)}=2\ii \im \sclp{m_0(\lambda+0\ii),q},
    \end{align*}
    where we used the self-adjointness of $\pot$. This establishes \eqref{eqn:betasquaredident}. 
\end{proof}

In passing we note that
\begin{equation}
    \label{eq:lowenergy2}
    \beta(0) = 0
    \qquad\text{if}\ 0\not\in\sigma_\mathrm{p}(\CMop) \,.
\end{equation}
This follows from \eqref{eqn:fbetaident}, since $m_\e(0-0\ii)=1$ by \eqref{eq:lowenergy1} and since the Fourier transform of a function in $L^1_+(\R)$ vanishes at the origin.

\begin{remark}\label{rem:b}
    The definition of $m_0(k)$ and $\beta(\lambda)$ required the assumption $q\in L^2_+ \cap L^1_+$. Let us see what remains if we only assume $q\in L^2_+$. In this case the inhomogeneous Jost solution can still be defined for $k\in\C\setminus\sigma(L_q)$ by $m_0(k):=R(k) q$. Moreover, formula \eqref{eqn:fbetaident} means that $\beta = \sqrt{2\pi} \ii \Phi(q)$, where $\Phi$ is the distorted Fourier transform from \eqref{eqn:Phif}. Since, according to Theorem \ref{thm:meFT}, $\Phi$ has an extension to $L^2_+(\R)$, we can \emph{define} $\beta := \sqrt{2\pi} \ii \Phi(q)$ for $q\in L^2_+(\R)$. Now $\beta$ is no longer defined pointwise, but rather as an $L^2$-function. This extension will be relevant in Section \ref{sec:traceform}.
\end{remark}


\subsection*{A differential formula for $m_\e$ and $\Gamma$}

While the homogeneous and inhomogeneous Jost functions considered so far are not continuous in $L^\infty(\R)$, but only in a weak sense, it is easy to see that by modulating $m_\e(\lambda\pm0\ii)$ by a plane wave, the functions $\ov{\e(\lambda)}m_\e(\lambda\pm0\ii)$ become continuous. In what follows, we show that it is in fact differentiable, and that it satisfies a relation connecting it to the inhomogeneous Jost function and the scattering coefficient. A similar relation connects the scattering matrix and the scattering coefficient.

\begin{lemma}\label{lem:emederiv}
  Let $q\in L^2_+(\R)\cap L^1_+(\R)$. Then $\lambda\mapsto \ov{\e(\lambda)}m_\e(\lambda\pm 0\ii)$ and $\lambda\mapsto\Gamma(\lambda)$ are differentiable on $\R_+\setminus\sigma_\mathrm{p}(\CMop)$ and for any $\lambda\in \R_+\backslash\sigma_\mathrm{p}(\CMop)$, 
  \begin{align*}
      & \partial_\lambda(\ov{\e(\lambda)} \, m_\e(\lambda+0\ii))=-\frac{\overline{\beta(\lambda)}\Gamma(\lambda)}{2\pi\ii }\ \ov{\e(\lambda)} \, m_0(\lambda+0\ii) \,, \\
      & \partial_\lambda(\ov{\e(\lambda)} \, m_\e(\lambda-0\ii))=-\frac{\overline{\beta(\lambda)}}{2\pi\ii }\ \ov{\e(\lambda)} \, m_0(\lambda-0\ii)
  \end{align*}
  and 
  \begin{align}\label{eqn:gammaderivident}
        \partial_\lambda \Gamma(\lambda)=-\frac{|\beta(\lambda)|^2}{2\pi\ii } \ \Gamma(\lambda) \,.
    \end{align}
\end{lemma}

\begin{proof}
    The proof is based on the operator identity
    \begin{equation}\label{eqn:Chprop}
        \ov{\e(\lambda)} C_+ \e(\lambda) = C_+ + C_\lambda
    \end{equation}
    with $C_\lambda\coloneqq\F^{-1}\1_{(-\lambda,0)}\F$ for $\lambda\geq 0$, where $\F$ is the Fourier transform, 
    $$
    (\F f)(\xi) = \frac1{\sqrt{2\pi}} \int_\R \e^{-\ii\xi x} f(x)\dd x \,.
    $$
    
    \medskip
    
    \emph{Step 1.} We show that for every $\phi\in L^\infty$ the $L^1$-valued function $\lambda\mapsto qC_\lambda \ov{q}\phi$ is differentiable on $[0,\infty)$ with
    $$
    \partial_\lambda q C_\lambda\ov{q}\phi = \frac1{2\pi} \int_\R \e^{\ii\lambda y} \ov{q(y)}\phi(y) \dd y \ \ov{\e(\lambda)} q \,.
    $$
    
    To see this, let $\lambda\geq 0$ and let $h\in\R$ with $\lambda+h\geq 0$. Then
    $$
    h^{-1} q(x)((C_{\lambda+h} - C_\lambda)\ov{q}\phi)(x) = \frac{1}{\sqrt{2\pi} h} q(x) \int_{-\lambda-h}^{-\lambda} \e^{\ii\xi x} \mathcal F(\ov{q}\phi)(\xi) \dd\xi \,. 
    $$
    Since $q\in L^1$ and $\phi\in L^\infty$ the quantity $h^{-1} \int_{-\lambda-h}^{-\lambda} \e^{\ii\xi x} \mathcal F(\ov{q}\phi)(\xi) \dd\xi$ is uniformly bounded with respect to $x$ and $h$ and converges, for each $x\in\R$, to $\e^{-\ii\lambda x} \mathcal F(\ov{q}\phi)(-\lambda)$ as $h\to 0$. Using again $q\in L^1$ and dominated convergence we see that as $h\to 0$ we have
    $$
    h^{-1} q(C_{\lambda+h} - C_\lambda)\ov{q}\phi \to \frac{1}{\sqrt{2\pi}}\mathcal F(\ov{q}\phi)(-\lambda) \, q \, \ov{\e(\lambda)}
    $$
    in $L^1$. This proves the claimed differentiability.
    
    \medskip

    \emph{Step 2.}
    For $\lambda\in[0,\infty)$, we note that setting, for $\phi\in L^\infty$,
    $$
    \widetilde T_{\lambda\pm\ii 0}\phi := \ov{\e(\lambda)} \, R_0(\lambda\pm0\ii) qC_+\ov{q} \e(\lambda)\phi
    $$
    defines a bounded operator on $L^\infty$. We shall show that for every $\phi\in L^\infty$, the $L^\infty$-valued function $\lambda\mapsto \widetilde T_{\lambda\pm\ii 0}\phi$ is differentiable on $[0,\infty)$ with
  $$
  \partial_\lambda \widetilde T_{\lambda\pm\ii 0}\phi = \frac{1}{2\pi} \int_\R \e^{\ii\lambda y} \overline{q(y)}\phi(y)\dd y \,
  \ \ov{\e(\lambda)} \ R_0(\lambda\pm\ii 0) q \,.
  $$

    Indeed, it follows from \eqref{eqn:Chprop} that
    \begin{align*}
        \widetilde T_{\lambda\pm\ii 0}\phi&=\ii\int_{\mp\infty}^x\e^{-\ii \lambda y}qC_+\e(\lambda)\ov{q}\phi\dd y=\ii\int_{\mp\infty}^x\pot\phi\dd y+\ii\int_{\mp\infty}^xqC_\lambda \ov{q}\phi\dd y,
    \end{align*}
    and therefore, if $h\in\R$ is so small that $\lambda+h\geq 0$, then
    \begin{align*}
        h^{-1}(\widetilde T_{\lambda+h\pm\ii 0}- \widetilde T_{\lambda\pm\ii 0})\phi(x)= \ii h^{-1}\int_{\mp\infty}^x q(C_{\lambda+h}-C_\lambda)\ov{q}\phi\dd y\,.
    \end{align*}
    By Step 1 we see that as $h\to 0$ we have, uniformly with respect to $x$,
    \begin{align*}
         h^{-1}(\widetilde{T}_{\lambda+h\pm\ii 0}-\widetilde{T}_{\lambda\pm\ii 0})\phi(x) & \to \frac{1}{2\pi} \int_\R \e^{\ii\lambda y'} \ov{q(y')}\phi(y') \dd y' \ \ii \int_{\mp\infty}^x q\e^{-\ii y\lambda }\dd y \\
         & \quad = \frac{1}{2\pi} \int_\R \e^{\ii\lambda y'} \ov{q(y')}\phi(y') \dd y' \ \ov{\e(\lambda)} \ R_0(\lambda\pm\ii 0) q \,. 
    \end{align*}
    This proves the claimed differentiability.

    \medskip

    \emph{Step 3.} Let us show the differentiability of $\lambda\mapsto\ov{\e(\lambda)}m_\e(\lambda\pm 0\ii)$ and the claimed formula for their derivatives.

    It follows from the resolvent identity \eqref{eq:resolventformula} that for all $\lambda\in[0,\infty)\setminus\sigma_\mathrm{p}(L_q)$
    $$
    1+R(\lambda\pm 0\ii) q C_+ \ov{q} = (1- R_0(\lambda\pm 0\ii) q C_+ \ov{q})^{-1}
    $$
    and, consequently,
    \begin{equation}
        \label{eq:diffemeproof}
            \ov{\e(\lambda)} \left( 1+R(\lambda\pm 0\ii) q C_+ \ov{q} \right) \e(\lambda)
        = (1- \widetilde T_{\lambda\pm 0\ii})^{-1} \,.
    \end{equation}
    Thus, by definition of $m_\e(\lambda\pm 0\ii)$ we have
    \begin{equation}
        \label{eq:diffemeproof2}
            \ov{\e(\lambda)}m_\e(\lambda\pm 0\ii) = (1- \widetilde T_{\lambda\pm 0\ii})^{-1} 1 \,.
    \end{equation}
    Fix $\lambda\in\R_+\setminus\sigma_\mathrm{p}(L_q)$. Note that by \eqref{eq:diffemeproof} the operators $(1- \widetilde T_{\lambda+h\pm 0\ii})^{-1}$ exist for all sufficiently small $h\in\R$. Moreover, it follows from Step 2 that
    $$
    \| \widetilde T_{\lambda+h\pm\ii 0}- \widetilde T_{\lambda\pm\ii 0} \|_{L^\infty\to L^\infty} \leq (2\pi)^{-1} \|q\|_1^2 |h| \,,
    $$
    so, in particular, $\widetilde T_{\lambda+h\pm\ii 0} \to \widetilde T_{\lambda\pm\ii 0}$ in norm as $h\to 0$. Since
    \begin{align*}
        (1- \widetilde T_{\lambda+h\pm 0\ii})^{-1} - (1- \widetilde T_{\lambda\pm 0\ii})^{-1} & = - \left( 1 - \left(1-(1-\widetilde T_{\lambda\pm 0\ii})^{-1}(\widetilde T_{\lambda+h\pm 0\ii} - \widetilde T_{\lambda\pm 0\ii}) \right)^{-1} \right) \\
        & \quad \times (1- \widetilde T_{\lambda\pm 0\ii})^{-1}
    \end{align*}
    we see that $(1- \widetilde T_{\lambda+h\pm 0\ii})^{-1} \to (1- \widetilde T_{\lambda\pm 0\ii})^{-1}$ in norm as $h\to 0$.

    Equation \eqref{eq:diffemeproof2} implies
    \begin{align*}
        & h^{-1} \left( \ov{\e(\lambda+h)}m_\e(\lambda+h\pm 0\ii) - \ov{\e(\lambda)}m_\e(\lambda\pm 0\ii) \right) \\
        & =(1- \widetilde T_{\lambda+h\pm 0\ii})^{-1} h^{-1} \left( \widetilde T_{\lambda+h\pm\ii 0}- \widetilde T_{\lambda\pm\ii 0} \right) (1- \widetilde T_{\lambda\pm 0\ii})^{-1} 1 \,.    
    \end{align*}
    According to Step 2, we have
    $$
    h^{-1} \left( \widetilde T_{\lambda+h\pm\ii 0}- \widetilde T_{\lambda\pm\ii 0} \right) (1- \widetilde T_{\lambda\pm0\ii})^{-1} 1 \to \frac1{2\pi} \int_\R \e^{\ii\lambda y} \overline{q(y)} (1- \widetilde T_{\lambda\pm 0\ii})^{-1} 1 \dd y \ \overline{\e(\lambda)} \ R_0(\lambda\pm \ii 0) q \,.
    $$
    Since $(1- \widetilde T_{\lambda+h\pm 0\ii})^{-1} \to (1- \widetilde T_{\lambda\pm 0\ii})^{-1}$ in norm, we conclude that
    \begin{align*}
        & h^{-1} \left( \ov{e(\lambda+h)}m_\e(\lambda+h\pm 0\ii) - \ov{e(\lambda)}m_\e(\lambda\pm 0\ii) \right) \\
        & \to \frac1{2\pi} \int_\R \e^{\ii\lambda y} \overline{q(y)} (1- \widetilde T_{\lambda\pm 0\ii})^{-1} 1 \dd y \ (1- \widetilde T_{\lambda\pm 0\ii})^{-1} \ \overline{\e(\lambda)} \ R_0(\lambda\pm\ii 0) q \,.    
    \end{align*}
    According to \eqref{eq:diffemeproof} and the resolvent identity, we have
    \begin{align*}
        (1- \widetilde T_{\lambda\pm 0\ii})^{-1} \ \overline{\e(\lambda)} \ R_0(\lambda\pm\ii 0) q & = \overline{\e(\lambda)} \, (1+R(\lambda\pm 0\ii) q C_+\ov{q}) R_0(\lambda\pm 0\ii) q \\&= \overline{\e(\lambda)}\, R(\lambda\pm 0\ii) q \\
        &= \overline{\e(\lambda)}\, m_0(\lambda\pm 0\ii) \,.
    \end{align*}
    Similarly, one finds
    $$
    (1- \widetilde T_{\lambda\pm 0\ii})^{-1} 1 = \overline{\e(\lambda)} \, (1+R(\lambda\pm 0\ii) q C_+\ov{q}) \e(\lambda) = \overline{\e(\lambda)} \, m_\e(\lambda\pm 0\ii)
    $$
    and therefore, using \eqref{eqn:fbetaident} and \eqref{eqn:megammarel},
    $$
    \int_\R \e^{\ii\lambda y} \overline{q(y)} (1- \widetilde T_{\lambda\pm 0\ii})^{-1} 1 \dd y = 
    \begin{cases}
    \ii \, \overline{\beta(\lambda)} \, \Gamma(\lambda) & \text{for the upper sign} \,,\\
    \ii \, \overline{\beta(\lambda)} & \text{for the lower sign} \,.
    \end{cases}
    $$
    This proves the claimed differentiability.

    \medskip

    \emph{Step 4.} We shall show that for any $\phi\in L^\infty$, the $L^1$-valued map $\lambda\mapsto \ov{\e(\lambda)}q C_+\ov{q}\e(\lambda)\phi$ is differentiable on $[0,\infty)$ with
    $$
    \partial_\lambda \left( \ov{\e(\lambda)}q C_+\ov{q}\e(\lambda)\phi \right) = \frac1{2\pi} \int_\R \e^{-\ii\lambda y} \ov{q(y)} \phi(y)\dd y \ \ov{\e(\lambda)} q \,.
    $$

    Indeed, in view of \eqref{eqn:Chprop} we have
    $$
    \ov{\e(\lambda)}q C_+\ov{q}\e(\lambda)\phi = qC_+\ov{q}\phi + q C_\lambda \ov{q}\phi \,,
    $$
    so the assertion follows immediately from Step 1.

    \medskip

    \emph{Step 5.} Let us show the differentiability of $\Gamma$.

    From the definition \eqref{eqn:Gammadef} of $\Gamma$ and the commutator identity \eqref{eqn:Chprop} we get
    \begin{align*}
        \Gamma(\lambda) & = 1 + \ii \, \langle \ov{\e(\lambda)} m_\e(\lambda+0\ii), \ov{\e(\lambda)} q C_+ \ov{q} \e(\lambda) \rangle \,.
    \end{align*}
    The second term on the right side is differentiable by Steps 2 and 4, together with the fact that the $L^\infty$ norm of $\ov{\e(\lambda)} m_\e(\lambda+0\ii)$ and the $L^1$ norm of $\ov{\e(\lambda)} q C_+ \ov{q} \e(\lambda)$ are bounded on compact subsets of $\R_+\setminus\sigma_\mathrm{p}(L_q)$. Using the formulas for the derivatives from Steps 2 and 4 we obtain
    \begin{align*}
        \partial_\lambda \Gamma(\lambda) & = -\frac{\ov{\beta(\lambda)}\,\Gamma(\lambda)}{2\pi} \, \langle \ov{\e(\lambda)} m_0(\lambda+0\ii), \ov{\e(\lambda)} q C_+ \ov{q} \e(\lambda) \rangle \\
        & \quad + \frac{\ii}{2\pi} \int_\R \e^{-\ii\lambda y} q(y) \dd y \,\langle \ov{\e(\lambda)} m_\e(\lambda+0\ii), \ov{\e(\lambda)} q \rangle \\
        & = -\frac{\ov{\beta(\lambda)}\,\Gamma(\lambda)}{2\pi} \left( \langle m_0(\lambda+0\ii), q C_+ \ov{q} \e(\lambda) \rangle
        + \int_\R \e^{-\ii\lambda y} q(y) \dd y \right),    
    \end{align*}
    where, in the last step, we used \eqref{eqn:megammarel} and \eqref{eqn:fbetaident}. The claimed identity now follows from the definition \eqref{eqn:Gammadef} of $\Gamma$. This completes the proof of the lemma.
\end{proof}



\subsection*{Example: Jost solutions for one-solitons}

According to \cite{gerard_calogero--moser_2023}, the one-soliton solutions for the continuum Calogero--Moser derivative nonlinear Schr\"odinger equation \eqref{eqn:CMeqn} are given by
$$
e^{\ii\theta -\ii\eta^2 t} \, \lambda^{1/2} \, q_\eta(\lambda(x-2\eta t) + y)
$$
with $\theta\in\R/2\pi\Z$, $y\in\R$, $\lambda\in\R_+$ and $\eta\in\overline{\R_+}$, where
\begin{align*}
    q_\eta(x)\coloneqq \frac{\sqrt{2} \, \e^{\ii \eta x}}{x+\ii} \,.
\end{align*}
The associated operator $\CMop[q_\eta]$ has $\lambda_1=\eta$ as eigenvalue with the corresponding eigenfunction given by $\phi_1=q_\eta$. This follows from the identity
\begin{align}\label{eqn:qetaCplus}
    C_+\abs{q_\eta}^2=C_+\frac{2}{x^2+1}=\frac{\ii}{x+\ii}
\end{align}
and the fact that 
\begin{align*}
    (-\ii \partial_x-\eta)q_\eta=q_\eta\frac{\ii}{x+\ii} \,. 
\end{align*}
The eigenvalue $\lambda_1$ is simple and the unique eigenvalue of $\CMop[q_\eta]$, as follows from \cite{gerard_calogero--moser_2023}; see the discussion after Lemma \ref{lem:qeigenfuncinprod}.

The corresponding homogeneous Jost solution is given by 
\begin{align*}
    m_\e(x,\lambda\pm0\ii)=\begin{cases}\e^{\ii\lambda x} \, \frac{x-\ii}{x+\ii} & \text{if}\ \lambda\in(\eta,\infty) \,,\\
    \e^{\ii \lambda x} & \text{if}\ \lambda\in(0,\eta) \,.
    \end{cases}
\end{align*}
This follows by Lemma \ref{lem:uniqueme} from the fact that the right side satisfies the same differential equation and has the same asymptotic behaviour as the left side.

For the scattering matrix, we obtain
$$
\Gamma(\lambda)= 1 \qquad\text{ for all }\lambda\in\R_+\setminus\{\eta\} \,.
$$
This can either be deduced from relation \eqref{eqn:megammarel}, or else by a contour integration using the definition \eqref{eqn:Gammadef} of $\Gamma$. We perform a related calculation momentarily.

Since $q\not\in L^1(\R)$, the inhomogeneous Jost solutions $m_0$ and the scattering coefficient $\beta$ are not defined by our discussion so far. We can, however, proceed directly and take $m_0$ as solution to the inhomogeneous equation, with asymptotics 
\begin{align*}
    \begin{cases}
        m_0(k)\rightarrow 0 \ \ \ \quad \qquad \mathrm{as} \ |x|\rightarrow \infty & \mathrm{if}\ k\in \C\backslash[0,\infty) \,,\\
        m_0(\lambda\pm0\ii)\rightarrow 0 \qquad  \mathrm{as}\ x\rightarrow \mp \infty &\mathrm{if}\ \lambda\geq 0 \,.
    \end{cases}
\end{align*}
Then we can define the scattering coefficients $\beta$ by \eqref{eqn:betadef}.

In particular, for $k\in \Chat\backslash\{\eta+0\ii,\eta-0\ii\}$ the inhomogeneous Jost solution is
\begin{align*}
    m_0(x,k)=-\frac{1}{k-\eta} \, q_\eta,
\end{align*}
since the right side satisfies $-\ii\psi' -qC_+\ov{q}\psi = k\psi$ (as a consequence of the equation for $q_\eta$) and vanishes at infinity.

We claim that the corresponding scattering coefficient satisfies
\begin{align*}
    \beta(\lambda)=0\qquad\text{ for all }\lambda\in\R_+\setminus\{\eta\} \,.
\end{align*}
Formally, this is consistent with the scattering relation \eqref{eqn:m0mebetarel} between the Jost solutions calculated above.

We prove the claimed formula for $\beta$ by computing the (not absolutely convergent) integral in its definition by contour integration. We have 
\begin{align*}
    \beta(\lambda)&=\ii\sclp{q_\eta C_+\ov{q_\eta}m_0(\lambda\pm0\ii)+q_\eta,\e(\lambda)}=-\frac{\ii}{\lambda-\eta}\sclp{q_\eta C_+\abs{q_\eta}^2,\e(\lambda)}+\ii \sclp{q_\eta,\e(\lambda)},
\end{align*}
so that, using \eqref{eqn:qetaCplus},
\begin{align}\label{eqn:betasoliton}
    \beta(\lambda)=\frac{\sqrt{2}}{\lambda-\eta}\int_{-\infty}^\infty\frac{\e^{\ii(\eta-\lambda) x}}{(x+\ii)^2}\dd x+\ii\sqrt{2} \int_{-\infty}^\infty \frac{\e^{\ii(\eta-\lambda)x}}{x+\ii}\dd x.
\end{align}

We calculate these terms by contour integration. For $R>0$, take $C_R^{\pm}$ to be the semi-circles $\{z\in \C\colon |z|=R, \pm \im z\geq 0\}$ with clockwise orientation. Then, if $\eta-\lambda>0$ we have 
\begin{align*}
    \int_{-R}^R \frac{\e^{\ii(\eta-\lambda) x}}{(x+\ii)^n}\dd x=\int_{C_R^+}\frac{\e^{\ii(\eta-\lambda) z}}{(z+\ii )^n}\dd z,\quad n=1,2.
\end{align*}
Since the right hand side goes to zero as $R\rightarrow \infty$ we conclude from \eqref{eqn:betasoliton} that $\beta(\lambda)=0$ for $\lambda<\eta$.

For $\eta-\lambda<0$ observe that
\begin{align}\label{eqn:betacontourint}
    \int_{-R}^R \frac{\e^{\ii(\eta-\lambda) x}}{(x+\ii)^n}\dd x+\int_{C_R^{-}}\frac{\e^{\ii (\eta-\lambda)z}}{(z+\ii )^n}\dd z=2\pi \ii\, \mathrm{Res}_{z=-\ii }\left(\frac{\e^{\ii (\eta-\lambda)z}}{(z+\ii )^n}\right), \quad n=1,2.
\end{align}
For $n=1$ we have
\begin{align*}
    \mathrm{Res}_{z=-\ii }\left(\frac{\e^{\ii (\eta-\lambda)z}}{(z+\ii )}\right)=\e^{\eta-\lambda},
\end{align*}
while for $n=2$,
\begin{align*}
    \mathrm{Res}_{z=-\ii }\left(\frac{\e^{\ii (\eta-\lambda)z}}{(z+\ii )^2}\right)=\ii(\eta-\lambda)\e^{\eta-\lambda},
\end{align*}
using the general formula for higher order poles. Taking $R\rightarrow\infty$ in \eqref{eqn:betacontourint} and using that the second term vanishes, after cancellations in \eqref{eqn:betasoliton} we again find that $\beta(\lambda)=0$ for $\lambda>\eta$. 

We note also that in Remark \ref{rem:b} we defined a function $B$, which is the substitute for $|\beta|^2$ for non-$L^1$ functions $q$. By definition $B$ is the density of the spectral measure of $P_\mathrm{ac}(L_q)q$ for $L_q$. In our example of the one-soliton potential, $q_\eta$ is an eigenfunction and therefore $P_\mathrm{ac}(L_{q_\eta})q_\eta = 0$. This shows that $B=0$ a.e.~and is consistent with the equality $\beta=0$ a.e.



\section{High energy asymptotics} 

In this section, we derive the asymptotics of the Jost functions and scattering quantities in the high energy limit. Our main result in this regard is the following.

\begin{theorem}\label{thm:simpasymp}
     Let $q\in L^2_{+}(\R)$. Then
     \begin{align}
         \lim_{\lambda\rightarrow \infty}\norm{m_\e(\lambda\pm0\ii)-\e^{\ii \lambda x}\e^{\ii \int_{\mp\infty}^x\abs{q(t)}^2\dd t}}_\infty&=0 \,,\label{eqn:simplemeasmp}
         \intertext{and}
         \lim_{\lambda\rightarrow \infty}\abs{\Gamma(\lambda)-\e^{\ii\int_{\R}|q(t)|^2\dd t}}&=0 \,. \label{eqn:simplegammaasmp}
         \end{align}
    If, in addition, $q\in L^1(\R)$, then
    \begin{align}
         \lim_{\abs{k}\rightarrow \infty}\norm{m_0(k)}_\infty&=0 \,,\label{eqn:simplem0asymp}
         \intertext{and}
         \lim_{\lambda\rightarrow \infty} \beta(\lambda) &=0 \,.\label{eqn:simplebetaasmp}
     \end{align} 
\end{theorem}

We emphasise that here and below, when writing $\lim_{|k|\to\infty}$ we mean a limit that is uniform with respect to the argument of $k$.

The main difficulty will be to prove \eqref{eqn:simplemeasmp} and \eqref{eqn:simplem0asymp}. The remaining limits \eqref{eqn:simplegammaasmp} and \eqref{eqn:simplebetaasmp} will then follow relatively easily from the definition of these scattering coefficients. To prove \eqref{eqn:simplemeasmp} and \eqref{eqn:simplem0asymp}, for $k\in \Chat$ we denote
\begin{align*}
    T_k\coloneqq R_0(k)\pot \,.
\end{align*}
Under the assumption $q\in L^2_+(\R)$ this defines a bounded operator $T_k\colon L^\infty\rightarrow L^\infty$ and we recall that the functions $m_\e(\lambda\pm0\ii)$ and $m_0(k)$ are defined through an equation that involves $(1-T_k)^{-1}$. Namely, from the equations \eqref{eqn:meinteq} and \eqref{eqn:m0inteq} we have 
\begin{align}\label{eq:meinteqt}
    m_\e(\lambda\pm0\ii)=(1-T_{\lambda\pm0\ii})^{-1}\e(\lambda)
\end{align}
and
\begin{align}\label{eq:m0inteqt}
    m_0(k)=(1-T_k)^{-1}G_k\ast q \,.
\end{align}

For sufficiently large $\abs{k}$ we will verify that $1-T_k$ is invertible as an $L^\infty\rightarrow L^\infty$ operator and $(1-T_k)^{-1}$ can be expanded into a Neumann series. Using the notation $a_\pm:=\max\{\pm a,0\}$ for $a\in\R$, when $|k|$ is large, at least one of $|\im k| + (\re k)_-$ and $|\im k| + (\re k)_+$ is large and the Neumann series will take different forms in the two cases. In the latter case we also need the operators
\begin{equation}
    \label{eqn:Tksplit}
    (T^-_k\phi)(x) \coloneqq \ii\int_{\mp\infty}^x \e^{\ii k(x-y)}qC_{-}\overline{q}\phi \dd y
\end{equation}
with $C_{-}\coloneqq 1-C_+$, and
\begin{align}
    \label{eq:sktilde}
    (\widetilde{S}_k\phi)(x) \coloneqq \ii\int_{\mp\infty}^x \e^{\ii k(x-y)}\e^{\ii\int_y^x\abs{q(t)}^2\dd t}\abs{q(y)}^2\phi(y)\dd y \,.
\end{align}
Clearly,
\begin{equation}
    \label{eq:sktildenorm}
    \|\widetilde{S}_k\|_{L^\infty\rightarrow L^\infty}\leq \norm{q}_2^2 \,.
\end{equation}

\begin{lemma}\label{lem:Tkinvert}
    Let $q\in L^2_{+}(\R)$. Then, there exists $\kappa_0>0$ such that $1-T_k$ is invertible on $L^\infty(\R)$ for all $k\in\Chat$ with $\abs{k}>\kappa_0$. Moreover: 
    \begin{itemize}
        \item[(a)]
        As $|\im k| + (\re k)_-\to\infty$, 
        \begin{align*}
            \|T_k\|_{L^\infty\rightarrow L^\infty}= o(1)
        \end{align*}
        and, in particular, if $|\im k| + (\re k)_-$ is sufficiently large, then
        \begin{align}\label{eqn:NeuexpStolz}
            (1-T_k)^{-1}=\sum_{n=0}^\infty T_k^n \,.
        \end{align}
        \item[(b)] 
        As $|\im k| + (\re k)_+\to\infty$, 
        \begin{align*}
            \|T^-_k\|_{L^\infty\rightarrow L^\infty}=o(1)
        \end{align*}
        and, in particular, if $|\im k| + (\re k)_+$ is sufficently large, then
        \begin{align}\label{eqn:Neuexpreal}
            (1-T_k)^{-1}=\sum_{n=0}^\infty(-(1+\widetilde{S}_k)T^-_k)^n(1+\widetilde{S}_k) \,.
        \end{align}
    \end{itemize}
\end{lemma}

The following proof uses some ideas from \cite{wu_jost_2017}.

\begin{proof}
    \emph{Step 1.}
    Let $\epsilon>0$. We show that there is a $\kappa_\epsilon<\infty$ such that if $k\in\C$ satisfies $\min\{|\im k|,-\re k\} \geq \kappa_\epsilon$, then
    \begin{align}\label{eq:tkinvertproof}
        \|T_k\|_{L^\infty\rightarrow L^\infty}\leq \varepsilon.
    \end{align}
    In particular, choosing $\epsilon<1$ we see that $1-T_k$ is invertible and its inverse is given by a convergent Neumann series.
    
    Thus, fix $\varepsilon>0$ and suppose $\im k\geq 0$, the case $\im k\leq 0$ being similar. Then
    \begin{align*}
        \abs{T_k\phi(x)}= \abs{ \int_{-\infty}^{x}\e^{\ii k (x-y)}q(y)[C_+\ov{q}\phi](y)\dd y}= \abs{ \int_0^{\infty}\e^{\ii k y}q(x-y)[C_+\ov{q}\phi](x-y)\dd y}.
    \end{align*}
    Since $q\in L^2$, there is a $\delta>0$ such that
    $$
    \sup_x\left(\int_{x-\delta}^{x}\abs{q(y)}^2\dd y\right)^{1/2} \leq \frac{\epsilon}{2} \|q\|_2^{-1} \,.
    $$
    We then split the integral as 
    \begin{align*}
        \abs{T_k\phi(x)}&\leq   \int_0^{\delta}\abs{q(x-y)[C_+\ov{q}\phi](x-y)}\dd y+  \int_\delta^{\infty} \e^{- \im k y} |q(x-y)[C_+\ov{q}\phi](x-y)| \dd y \\
        &\leq \left( \int_0^\delta |q(x-y)|^2 \dd y \right)^{1/2} \|q \phi\|_2
        + \e^{- (\im k) \delta} \|q\|_2 \|q \phi\|_2 \\
        & \leq \frac{\varepsilon}{2}\norm{\phi}_{\infty}+\e^{- (\im k) \delta} \|q\|_2^2 \|\phi\|_\infty \,.
    \end{align*}
    Hence, we can choose $\im k$ large enough so that \eqref{eq:tkinvertproof} holds.
        
    Next, suppose $\re k<0$. Then we estimate 
      \begin{align*}
        \abs{T_k\phi(x)}&\leq 
        \abs{\int_{0}^\infty \overline{\F\left(\ov{\e^{\ii k(x-\cdot )}q}\1_{<x}\right)(\xi)}\F(\overline{q}\phi)(\xi)\dd \xi}\\
        &\leq \norm{\overline{q}\phi}_2 \left( \frac1{2\pi} \int_{0}^\infty\abs{\int^x_{-\infty}\ov{\e^{\ii k(x-y)}q(y)}\e^{-\ii \xi y}\dd y}^2\dd \xi\right)^{1/2}\\
        &=\norm{\overline{q}\phi}_2 \left( \frac1{2\pi} \int_{-\re k}^\infty\abs{\int^x_{-\infty}\e^{-\im k(x-y)}q(y)\e^{\ii \xi y}\dd y}^2\dd \xi\right)^{1/2}.
    \end{align*}
    It follows from Lemma \ref{riemannlebesgue} below (with $\sigma =\im k$ and $\Xi=-\re k$) that the second term tends to zero as $\re k\to-\infty$, uniformly in $\im k\geq 0$ and $x\in\R$. Thus, bounding $\norm{\overline{q}\phi}_2\leq \norm{q}_2\norm{\phi}_\infty$, we arrive again at \eqref{eq:tkinvertproof} for $-\re k$ sufficiently large.

    \medskip

    \emph{Step 2.} We now turn to the second claim. We split $T_k=S_k-T^-_k$ with $T^-_k$ as defined in \eqref{eqn:Tksplit} and with
    \begin{align*}
    (S_k\phi)(x) \coloneqq \ii\int_{\mp\infty}^x \e^{\ii k(x-y)}\abs{q(y)}^2 \phi(y) \dd y \,.
    \end{align*}
    Arguing as in \cite[Eq.~(6.8)]{wu_jost_2017} we find that $1-S_k$ is invertible with
    $$
    (1-S_k)^{-1} = 1+\widetilde{S}_k \,.
    $$

    We aim to show that for any $\varepsilon>0$ there is a $\kappa_\epsilon<\infty$ such that if $k\in\Chat$ satisfies $\min\{|\im k|, \re k\} \geq\kappa_\epsilon$, then
    \begin{align}
        \label{eq:tktildedecay}
        \|T^-_k\|_{L^\infty\rightarrow L^\infty} \leq \varepsilon.
    \end{align}

    Once we have shown this, recalling also \eqref{eq:sktildenorm}, we obtain the invertibility of $(1-T_k)^{-1}$ via
    \begin{align*}
        (1-T_k)^{-1}=(1-S_k+T^-_k)^{-1} =(1+(1-S_k)^{-1}T^-_k)^{-1}(1-S_k)^{-1}
    \end{align*}
    and, at the same time, we obtain a convergent Neumann expansion for $(1-T_k)^{-1}$ by expanding $(1+(1-S_k)^{-1} T^-_k)^{-1} = (1+(1+\widetilde{S}_k)T^-_k)^{-1}$ into a Neumann series.

    To prove \eqref{eq:tktildedecay}, we fix $\varepsilon>0$ and suppose $\im k\geq 0$, the case $\im k\leq 0$ being similar. Then repeating the argument from Step 1 above we can show that \eqref{eq:tktildedecay} holds when $\im k$ is large enough. Next, suppose $\re k>0$. Then we estimate
    \begin{align*}
        |T^-_k\phi(x)|&\leq 
        \abs{\int_{-\infty}^0 \overline{\F\left(\ov{\e^{\ii k(x-\cdot )}q}\1_{<x}\right)(\xi)}\F(\overline{q}\phi)(\xi)\dd \xi}\\
        &\leq \norm{\overline{q}\phi}_2 \left( \frac1{2\pi} \int_{-\infty}^0 \abs{\int^x_{-\infty} \overline{\e^{\ii k (x-y)}q(y)} \e^{-\ii \xi y}\dd y}^2\dd \xi\right)^{1/2}\\
        &\leq \norm{\overline{q}\phi}_2\left( \frac1{2\pi} \int_{\re k}^\infty \abs{\int^x_{-\infty}\e^{- \im k (x-y)}q(y)\e^{-\ii \xi y}\dd y}^2\dd \xi\right)^{1/2}.
    \end{align*}
    As before, by Lemma \ref{riemannlebesgue} we arrive at \eqref{eq:tktildedecay}. This concludes the proof.
\end{proof}

We note that, a priori, the convergence of both Neumann series in Lemma \ref{lem:Tkinvert} may be arbitrarily slow as $\abs{k} \to \infty$, in contrast to the Benjamin--Ono case. Nevertheless, we can establish the asymptotics in Theorem \ref{thm:simpasymp} without imposing any additional regularity on $q$.

\begin{lemma}\label{riemannlebesgue}
    If $f\in L^1(\R)$, then
    $$
    \lim_{|\xi|\to\infty} \sup_{x\in\R, \, \sigma\in\overline{\R_+}} \left| \int_{-\infty}^x \e^{\ii\xi(x-y)} \e^{-\sigma(x-y)} f(y)\dd y \right| = 0 \,.
    $$
    If $f\in L^2(\R)$, then
    $$
    \lim_{\Xi\to\infty} \sup_{x\in\R,\, \sigma\in\overline{\R_+}} \int_\Xi^\infty \left| \int_{-\infty}^x \e^{\ii\xi(x-y)} \e^{-\sigma(x-y)} f(y)\dd y \right|^2 \dd \xi =0 \,.
    $$
\end{lemma}

\begin{proof}
    We first prove the assertion under the assumption that $f\in C^1_c(\R)$. In this case we can integrate by parts and find
    $$
    \int_{-\infty}^x \e^{\ii\xi(x-y)} \e^{-\sigma(x-y)} f(y)\dd y
    = \frac{1}{-\ii \xi+\sigma} f(x) - \frac{1}{-\ii\xi+\sigma} \int_{-\infty}^x \e^{\ii\xi(x-y)} \e^{-\sigma(x-y)} f'(y)\dd y \,.
    $$
    Thus,
    $$
    \left| \int_{-\infty}^x \e^{\ii\xi(x-y)} \e^{-\sigma(x-y)} f(y)\dd y \right| \leq \frac1{\sqrt{\xi^2+\sigma^2}} \left( \|f\|_\infty + \|f'\|_1 \right) \leq \frac1{|\xi|} \left( \|f\|_\infty + \|f'\|_1 \right),
    $$
    from which both assertions follow easily.

    Now for general $f\in L^p(\R)$, $p=1,2$, and $\epsilon>0$ we choose $\tilde f\in C^1_c(\R)$ with $\|\tilde f - f\|_p \leq \epsilon$. We bound
    $$
    \left| \int_{-\infty}^x \e^{\ii\xi(x-y)} \e^{-\sigma(x-y)} (\tilde f(y)-f(y)) \dd y \right| \leq \|\tilde f - f\|_1 \leq \epsilon \,,
    $$
    so the first assertion for $f$ follows from that for $\tilde f$. Meanwhile, by Plancherel's theorem
    \begin{align*}
        & \int_\Xi^\infty \left| \int_{-\infty}^x \e^{\ii\xi(x-y)} \e^{-\sigma(x-y)} (\tilde f(y) - f(y)) \dd y \right|^2 \dd \xi \\
        & \quad \leq \int_\R \left| \int_{-\infty}^x \e^{\ii\xi(x-y)} \e^{-\sigma(x-y)} (\tilde f(y) - f(y)) \dd y \right|^2 \dd \xi \\
        & \quad = 2\pi \int_{-\infty}^x \left| \e^{-\sigma(x-y)} (\tilde f(y)- f(y)) \right|^2 \dd y \\
        & \quad \leq 2\pi \| \tilde f - f\|_2^2 \leq 2\pi\epsilon^2 \,.
    \end{align*}
    Thus, once again the assertion for $f$ follows from that for $\tilde f$.
\end{proof}

\begin{proof}[Proof of Theorem \ref{thm:simpasymp}]
    We begin with the proof of \eqref{eqn:simplemeasmp}. Applying the Neumann expansion \eqref{eqn:Neuexpreal} to the integral equation \eqref{eq:meinteqt} gives
    \begin{align*}
        m_\e(\lambda+0\ii)=(1+\widetilde{S}_{\lambda+0\ii})\e(\lambda)+o_{\lambda\rightarrow \infty}(1)
    \end{align*}
    with the error term in $L^\infty$. Using \eqref{eq:sktilde} we can compute the leading term,
    \begin{align}
        \label{eq:tildeskone}
        (1+\widetilde{S}_{\lambda\pm0\ii})\e(\lambda)&=\e^{\ii \lambda x}+\ii\e^{\ii \lambda x}\int_{\mp \infty}^x \e^{\ii\int_y^x\abs{q(t)}^2\dd t}\abs{q(y)}^2\dd y \notag \\
        &=\e^{\ii \lambda x}-\e^{\ii \lambda x}\e^{\ii\int_{-\infty}^x\abs{q(t)}^2\dd t}\int_{\mp \infty}^x \partial_y\left(\e^{-\ii\int_{-\infty}^y\abs{q(t)}^2\dd t}\right)\dd y \notag\\
        &=\e^{\ii\lambda x}\e^{\ii \int_{\mp\infty}^x\abs{q(t)}^2\dd t},
    \end{align}
    which proves the asymptotics \eqref{eqn:simplemeasmp}. 
    
    To establish \eqref{eqn:simplegammaasmp}, we write 
    \begin{align*}
        \Gamma(\lambda)&=1+\ii \int_{-\infty}^\infty \e^{-\ii \lambda x}qC_+\overline{q}m_\e(\lambda+0\ii)\dd x\\
        &=1+\ii \int_{-\infty}^\infty \e^{-\ii \lambda x}\abs{q}^2m_\e(\lambda+0\ii)\dd x-\ii \int_{-\infty}^\infty \e^{-\ii \lambda x}qC_{-}\overline{q}m_\e(\lambda+0\ii)\dd x.
    \end{align*}
    For the first term, we substitute the asymptotics for $m_\e$ from \eqref{eqn:simplemeasmp}, giving
    \begin{align*}
        1+\ii \int_{-\infty}^\infty \e^{-\ii \lambda x}\abs{q}^2m_\e(\lambda+0\ii)\dd x&=1+\ii\int_{-\infty}^\infty \abs{q(x)}^2\e^{\ii\int_{-\infty}^x\abs{q(t)}^2\dd t}\dd x+o_{\lambda\rightarrow \infty}(1)\\
        &=1+ \int_{-\infty}^\infty\partial_x\left(\e^{\ii\int_{-\infty}^x\abs{q(t)}^2\dd t}\right)\dd x+o_{\lambda\rightarrow \infty}(1)\\
        &=\e^{\ii \int_\R \abs{q(t)}^2\dd t}+o_{\lambda\rightarrow \infty}(1).
    \end{align*}
    For the second term, we estimate, using Lemma \ref{lem:Tkinvert},
    \begin{align*}
        \abs{\ii \int_{-\infty}^\infty \e^{-\ii \lambda x}qC_{-}\overline{q}m_\e(\lambda+0\ii)\dd x}\leq \|T^-_{\lambda+0\ii}m_\e(\lambda+0\ii)\|_{\infty}=o_{\lambda\rightarrow \infty}(1)
    \end{align*}
    which completes the proof of \eqref{eqn:simplegammaasmp}.
    
    In the remainder of the proof we assume, in addition, $q\in L^1$. To prove \eqref{eqn:simplem0asymp} we assume $\im k\geq 0$, the opposite case being similar. We claim that
    \begin{align}\label{eqn:Gkinftyzero}
        \lim_{\abs{k}\rightarrow\infty} \norm{G_k\ast q}_\infty=0 \,.
    \end{align}
    Indeed, when $\im k\to\infty$ this is easy (see the proof of Lemma \ref{lem:Tkinvert}) and when $|\re k|\to\infty$ this follows from the first part of Lemma \ref{riemannlebesgue} (with $\sigma = \im k$ and $\xi=\re k$).

    In view of the integral equation \eqref{eq:m0inteqt} for $m_0(k)$ we find
    $$
    \| m_0(k)\|_\infty \leq \|(1-T_k)^{-1}\|_{L^\infty\to L^\infty} \|G_k\ast q\|_\infty \,.
    $$
    According to Lemma \ref{lem:Tkinvert} $\|(1-T_k)^{-1}\|_{L^\infty\to L^\infty}$ is uniformly bounded as $|k|\to\infty$, so \eqref{eqn:simplem0asymp} follows from \eqref{eqn:Gkinftyzero}.
        
    Finally, to prove \eqref{eqn:simplebetaasmp}, we observe
    \begin{align*}
        \abs{\beta(\lambda)}=\abs{ \int_{-\infty}^\infty q(C_+\overline{q}m_0(\lambda+0\ii)+1)\e^{-\ii \lambda x}}\leq \norm{q}_2^2\norm{m_0(
        \lambda+0\ii)}_\infty+ \sqrt{2\pi}\,  \abs{\F(q)(\lambda)},
    \end{align*}
    and the claim follows from Riemann--Lebesgue and the asymptotics \eqref{eqn:simplem0asymp} for $m_0$. 
\end{proof}

Next, we derive sharper asymptotics for $m_0$ in any fixed direction off the real axis, under the assumption that $q$ is continuous.
\begin{lemma}\label{lem:m0asmpqcts}
    Let $q\in L^2_+(\R)\cap L^1(\R)$ be continuous. Then, for any $k\in \C\backslash \R$, 
    \begin{align}\label{eqn:qreconstruct2}
        \lim_{r\rightarrow \infty} \norm{rk \, m_0(r k)+ q}_\infty=0.
    \end{align}
\end{lemma}

\begin{proof}
    For any $k\in \C\backslash\R$ and $r>0$ sufficiently large, from the Neumann expansion \eqref{eqn:NeuexpStolz}, we have 
    \begin{align}\label{eqn:m0asymprkexpan}
        m_0(rk)=G_{rk}\ast q+\sum_{n=1}^\infty T_{rk}^nG_{rk}\ast q.
    \end{align}
    Then we use that 
    \begin{align*}
        K_r(x)\coloneqq -rk G_k(rx)=-rk G_{rk}(x)
    \end{align*}
    defines an approximation of the identity, see \eqref{eqn:Gk}. In particular, we have pointwise convergence 
    \begin{align*}
        \lim_{r\rightarrow \infty}\left[K_r\ast q\right](x)= q(x).
    \end{align*}
    
    Since $q$ is continuous, standard results imply convergence in $L^\infty(\R)$. That is
    \begin{align*}
        \lim_{r\rightarrow \infty}\norm{-rkG_{rk}\ast q-q}_\infty=0,
    \end{align*}
    which, together with the Neumann expansion \eqref{eqn:m0asymprkexpan}, gives the desired result. 
\end{proof}


\subsection*{Complete asymptotic expansion}
Finally, we derive a full asymptotic expansion of the Jost functions and scattering coefficients under the assumption that the function $q$ is Schwartz. We recall that
$$
\mathcal{S}_+(\R) := \mathcal S(\R)\cap L^2_+(\R) \,.
$$
This extends the approach of \cite{wu_jost_2017} for Benjamin--Ono by providing a complete expansion for $m_0$, which will be essential for the proof of higher-order trace formulas in Section~\ref{sec:tracehigher}. While truncated expansions can also be obtained under weaker regularity assumptions, we do not pursue that here.

\begin{theorem}\label{thm:m0fullexpan}
    Suppose that $q\in \mathcal{S}_+(\R)$ and define the sequence $\{\mn_n\}^\infty_{n=1}$ recursively by
    \begin{align*}
        \mn_{n+1} :=(-\ii \partial_x-qC_+\overline{q})\mn_n, \qquad\mn_1:=-q.
    \end{align*}
    Then $c_n\in\mathcal S_+(\R)$ for all $n\in\N$, and for every $M\in \N$ there exist constants $\kappa_M>0$ and $C_M>0$ such that for all $k\in\Chat$ with $\abs{k}\geq \kappa_M$
    \begin{align}\label{eqn:m0fullasymp}
        \norm{m_0(x,k)-\sum_{n=1}^{M-1}\frac{\mn_n(x)}{k^n}}_\infty\leq \frac{C_M}{\abs{k}^{M}}
    \end{align}
    and for all $\lambda\in\R_+$ with $\lambda\geq \kappa_M$
    \begin{align}
        \norm{m_\e(x,\lambda\pm0\ii)-\e^{\ii \lambda x}\e^{\ii \int_{\mp\infty}^x\abs{q(t)}^2\dd t}}_\infty&\leq \frac{C_M}{\lambda^M} \,,\label{eqn:mefullasymp}\\
        \abs{\Gamma(\lambda)-\e^{\ii\int_{\R}\abs{q(t)}^2\dd t}}&\leq \frac{C_M}{\lambda^M} \,,\label{eqn:gamamfullasymp}\\
        \abs{\beta(\lambda)}&\leq \frac{C_M}{\lambda^M} \,. \label{eqn:betafullasymp}
    \end{align}
\end{theorem}

\begin{proof}
     \emph{Step 1.}
     Given that $q\in\mathcal{S}_+(\R)$, it follows by induction that each $\mn_n$ is well defined and belongs to $\mathcal{S}_+(\R)$. This uses the fact that $C_+$ maps $\mathcal{S}(\R)$ to $C^\infty(\R)\cap L^2_+(\R)$, so that $\pot \mn_n \in \mathcal{S}(\R)$. 

    \medskip

     \emph{Step 2.}
    Fix $M\in \N$ and, for $k\in \Chat\setminus\{0\}$, define the remainder term
    \begin{align*}
        r_M(x,k)\coloneqq m_0(x,k)-\sum_{n=1}^M \frac{\mn_n(x)}{k^n}.
    \end{align*}
    Using the differential equation for $m_0$, we compute 
    \begin{align*}
        (-\ii\partial_x-k-\pot) r_M(k)&= (-\ii\partial_x-k-\pot)m_0(k)-\sum_{n=1}^M\frac{1}{k^n}(-\ii\partial_x-k-\pot)\mn_n\\
        &=q+\sum_{n=1}^M\frac{\mn_n}{k^{n-1}}-\sum_{n=1}^M\frac{(-\ii\partial_x-\pot)\mn_n}{k^n}\\
        &=q+\sum_{n=1}^M\frac{\mn_n}{k^{n-1}}-\sum_{n=1}^M\frac{\mn_{n+1}}{k^n}=-\frac{\mn_{M+1}}{k^M},
    \end{align*}
    using the recurrence relation defining $\mn_n$. Let us assume $k\in\Chat\setminus\sigma_\mathrm{p}(\CMop)^\diamond$. Since $r_M(k)$ belongs to $L^\infty(\R)$ and decays rapidly at infinity and since $\mn_{M+1}\in L^1_+(\R)$, it follows from Lemma \ref{lem:pertresolvent} that
    $$
    r_M(k) = - \frac1{k^M}\, R(k) c_{M+1}  
    $$    
    and therefore, by the resolvent identity \eqref{eq:resolventformula},
    \begin{align*}
        r_M(k)=-\frac{1}{k^M}G_{k}\ast \mn_{M+1}+R_0(k) q C_+ \ov{q} r_M(k) \,.
    \end{align*}
    In terms of the operator $T_k$, this can be written as
    \begin{align*}
        r_M(k)=-\frac{1}{k^{M}}(1-T_k)^{-1}(G_{k}\ast \mn_{M+1}).
    \end{align*}
    By Theorem \ref{thm:simpasymp} the assumption $k\in\Chat\setminus\sigma_\mathrm{p}(\CMop)^\diamond$ is satisfied for all sufficiently large $|k|$ and, repeating the argument for proving \eqref{eqn:simplem0asymp}, it follows that  
    \begin{align*}
        \norm{r_M(k)}_\infty=o_{\abs{k}\rightarrow\infty}(\abs{k}^{-M}) \,.
    \end{align*}
    This gives the desired asymptotic expansion \eqref{eqn:m0fullasymp} for $m_0$.
    
    \medskip

     \emph{Step 3.}
    To obtain the asymptotic behaviour of $\beta(\lambda)$, we substitute the expansion \eqref{eqn:m0fullasymp} for $m_0$ into its definition. We find 
    \begin{align*}
        \abs{\beta(\lambda)}&\leq \sqrt{2\pi}\, \abs{\widehat{q}(\lambda)}+\sum_{n=1}^{M-1} \frac1{\lambda^n} \abs{\int_\R \e^{-\ii\lambda y} \pot \mn_n\dd y}+\mathcal{O}_{\lambda\rightarrow\infty}(\lambda^{-M}).
    \end{align*}
    Using the recurrence relation again, this yields
    \begin{align*}
        \frac{\abs{\beta(\lambda)}}{\sqrt{2\pi}}&\leq \, \abs{\widehat{q}(\lambda)} +  \sum_{n=1}^{M-1} \frac1{\lambda^n} \left( \lambda \abs{\widehat{ \mn_n}(\lambda)}+\abs{\widehat{\mn_{n+1}}(\lambda)}\right)+\mathcal{O}_{\lambda\rightarrow\infty}(\lambda^{-M})=\mathcal{O}_{\lambda\rightarrow\infty}(\lambda^{-M}),
    \end{align*}
    since each $c_n$ belongs to $\mathcal{S}(\R)$. This establishes \eqref{eqn:betafullasymp}.   

    \medskip

     \emph{Step 4.}
    To obtain the asymptotic behaviour \eqref{eqn:mefullasymp} of $m_\e(\lambda\pm0\ii)$, we use the Neumann series expansion \eqref{eqn:Neuexpreal},
    \begin{align*}
        m_\e(\lambda\pm0\ii)- (1+\widetilde{S}_{\lambda\pm0\ii}) \e(\lambda) =\sum_{n=1}^\infty (-(1+\widetilde{S}_{\lambda\pm0\ii}) T^-_{\lambda\pm0\ii})^n (1+\widetilde{S}_{\lambda\pm0\ii})\e(\lambda),
    \end{align*}
    and we recall from \eqref{eq:tildeskone} that 
    \begin{align*}
        (1+\widetilde{S}_{\lambda\pm0\ii})\e(\lambda)=\e^{\ii\lambda x}\e^{\ii \int_{\mp\infty}^x\abs{q(t)}^2\dd t}.
    \end{align*}
    To estimate the first iterate, we bound
    \begin{align*}
        |T^-_{\lambda\pm0\ii}\phi(x)|&= \abs{\int_{\mp\infty}^x \e^{-\ii \lambda y}qC_{-}\overline{q}\phi\dd y}
        \leq\norm{q}_2 \left(\int_{-\infty}^\infty \abs{C_{-}\overline{q}\phi}^2\dd y\right)^{1/2}\\
        &\leq\norm{q}_2 \left( \frac1{2\pi} \int_{-\infty}^0\abs{\int^\infty_{-\infty}\e^{-\ii y \xi}\overline{q(y)}\phi(y)\dd y}^2\dd \xi\right)^{1/2}.
    \end{align*}
    Applying this to $\phi=(1+\widetilde{S}_{\lambda\pm0\ii})\e(\lambda)$, we obtain
    \begin{align*}
        \|T^-_{\lambda\pm0\ii}(1+\widetilde{S}_{\lambda\pm0\ii})\e(\lambda)\|&\leq \norm{q}_2 \left( \frac1{2\pi} \int_{-\infty}^0\abs{\int^\infty_{-\infty}\e^{\ii y(\lambda- \xi)}\e^{\ii\int_{\mp\infty}^y\abs{q(t)}^2\dd t}\overline{q(y)}\dd y}^2\dd \xi\right)^{1/2}\\
        &=\norm{q}_2 \left(\int_{\lambda}^\infty \abs{\F(\e^{-\ii\int_{\mp\infty}^y\abs{q(t)}^2\dd t} q)(\xi)}^2\dd \xi\right)^{1/2}.
    \end{align*}
    Since $q(x) \e^{-\ii \int_{\pm\infty}^x\abs{q(t)}^2\dd t}$ is a Schwartz function, the tail integral decays rapidly as $\lambda\rightarrow\infty$ and we have
    $$
    \|T^-_{\lambda\pm0\ii}(1+\widetilde{S}_{\lambda\pm0\ii})\e(\lambda)\| \leq C_M \lambda^{-M}
    $$
    for any $M$. This bound, together with
    \begin{align*}
    \| m_\e(\lambda+0\ii) - (1+\widetilde{S}_{\lambda\pm0\ii})e(\lambda)\|_\infty
    & \leq \sum_{n=1}^\infty \| (1+\widetilde{S}_{\lambda\pm0\ii}) T^-_{\lambda\pm0\ii} \|_{L^\infty\to L^\infty} ^{n-1} \\
    & \quad\times \| 1+\widetilde{S}_{\lambda\pm0\ii} \|_{L^\infty\to L^\infty} \|T^-_{\lambda\pm0\ii}(1+\widetilde{S}_{\lambda\pm0\ii})\e(\lambda)\|_\infty
    \end{align*}
    and the norm bounds for $\widetilde S_{\lambda\pm0\ii}$ in \eqref{eq:sktildenorm} and for $T^-_{\lambda\pm0\ii}$ in Lemma \ref{lem:Tkinvert}, establish the asymptotic behaviour of $m_\e(\lambda\pm0\ii)$ in \eqref{eqn:mefullasymp}. 

    \medskip

     \emph{Step 5.}
    Finally, the asymptotics \eqref{eqn:gamamfullasymp} for $\Gamma(\lambda)$ follow by substituting the expansion \eqref{eqn:mefullasymp} for $m_\e(\lambda+0\ii)$ into the definition of $\Gamma(\lambda)$ and using the first-order term identified in Theorem \ref{thm:simpasymp}.
\end{proof}

We note that in the special case of the one-soliton potential $q=q_\eta$, discussed at the end of Section \ref{sec:jostinhom}, the asymptotics in Theorem \ref{thm:m0fullexpan} are in agreement with the explicit expressions for the Jost functions and scattering coefficients found there.


\section{Expansion about eigenvalues}\label{sec:expansionevs}

In this section, we analyse the behaviour of the Jost functions $m_0$ and $m_\e$ near the eigenvalues of $\CMop$. We emphasise that the material in this section is not needed for the proof of the trace formulas in Sections \ref{sec:traceform} and \ref{sec:tracehigher}, but it is needed to introduce the full set of scattering data, considered in Section \ref{sec:timeevolscat}.

Our main result is an expansion of $m_0(k)$ around $k=\lambda_j\pm0\ii$. Note that, $m_0(k)$ is not defined for $k\in \sigma_{\mathrm{p}}(\CMop)^\diamond$. We show that it has simple poles at the negative eigenvalues of $\CMop$. Remarkably we can show that the same behaviour extends to the embedded eigenvalues.

We normalise the eigenfunctions $\phi_j$ according to
\begin{align}\label{eqn:eigfncnormal}
    \sclp{q,\phi_j}=2\pi\ii \,.
\end{align}
It follows from Lemma \ref{lem:qeigenfuncinprod} that this is possible and that
\begin{align*}
    \frac{\sclp{q,\phi_j}}{\norm{\phi_j}_2^2}=\ii.
\end{align*}

\begin{theorem}\label{thm:embeddexpan}
    Let $q\in L^2_+(\R)\cap L^1_+(\R)$. Then, the following hold:
    
    \begin{itemize}
    \item[(a)] There is a unique map $h\in C(\Chat;(L^\infty(\R),w^\ast))$, such that 
    \begin{align*}
        m_0(k)=-\ii\sum_{j=1}^N\frac{ \phi_j}{k-\lambda_j}+h(k)
        \qquad\text{for all}\ k\in\Chat \,,
    \end{align*}
    where $\{\lambda_j\}_{j=1}^N$ are the eigenvalues of $\CMop$ and $\{\phi_j\}_{j=1}^N$ are the corresponding eigenfunctions, normalised according to \eqref{eqn:eigfncnormal}.
     \item[(b)] The maps $\lambda\mapsto \ov{\e(\lambda)}m_\e(\lambda\pm0\ii)$ have unique extensions in $C(\R_+;L^\infty(\R))$. {Moreover},
    \begin{align*}
        \norm{\ov{\e(\lambda)}m_\e(\lambda\pm0\ii)-1}_{L^\infty(\R)} = o_{\lambda\rightarrow 0_+}(1).
    \end{align*}    
    \end{itemize}
\end{theorem}

In this theorem $(L^\infty(\R),w^\ast)$ denotes the space $L^\infty(\R)=(L^1(\R))^*$ equipped with the weak-$*$ topology. The assertion $h\in C(\Chat;(L^\infty(\R),w^\ast))$ simply means that $k\mapsto\langle h(k),f\rangle$ is continuous on $\Chat$ for every $f\in L^1(\R)$ (or, in the situation of Theorem \ref{thm:embeddexpan}, every $f\in L^1_+(\R)$). We note that this form of continuity is used in part (a), but \emph{not} in part (b).

One might wonder about the $L^1$ assumption of Theorem \ref{thm:embeddexpan}. We note that in the one-soliton example at the end of Section \ref{sec:jostinhom} the functions $\lambda\mapsto\ov{\e(\lambda)}m_\e(\lambda\pm0\ii)$ are \emph{not} continuous in $L^\infty(\R)$ at the eigenvalue $\eta$. This shows that Theorem \ref{thm:embeddexpan} is not valid for all $q\in L^2_+$. We also not that the $L^1$-assumption just barely fails in the one-soliton example, which indicates that this assumption is rather optimal.

\begin{proof}
    Eigenfunctions belong to the operator domain $H^1_+(\R)$ \cite{killip_scaling_2025}, which is contained in $L^\infty(\R)$. This, together with Lemma \ref{efl1}, gives  
    \begin{align}\label{eqn:assump2}
        P_{\mathrm{p}}(\CMop)L^2_+(\R)\subset L^1_+(\R)\cap L^\infty(\R).
    \end{align}
    Therefore, the projections
    \begin{align*}
        P_jf\coloneqq\frac{\sclp{f,\phi_j}}{\norm{\phi_j}^2_2} \, \phi_j
    \end{align*}
    define bounded operators on $L^\infty(\R)$. Set $Q_j\coloneqq 1-P_j$ and write
    \begin{align}
        \label{eq:embeddedexpanproof0a}
        m_0(k)&=P_jm_0(k)+Q_jm_0(k)
        \intertext{and} 
        \label{eq:embeddedexpanproof0b}
        \ov{\e(\lambda)}m_\e(\lambda\pm 0\ii)&=\ov{\e(\lambda)}P_jm_\e(\lambda\pm 0\ii)+\ov{\e(\lambda)}Q_jm_\e(\lambda\pm 0\ii) \,.
    \end{align}
    The index $j$ is fixed and, in the setting of part (b) we assume $\lambda_j\geq 0$.
    
    We shall show in Step 1 that
    \begin{equation}
        \label{eq:embeddedexpanproof1}
        P_jm_0(k) = -\ii \frac{\phi_j}{k-\lambda_j}
        \qquad\text{and}\qquad
        P_j m_\e(\lambda\pm0\ii) = 0
    \end{equation}
    in a neighbourhood of $\lambda_j\pm0\ii$ in $\Chat$ and of $\lambda_j$ in $[0,\infty)$, respectively. In Step 2 we show that
    \begin{equation}
        \label{eq:embeddedexpanproof2}
        Q_jm_0(k) \ \text{is}\ (L^\infty,w^*)\text{-continuous}
        \qquad\text{and}\qquad
        \ov{\e(\lambda)}Q_j m_\e(\lambda\pm0\ii)
        \ \text{is}\ L^\infty\text{-continuous}
    \end{equation}
    in a neighbourhood of $\lambda_j\pm0\ii$ in $\Chat$ and of $\lambda_j$ in $[0,\infty)$, respectively.

    Once these two properties are shown, we immediately obtain the assertions of the theorem by \eqref{eq:embeddedexpanproof0a} and \eqref{eq:embeddedexpanproof0b}, except for the asymptotics of $\ov{\e(\lambda)}m_\e(\lambda\pm0\ii)$ as $\lambda\to 0_+$, which are dealt with in Step 3.
    
\medskip

\emph{Step 1.} We shall show that 
\begin{align}\label{eqn:m0meinnerphij}
    \quad \sclp{m_0(k),\phi_j}=-\frac{\sclp{q,\phi_j}}{k-\lambda_j}  \quad \text{and}\quad \sclp{m_\e(\lambda\pm0\ii),\phi_j}=0
\end{align}
for all $k\in\Chat\setminus\sigma_\mathrm{p}(\CMop)^\diamond$ and $\lambda\in[0,\infty)\setminus\sigma_\mathrm{p}(\CMop)$, respectively. (We mention that a more intricate argument shows that this holds for all $k\in \Chat\setminus\{\lambda_j\}^\diamond$ and $\lambda\in \R_+\setminus\{\lambda_j\}$, but we will not need this.) Note that \eqref{eqn:m0meinnerphij} proves \eqref{eq:embeddedexpanproof1} in view of the normalisation \eqref{eqn:eigfncnormal}.

For the proof of the identities in \eqref{eqn:m0meinnerphij}, we use the formula
\begin{align}\label{eqn:Glambdaphij}
    \phi_j-R_0(k)\pot \phi_j=(\lambda_j-k)G_{k}\ast \phi_j,
\end{align}
for any $k\in\Chat$, which follows by the same reasoning as Lemma \ref{lem:pertresolvent} since $\phi_j\in L^2_+(\R)\cap L^1_+(\R)$ and
\begin{align}\label{eqn:Glambdaphij2}
    (-\ii\partial_x-k)\phi_j=(\lambda_j-k)\phi_j+\pot \phi_j.
\end{align}

\medskip

\emph{Step 1a.} 
Let us prove the first identity in \eqref{eqn:m0meinnerphij}. In the case with $k\in \C\backslash\sigma(\CMop)$, this follows immediately since $m_0(k)=R(k)q$, and therefore
\begin{align*}
    \sclp{m_0(k),\phi_j}=\langle q,R(\ov{k})\phi_j\rangle=-\frac{\sclp{q,\phi_j}}{k-\lambda_j}.
\end{align*}
Using the $(L^\infty,w^*)$-continuity of $m_0(k)$, which we have observed right after its definition, we obtain the identity for all $k\in\C\setminus\sigma_\mathrm{p}(\CMop)^\diamond$, as claimed.

\medskip

\emph{Step 1b.} 
Next, we prove the second identity in \eqref{eqn:m0meinnerphij} for $\lambda\in[0,\infty)\setminus\sigma_\mathrm{p}(\CMop)$. Using the definition of $m_\e$, we find
\begin{align*}
    \sclp{m_\e(\lambda\pm 0\ii),\phi_j} & = \sclp{\e(\lambda),\phi_j} + \sclp{R(\lambda\pm0\ii)qC_+\ov{q}\e(\lambda),\phi_j} \\
    & = \sclp{\e(\lambda),\phi_j} + \sclp{\e(\lambda),q C_+ \ov{q} R(\lambda\mp 0\ii)\phi_j} \\
    & = \sclp{\e(\lambda),\phi_j} + \frac{1}{\lambda_j - \lambda} \, \sclp{\e(\lambda),q C_+ \ov{q} \phi_j}.
\end{align*}
Because of \eqref{eqn:Glambdaphij2}, we have
\begin{equation}
    \label{eq:efeqev}
    \sclp{\e(\lambda),q C_+ \ov{q} \phi_j} = \sclp{\e(\lambda), -\ii \phi_j'- \lambda_j\phi_j} = (\lambda-\lambda_j) \sclp{\e(\lambda),\phi_j} \,.
\end{equation}
Combining these two equations, we arrive at the second identity in \eqref{eqn:m0meinnerphij}.

\medskip

\emph{Step 2.} We now turn to the proof of \eqref{eq:embeddedexpanproof2}. From the integral equation \eqref{eqn:m0inteq} for $m_0$, the normalisation of $\phi_j$, the identity \eqref{eqn:Glambdaphij} and the first part of \eqref{eqn:m0meinnerphij} we find  
\begin{align}
    \label{eq:embeddedexpanproof3a}
    Q_jm_0(k)
    &=Q_jR_0(k)Q_jq+Q_jR_0(k)\pot Q_jm_0(k)
\end{align}
for $k\in \Chat$. This gives
\begin{align}
    \label{eq:embeddedexpanproof3b}
    C_+\ov{q}Q_jm_0(k)
    &=C_+\ov{q}Q_jR_0(k)Q_jq+C_+\ov{q}Q_jR_0(k)\pot Q_jm_0(k).
\end{align}

Similarly, using integral equation \eqref{eqn:meinteq} for $m_\e$ and the second part of \eqref{eqn:m0meinnerphij} we find
\begin{align}
    \label{eq:embeddedexpanproof4a}
    Q_jm_\e(\lambda\pm0\ii)=Q_j\e(\lambda)+Q_jR_0(\lambda\pm0\ii)\pot Q_jm_\e(\lambda\pm0\ii)
\end{align}
for $\lambda\in \R_+$, which gives
\begin{align}
    \label{eq:embeddedexpanproof4b}
C_+\ov{q} Q_jm_\e(\lambda\pm0\ii)=C_+\ov{q} Q_j\e(\lambda) + C_+\ov{q} Q_jR_0(\lambda\pm0\ii)\pot Q_jm_\e(\lambda\pm0\ii) \,.
\end{align}

Identities \eqref{eq:embeddedexpanproof3b} and \eqref{eq:embeddedexpanproof4b} motivate the study of the operators $1-C_+\ov{q}Q_jR_0(k)qC_+$, acting on $L^2_+$.

\medskip

\emph{Step 2a.} Let us show that $k\mapsto C_+\ov{q}Q_jR_0(k)qC_+$ is a family of $L_+^2$-compact operators a $\Chat$-neighbourhood of $\lambda_j\pm0\ii$.

Indeed, since this is true for $C_+\ov{q}R_0(k)qC_+$ by Lemma \ref{lem:qRqcompact}, it suffices to prove it for $C_+\ov{q}P_jR_0(k)qC_+$. Note that for any $f\in L^2_+$ we have 
\begin{align*}
    \norm{\phi_j}^2_2 \, C_+\ov{q}P_jR_0(k)qC_+f = \sclp{R_0(k)qC_+f,\phi_j} C_+\ov{q}\phi_j = \sclp{f,C_+\ov{q}R_0(\ov{k})\phi_j} C_+\ov{q}\phi_j.
\end{align*}
Thus $C_+\ov{q}P_jR_0(k)qC_+$ is a rank one operator and therefore compact. The continuity in $L^2_+$-operator norm follows from the $L^2_+$-continuity of $k\mapsto C_+\ov{q}R_0(\ov{k})\phi_j$, which follows from the argument used in Step 2 of the proof of Lemma \ref{lem:qRqcompact}. (Indeed, since $\phi_j\in L^1_+$, $R_0(\ov{k_n})\phi_j$ converges pointwise and boundedly, and therefore, by dominated convergence, $qR_0(\ov{k_n})\phi_j$ converges in $L^2$.)

\medskip

\emph{Step 2b.} Let us show that $\ker(1-C_+\ov{q} Q_j R_0(\lambda_j\pm0\ii)qC_+)=\{0\}$.

Indeed, suppose that $g\in L^2_+(\R)$ is a solution to
\begin{align}\label{eqn:Qjcontrassump}
    (1-C_+\ov{q}Q_jR_0(\lambda_j\pm0\ii)qC_+)g=0.
\end{align} 
Then,
\begin{align*}
    (1-C_+\ov{q}R_0(\lambda_j\pm0\ii)qC_+) g&=-C_+\ov{q}P_jR_0(\lambda_j\pm0\ii)qC_+g\\
    &=-\norm{\phi_j}_2^{-2}\sclp{R_0(\lambda_j\pm0\ii)qC_+g,\phi_j}C_+\ov{q}\phi_j.
\end{align*}
Applying $1-C_+\ov{q}R_0(\lambda_j\pm0\ii)qC_+$ to both sides, using the definition of $P_j$ and using Lemma \ref{lem:LaxeigenqRqfixpt} yields
\begin{align}\label{eqn:Qjkernel}
    (1-C_+\ov{q}R_0(\lambda_j\pm0\ii)qC_+)^2 g=0.
\end{align}

Since $C_+\ov{q}R_0(\lambda_j\pm0\ii)qC_+$ is compact, $1-C_+\ov{q}R_0(\lambda_j\pm0\ii)qC_+$ is a Fredholm operator of index zero on $L^2_+(\R)$. Namely, the dimensions of its kernel and cokernel coincide. Consequently,
\begin{align*}
    \ker\left( (1-C_+\ov{q}R_0(\lambda_j\pm0\ii)qC_+)^2\right) =\ker \left( 1-C_+\ov{q}R_0(\lambda_j\pm0\ii)qC_+ \right) =\mathrm{span}\{C_+\ov{q}\phi_j\} \,,
\end{align*}
where we used Lemma \ref{lem:LaxeigenqRqfixpt} and the simplicity of eigenvalues discussed after Lemma \ref{lem:qeigenfuncinprod}. Thus, \eqref{eqn:Qjkernel} implies that $g = \alpha\, C_+\ov{q}\phi_j$ for some $\alpha\in\C$. Since $R_0(\lambda_j\pm0\ii)\pot \phi_j=\phi_j$, it follows from \eqref{eqn:Qjcontrassump} that
$$
g= C_+ \ov{q}Q_jR_0(\lambda_j\pm0\ii)qC_+ g = \alpha\, C_+ \ov{q}Q_jR_0(\lambda_j\pm0\ii)qC_+ \ov{q}\phi_j = \alpha C_+ \ov{q} Q_j \phi_j = 0 \,,
$$
proving triviality of the kernel.

\medskip

\emph{Step 2c.} By the Fredholm alternative, it follows from the compactness in Step 2a and the injectivity in Step 2b that the operator $1-C_+\ov{q} Q_j R_0(k)qC_+$ is boundedly invertible for $k=\lambda_j\pm0\ii$. By the continuity in Step 2a, this holds for all $k$ in a $\Chat$-neighbourhood of $\lambda_j\pm 0$.

Therefore, solving equations \eqref{eq:embeddedexpanproof3b} and \eqref{eq:embeddedexpanproof4b} and inserting the resulting expressions into \eqref{eq:embeddedexpanproof3a} and \eqref{eq:embeddedexpanproof4a}, we obtain
\begin{align*}
    \quad Q_jm_0(k)&=Q_jR_0(k)Q_jq+Q_jR_0(k)qC_+ (1-C_+\ov{q}Q_jR_0(k)qC_+)^{-1} C_+\ov{q} Q_j R_0(k)Q_jq
\end{align*}
and
\begin{align*}
Q_jm_\e(\lambda\pm0\ii)
    &=Q_j\e(\lambda)+Q_jR_0(\lambda\pm 0\ii) qC_+ (1-C_+\ov{q}Q_jR_0(\lambda\pm0\ii)qC_+)^{-1} C_+\ov{q}Q_j\e(\lambda) \,.
\end{align*}

The fact that $k\mapsto\langle Q_jm_0(k), f\rangle$ is continuous in a $\Chat$-neighbourhood of $\lambda_j\pm0\ii$ for any $f\in L^1_+$ now follows by the same reasoning as Lemma \ref{lem:pertresolvcont}, using, in particular, the continuity in operator norm of $k\mapsto (1-C_+\ov{q}Q_jR_0(k)qC_+)^{-1}$.

To prove that $\lambda\mapsto \ov{\e(\lambda)} Q_jm_\e(\lambda\pm0\ii)$ is continuous in $L^\infty$, we write
\begin{equation}
    \label{eq:eqmecont}
    \begin{split}
    \overline{\e(\lambda)} Q_jm_\e(\lambda\pm0\ii) & = 1 - \|\phi_j\|_2^{-2} \langle \e(\lambda),\phi_j \rangle \, \ov{\e(\lambda)}\phi_j \\
    & \quad + \ov{\e(\lambda)} R_0(\lambda\pm0\ii) F(\lambda) - \|\phi_j\|_2^{-2} \langle R_0(\lambda\pm 0\ii) F(\lambda),\phi_j\rangle \, \ov{\e(\lambda)}\phi_j
    \end{split}
\end{equation}
with
$$
F(\lambda) := qC_+ (1-C_+\ov{q}Q_jR_0(\lambda\pm0\ii)qC_+)^{-1} C_+\ov{q}Q_j\e(\lambda) \,.
$$

We begin by noting that $\lambda\mapsto \langle \e(\lambda),\phi_j\rangle$ is continuous since $\phi_j\in L^1$. Moreover, $\lambda\mapsto\ov{q}\e(\lambda)$ is $L^2$-continuous by dominated convergence, so
$$
\ov{q}Q_j \e(\lambda) = \ov{q}\e(\lambda) - \|\phi_j\|_2^{-2} \langle \e(\lambda),\phi_j \rangle \ov{q}\phi_j
$$
is also $L^2$-continuous. It follows from Steps 2a and 2b that the inverse operator in the definition of $F$ is continuous as a bounded operator on $L^2$ and therefore $F$ is $L^1$-continuous. The $(L^\infty,w^*)$-continuity of $\lambda\mapsto R_0(\lambda\mp0\ii)\phi_j$ together with the strong continuity of $F$ imply continuity of $\lambda\mapsto \langle R_0(\lambda\pm 0\ii) F(\lambda),\phi_j\rangle$. 

A combination of these facts shows that the second and fourth term on the right side of \eqref{eq:eqmecont} are $L^\infty$-continuous, provided we can show that $\lambda\mapsto\ov{\e(\lambda)}\phi_j$ is $L^\infty$-continuous. To prove this, consider a sequence $(\lambda_n)\subset\R$ with $\lambda_n\to\lambda$ and let $\epsilon>0$. Since $\phi_j$ tends to zero at infinity (by Lemma \ref{lem:LaxeigenqRqfixpt}), there is an $R>0$ such that $|\phi_j(x)|\leq\epsilon/2$ if $|x|\geq R$. Thus, $|(\ov{\e(\lambda_n)}-\ov{\e(\lambda)})\phi_j(x)|\leq \epsilon$ if $|x|\geq R$. Meanwhile, for $|x|\leq R$ we have $|(\ov{\e(\lambda_n)}-\ov{\e(\lambda)})\phi_j(x)|\leq |x||\lambda_n-\lambda| \|\phi_j\|_\infty \leq R |\lambda_n - \lambda| \|\phi_j\|_\infty$ and this is $\leq \epsilon$ provided $n$ is large enough, thus proving the claimed continuity.

It remains to prove continuity of the third term on the right side of \eqref{eq:eqmecont}. For this purpose we note that $\lambda\mapsto\ov{\e(\lambda)}F(\lambda)$ is $L^1$-continuous. This follows by bounding
$$
\| \ov{\e(\lambda_n)}F(\lambda_n) - \ov{\e(\lambda)} F(\lambda) \|_1 \leq \| \ov{\e(\lambda_n)}(F(\lambda_n)-F(\lambda))\|_1 + \|(\ov{\e(\lambda_n)}-\ov{\e(\lambda)})F(\lambda)\|_1
$$
and noting that the first term goes to zero by $L^1$-continuity of $F$ and the second one by dominated convergence. Noting that
$$
(\ov{\e(\lambda)} R_0(\lambda\pm0\ii) F(\lambda))(x) = \ii \int_{\mp\infty}^x (\ov{\e(\lambda)}F(\lambda))(y) \dd y \,,
$$
we deduce the $L^\infty$-continuity of this term, thereby completing the proof of the $L^\infty$ continuity of $\overline{\e(\lambda)} Q_jm_\e(\lambda\pm0\ii)$.

\medskip

\emph{Step 3.}
To finish the proof of the lemma it remains to study the behaviour of $\ov{\e(\lambda)}m_\e(\lambda\pm0\ii)$ as $\lambda\to 0_+$. 

The claimed result is straightforward when $0$ is not an eigenvalue of $\CMop$ by the continuity of $\ov{\e(\lambda)}m_\e(\lambda\pm0\ii)$ at $0$, which we have already shown, and the fact that $m_\e(0\pm0\ii)=1$, see \eqref{eq:lowenergy1}.

When $0$ is an eigenvalue of $L_q$, let $P$ denote the projection onto the corresponding eigenfunction and set $Q :=1-P$. The decomposition \eqref{eq:embeddedexpanproof0b}, the vanishing in \eqref{eq:embeddedexpanproof1} and the formula for $m_\e(\lambda\pm0\ii)$ in Step 2c give
    \begin{align*}
        \ov{\e(\lambda)}m_{\e}(\lambda\pm0\ii)-1=&-\ov{\e(\lambda)}P\e(\lambda)\\&+\ov{\e(\lambda)}Q R_0(k)qC_+(1-C_+\ov{q}QR_0(\lambda\pm0\ii)qC_+)^{-1}C_+\ov{q}Q\e(\lambda)
    \end{align*}
for all $\lambda$ in a right-neighbourhood of $0$. We have also seen that $(1-C_+\ov{q}QR_0(\lambda\pm0\ii)qC_+)^{-1}$ is bounded and continuous for $\lambda$ in a right-neighbourhood of $0$. The result then follows since $\ov{\e(\lambda)}P\e(\lambda)\to 0$ in $L^\infty$ and $C_+\ov{q}Q\e(\lambda)\to 0$ in $L^2$. The latter limits follow since $\phi\in L^1_+$ implies $\langle \e(\lambda),\phi\rangle \to 0$ and $C_+\ov{q}=0$. This completes the proof.
\end{proof}

\begin{remark}
    We emphasise that in the previous proof we have shown that 
    \begin{align*}
        \sclp{m_{\e}(\lambda\pm0\ii),\phi_j}=0,
    \end{align*}
    for all $\lambda\in[0,\infty)\setminus\sigma_\mathrm{p}(\CMop)$ and each $j$. This indicates that we can think of $m_{\e}$ as eigenfunctions of the continuous spectrum, orthogonal to the eigenfunctions from the point spectrum. This orthogonality also follows from $\ran\Phi^* = (\ran P_\mathrm p(L_q))^\bot$, proved in Theorem \ref{thm:meFT}, but here we have given an independent proof.
\end{remark}

\begin{remark}
    The fact that $\ov{\e(\lambda)}m_\e(\lambda\pm0\ii)\to 1$ as $\lambda\to 0_+$ means that in the Calogero--Moser case there is no distinction between zero being a resonance or not. This is in contrast to the Benjamin--Ono case.
\end{remark}

Next, we turn our attention to the scattering coefficients. Recall that the functions $\beta$ and $\Gamma$ were defined on $[0,\infty)\backslash\sigma_{\mathrm{p}}(\CMop)$. We now extend them to all of $[0,\infty)$.

\begin{corollary}
    Let $q\in L_{+}^2(\R)\cap L^1_+(\R)$. Then, $\beta$ and $\Gamma$ have unique extensions in $C([0,\infty))$. Moreover, 
    \begin{align*}
        \beta(\lambda) =o_{\lambda\rightarrow 0_+}(1) \qquad \text{and}\qquad\Gamma(\lambda)-1=o_{\lambda\rightarrow 0_+}(\lambda) . 
    \end{align*}
\end{corollary}

\begin{proof}
    As in the proof of Lemma \ref{lem:emederiv} we write
    $$
    \Gamma(\lambda) = 1 + \ii \langle \ov{\e(\lambda)} m_\e(\lambda+0\ii), \ov{\e(\lambda)} qC_+\ov{q}\e(\lambda) \rangle \,.
    $$
    By Theorem \ref{thm:embeddexpan}, $\lambda\mapsto \ov{\e(\lambda)} m_\e(\lambda+0\ii)$ extends $L^\infty$-continuously to all of $[0,\infty)$ and, by Step 4 in the proof of Lemma \ref{lem:emederiv}, $\lambda\mapsto \ov{\e(\lambda)} qC_+\ov{q}\e(\lambda)$ is $L^1$-continuous on $[0,\infty)$. This defines a continuous extension of $\Gamma$ to all of $[0,\infty)$. As $\lambda\to 0$ we have $\ov{\e(\lambda)} m_\e(\lambda+0\ii)\to 1$ in $L^\infty$ and $\ov{\e(\lambda)} qC_+\ov{q}\e(\lambda)\to 0$ in $L^1$ (since $C_+\ov{q}=0$), so $\Gamma(0)=0$.

    For $\beta$, we fix an eigenvalue $\lambda_j\in[0,\infty)$ and apply the expansion for $m_0$ about $\lambda_j$ from Theorem \ref{thm:embeddexpan}. Substituting into the definition of $\beta$, we obtain 
    \begin{align*}
        \beta(\lambda)&=\frac{1}{\lambda-\lambda_j} \int_{-\infty}^\infty \e^{-\ii y\lambda}\pot \phi_j\dd y +\ii \sqrt{2\pi}\, (\F q)(\lambda)+\ii \int_{-\infty}^\infty \e^{-\ii y\lambda}\pot h(\lambda+0\ii)\dd y \,.
    \end{align*}
    Here the Fourier transform $\F q$ is continuous since $q\in L^1$. Moreover, the last term is continuous since $\lambda\mapsto q C_+\ov{q}\e(\lambda)$ is continuous in $L^1$ and since $\lambda\mapsto h(\lambda\pm 0\ii)$ is bounded in $L^\infty$ for $\lambda$ in a neighbourhood of $\lambda_j$. The latter is implicit in the proof of Theorem~\ref{thm:embeddexpan}.
    
    To handle the first term, we use \eqref{eq:efeqev} to write
    \begin{align*}
        \frac{1}{\lambda-\lambda_j} \int_{-\infty}^\infty \e^{-\ii y\lambda}\pot \phi_j\dd y
        =\frac{1}{\lambda-\lambda_j}(\lambda-\lambda_j) \langle \phi_j,\e(\lambda) \rangle = \sqrt{2\pi}\, (\F \phi_j)(\lambda) \,.
    \end{align*}
    Since $\phi_j\in L^1(\R)$, this term is continuous as well. This defines a continuous extension of $\beta$ to all of $[0,\infty)$.
        
    The limit $\beta(\lambda)\rightarrow 0$ as $\lambda\rightarrow 0_+$ follows from the continuity and \eqref{eq:lowenergy2} when $0$ is not an eigenvalue of $L_q$. When it is, it follows from the above formula for $\beta$ with $\lambda_j=0$, since $qC_+\ov{q}h,q,\phi\in L^1(\R)$ and their Fourier transforms vanish on negative frequencies.
    
    Finally, for the asymptotics of $\Gamma(\lambda)$ as $\lambda\rightarrow 0_+$, we deduce from $\Gamma(0)=0$ and the differential equation \eqref{eqn:gammaderivident} in Lemma \ref{lem:emederiv} that
    $$
    \Gamma(\lambda) = \exp\left(-\frac{1}{2\pi\ii} \int_0^\lambda |\beta(\lambda')|^2\,\dd\lambda' \right).
    $$
    The stated asymptotic order now follows from that of $\beta$. 
\end{proof}


\subsection*{Refined expansion around an eigenvalue}

In Theorem \ref{thm:embeddexpan} we have shown that $m_0$ has a pole at an eigenvalue $\lambda_j$ and computed the residue. Now we derive the next term in the Laurent expansion of $m_0$ and define the phase coefficients proposed for the inverse scattering transform.

\begin{lemma}\label{lem:laurentexpanimprov}
    Let $q\in L_{+}^2(\R)\cap L^1_+(\R)$. Let $h$ be as in Theorem \ref{thm:embeddexpan}. Then, for all $j=1,\ldots,N$, there is a number $\gamma_j\in\C$ such that, for either choice $\pm$ of sign,
    $$
    h(x,\lambda_j\pm0\ii) = (\gamma_j + x)\, \phi_j(x) + \ii \sum_{j'\neq j} \frac{\phi_j(x)}{\lambda_j - \lambda_{j'}} \,.
    $$
    Moreover, if $\lambda_j\geq 0$, then $\beta(\lambda_j)=0$.
\end{lemma}

\begin{proof}
    We recall from Lemma \ref{efl1} that $q\in L^2_+\cap L^1_+$ also gives
    \begin{align}\label{eqn:assump3}
        P_{\mathrm{p}}(\CMop)L^2_+(\R)\subset \langle x\rangle^{-1} L^\infty(\R).
    \end{align}

    \medskip
    
    \emph{Step 1.}
    Fix $j\in \{1,\ldots,N\}$. Then by Theorem \ref{thm:embeddexpan}
    \begin{align*}
        m_0(k)=-\frac{\ii\phi_j}{k-\lambda_j}+ h_j(k) \,,
    \end{align*}
    where $h_j$ is $(L^\infty,w^*)$ continuous in $\Chat$-neighbourhoods of both $\lambda_j+ 0\ii$ and $\lambda_j-0\ii$. 

    To refine this expansion, we use the integral equation \eqref{eqn:m0inteq} for $m_0$, obtaining 
    \begin{equation}
        \begin{split}\label{eqn:improvlaurentexp1}
        h_j(k)&=\frac{\ii\phi_j}{k-\lambda_j}+R_0(k)q+R_0(k)\pot\left[-\frac{\ii\phi_j}{k-\lambda_j}+h_j(k)\right]\\
        &=\frac{\ii}{k-\lambda_j}(R_0(\lambda_j\pm0\ii)-R_0(k))\pot\phi_j+R_0(k) q+R_0(k)\pot h_j(k),
        \end{split}
    \end{equation}
    where we used $\phi_j=R_0(\lambda_j\pm0\ii)\pot\phi_j$. 
    
    We now take $k=(\lambda_j+h)\pm 0\ii$ with $h\in\R$ and take the limit $h\to 0$. By Theorem \ref{thm:embeddexpan} $h_j(k)$ is $w^*$-continuous and by Lemma \ref{lem:qRqcompact} $R_0(k)q$ is $w^*$-continuous. Moreover, it is easy to see that $R_0(k)\pot h_j(k)$ is also $w^*$-continuous. Indeed, if $f\in L^1_+$, then, by Step 2 in the proof of Lemma \ref{lem:qRqcompact}, $R_0(\ov{k})f$ converges pointwise and boundedly and therefore, by dominated convergence, $\ov{q}R_0(\ov{k})f$ converges in $L^2$ and $qC_+\ov{q}R_0(\ov{k})f$ converges in $L^1$. Since this convergence is strong, when paired with the $w^*$-convergence of $h_j(k)$ in $L^\infty$ we obtain convergence of $\langle h_j(k), qC_+\ov{q}R_0(\ov{k})f\rangle$, as claimed.

    In Step 2 we will show that
    \begin{align*}
        & \frac{\ii}{k-\lambda_j}(R_0(\lambda_j+0\ii)-R_0(k))\pot\phi_j =-\frac{\ii}{h}(R_0(\lambda_j+h+0\ii)-R_0(\lambda_j+0\ii))\pot\phi_j\\
        & \to -\ii \left( \ii x R_0(\lambda_j+0\ii)qC_+\ov{q}\phi_j - \ii R_0(\lambda_j+0\ii)q C_+ \ov{q} x\phi_j + (2\pi)^{-1} \langle \phi_j,q \rangle \, R_0(\lambda_j+0\ii) q \right) \\
        & \quad = x \phi_j - R_0(\lambda_j+0\ii)q C_+ \ov{q} x\phi_j - R_0(\lambda_j+0\ii) q \,.
    \end{align*}
    In the last equation we use $R_0(\lambda_j+0\ii)qC_+\ov{q}\phi_j=\phi_j$ as well as the normalisation \eqref{eqn:eigfncnormal}.

    The convergence shown in Step 2 holds only when integrated against a test function in $\langle x \rangle^{-1} L^1$, but since this set is dense in $L^1$, we deduce in the limit from \eqref{eqn:improvlaurentexp1} that
    \begin{align*}
        h_j(\lambda_j\pm 0\ii)= x \phi_j-R_0(\lambda_j+0\ii)\pot x\phi_j +R_0(\lambda_j+0\ii)\pot h_j(\lambda_j+0\ii) \,,
    \end{align*}
    which can be rewritten as
    \begin{align*}
        h_j(\lambda_j\pm0\ii)-x\phi_j=R_0(\lambda_j\pm0\ii)\pot(h_j(\lambda_j\pm0\ii)-x\phi_j) \,.
    \end{align*}
    By \eqref{eqn:assump3}, we have $h_j(\lambda_j\pm0\ii)-x\phi_j\in L^\infty(\R)$. Thus, Lemma \ref{lem:LaxeigenqRqfixpt} shows that $h_j(\lambda_j\pm0\ii)- x\phi_j$ are eigenfunctions of $\CMop$ with corresponding eigenvalue $\lambda_j$. Therefore, by simplicity (see the discussion after Lemma \ref{lem:qeigenfuncinprod}), 
    \begin{align*}
        h_j(\lambda_j\pm0\ii)- x\phi_j=\gamma_j^\pm\phi_j
    \end{align*}
    for some constants $\gamma_j^\pm\in \C$.

    It remains to be shown that $\gamma_j^+=\gamma_j^{-}$. When $\lambda_j< 0$, this is clear since $h_j(\lambda_j+0\ii)=h_j(\lambda_j-0\ii)$. When $\lambda_j\geq0$, we observe that 
    \begin{align*}
        (\gamma^{+}_j-\gamma^{-}_j)\phi_j&=h_j(\lambda_j+0\ii)-h_j(\lambda_j-0\ii)\\
        &=m_0(\lambda_j+0\ii)-m_0(\lambda_j-0\ii)=\beta(\lambda_j)m_\e(\lambda_j-0\ii),
    \end{align*}
    where the final identity uses \eqref{eqn:m0mebetarel} from Lemma \ref{lem:scatobjsbeta}. Note that $(\gamma^{+}_j-\gamma^{-}_j)\phi_j$ is square-integrable, while, in view of the asymptotics in Lemma \ref{lem:uniqueme}, $m_\e(\lambda_j-0\ii)$ is not square-integrable near $+\infty$. We deduce that $\beta(\lambda_j)=0$, and hence $\gamma^+_j=\gamma^{-}_j$. This concludes the proof, except for the convergence result.

    \medskip

    \emph{Step 2.} Let $\phi\in L^1_+$ with $\langle x\rangle\phi \in L^\infty$ and let $\lambda\in\R$. We are going to show that
    \begin{align*}
        & h^{-1} ( R_0(\lambda +h\pm 0\ii) - R_0(\lambda\pm 0\ii)) qC_+\ov{q}\phi \\
    & \to \ii x R_0(\lambda\pm 0\ii) qC_+\ov{q}\phi - \ii R_0(\lambda\pm 0\ii) qC_+\ov{q} x\phi + (2\pi)^{-1} \langle\phi,q\rangle \, R_0(\lambda\pm\ii 0) q
    \end{align*}
    as $h\to 0$ with $h\in\R$. The convergence is understood in a weak sense when integrated again a test function $f\in L^1_+$ with $\langle x \rangle f\in L^1$.

    We prove this for the upper sign, the opposite case being similar. We begin by writing
    \begin{align*}
        (R_0(\lambda+0\ii)&\pot\phi)(x) \\&=\int_\R \e^{-\ii h x}G_{\lambda+h+0\ii}(x-y)\e^{\ii h y}\pot \phi\dd y\\
        &=\int_\R \e^{-\ii h x}G_{\lambda+h+0\ii}(x-y)\pot \e(h)\phi\dd y - \int_\R G_{\lambda+0\ii}(x-y) qC_{h}\overline{q}\phi\dd y
    \end{align*}
    with $C_h := \F^{-1}\1_{(-h,0)}\F$ if $h\geq 0$ and $C_{h} := -\F^{-1}\1_{(0,-h)}\F$ if $h\leq 0$, see \eqref{eqn:Chprop}. Thus,
    \begin{align*}
        & h^{-1} ( R_0(\lambda +h\pm 0\ii) - R_0(\lambda\pm 0\ii)) qC_+\ov{q}\phi \\ 
        & = \frac{1}{h}\int_\R G_{\lambda+h+0\ii}(x-y)\left(qC_+\ov{q}\phi-\e^{-\ii h x}\pot \e(h)\phi\right)\dd y + \frac{1}{h}\int_\R G_{\lambda+0\ii}(x-y)qC_{h} \overline{q}\phi\dd y\\
        & \eqqcolon (1)+(2) \,.
    \end{align*}

    For the term $(1)$, we write 
    \begin{align*}
        (1) & =\e^{-\ii hx}\int_\R G_{\lambda+h+0\ii}(x-y)\left(\frac{\e^{\ii hx}-1}{h}qC_+\ov{q}\phi+\pot \frac{1-\e(h)}{h}\phi\right)\dd y \\
        & = \int_\R \frac{1-\e^{-\ii hx}}{h} G_{\lambda+h+0\ii}(x-y)qC_+\ov{q}\phi \dd y + \int_\R G_{\lambda+0\ii}(x-y) \ov{\e(h)} \pot \frac{1-\e(h)}{h}\phi \dd y \,.
    \end{align*}
    Thus, if $f$ is a test function as specified, then
    $$
    \langle (1),f \rangle = \langle R_0(\lambda+h+0\ii) qC_+\ov{q}\phi, \frac{1-\e(h)}{h} f \rangle
    + \langle \ov{\e(h)} q C_+ \ov{q} \frac{1-\e(h)}{h} \phi, R_0(\lambda-0\ii)f \rangle \,.
    $$
    By dominated convergence, the extra assumption on $f$ implies that $h^{-1}(1-\e(h))f\to -\ii xf$ in $L^1$. This strong convergence coupled with the $(L^\infty,w^*)$-convergence of $R_0(\lambda+h+0\ii) qC_+\ov{q}\phi$ (Lemma \ref{lem:qRqcompact}) shows that
    $$
    \langle R_0(\lambda+h+0\ii) qC_+\ov{q}\phi, \frac{1-\e(h)}{h} f \rangle
    \to \ii\ \langle R_0(\lambda+0\ii) qC_+\ov{q}\phi, x f \rangle
    $$
    Similarly, the extra boundedness assumption on $\phi$ and dominated convergence implies $h^{-1}\ov{q}(1-\e(h))\phi\to -\ii \ov{q} x\phi$ in $L^2$, and, by dominated convergence, $\ov{\e(h)} q\to q$ in $L^2$. This shows that
    $$
    \langle \ov{\e(h)} q C_+ \ov{q} \frac{1-\e(h)}{h} \phi, R_0(\lambda-0\ii)f \rangle \to -\ii\, \langle q C_+ \ov{q} x \phi, R_0(\lambda-0\ii)f \rangle \,.
    $$
    This gives the limit of term $(1)$.

    For the term $(2)$ with $h\geq 0$ we can apply Step 1 in the proof of Lemma \ref{lem:emederiv}, noting that $C_0=0$, and obtain
    $$
    (2) \to \frac1{2\pi} \int_\R \phi(y') \ov{q(y')} \dd y' \
    \int_\R G_{\lambda+0\ii}(x-y) q(y)\dd y = (2\pi)^{-1} \langle \phi,q \rangle \ R_0(\lambda+0\ii) q \,.
    $$
    The same holds for $h\leq 0$ by essentially the same argument. This gives the limit of term $(2)$ and thereby concludes the proof of the lemma.
\end{proof}


\newpage

\part{Trace formulas}

\section{First order trace formulas} \label{sec:traceform}
The objective in this and the next section is to prove identities, so called trace formulas, that relate the spectral and scattering quantities of $\CMop$ to quantities defined in terms of $q$. All these quantities are conserved by the Calogero--Moser equation.

In Section \ref{sec:jostinhom} we have introduced $\beta$ for $q\in L^2_{+}(\R)\cap L^1_+(\R)$. In this section we extend this definition to all $q\in L^2_+(\R)$ by setting
$$
\beta:= \sqrt{2\pi}\ii \Phi(q) \,,
$$
see Remark \ref{rem:b}.

The first trace identity holds for all $q\in L^2_{+}(\R)$ and is a direct consequence of the analysis in Section \ref{sec:distortedFT}.

\begin{theorem}\label{thm:traceformula1}
    Let $q\in L^2_{+}(\R)$. Then
    \begin{align}\label{eqn:traceident1}
        \int_{-\infty}^\infty |q(x)|^2\dd x= \frac{1}{2\pi}\int_0^\infty \abs{\beta(\lambda)}^2\dd \lambda + 2\pi N(\CMop) \,.
    \end{align}
\end{theorem}

We note that this trace formula accounts for the deficit term in the inequality \eqref{eqn:Nfinbound}.

Our second main result in this section concerns a formula for $\Tr( f(\CMop)-f(\CMop[0]) )$ for a rather large class of functions $f$. Such formulas are called Birman--Krein trace formulas in the textbook \cite{dyatlov_zworski}, honoring \cite{birmankrein}.

Let $\mathcal C$ denote the class of all functions $f$ on $\R$ that have two locally bounded derivatives and satisfy, for some $\epsilon>0$,
$$
(\lambda^2 f'(\lambda))' = O_{\lambda\to\infty}(\lambda^{-1-\epsilon}) \,.
$$
Clearly, $\mathcal S(\R)\subset\mathcal C$ and also the functions $\lambda\mapsto (\lambda-k)^{-1}$ with $k\in\C\setminus\R$ belong to $\mathcal C$.

As elsewhere, $\{\lambda_j\}_{j=1}^N$ with $N=N(\CMop)$ denote the eigenvalues of $\CMop$.

\begin{theorem}
    \label{thm:birmankrein}
    Let $q\in L^2_{+}(\R)$ and let $f\in\mathcal C$. Then $f(\CMop)-f(\CMop[0])$ is trace class and
    \begin{equation}
        \label{eq:birmankrein}
        \Tr( f(\CMop)-f(\CMop[0]) ) = \frac1{(2\pi)^2} \int_0^\infty f(\lambda) |\beta(\lambda)|^2\dd\lambda + \sum_{j=1}^N f(\lambda_j) \,.
    \end{equation}
\end{theorem}

In particular, for $f(\lambda)=(\lambda-k)^{-1}$ we obtain
$$
\Tr \left( (\CMop-k)^{-1} - (\CMop[0]-k)^{-1} \right) = \frac1{(2\pi)^2} \int_0^\infty \frac{|\beta(\lambda)|^2}{\lambda-k} \dd\lambda + \sum_{j=1}^N \frac{1}{\lambda_j-k} \,.
$$
It is easy to see that this formula, which holds for $k\in\C\setminus\R$ by Theorem \ref{thm:birmankrein}, holds for all $k\in\C\setminus\sigma(\CMop)$. In fact, it is this formula that we will prove and then the full statement of Theorem \ref{thm:birmankrein} follows by abstract arguments.

The first ingredient in the theorems in this section is the following result.

\begin{lemma}\label{lem:riemannhilbtraceform}
    Let $q\in L^2_{+}(\R)$. Then, for any Borel set $\Lambda\subset\R$,
    $$
    \langle\1_\Lambda(\CMop)q,q\rangle = \frac1{2\pi} \int_{\Lambda\cap\R_+} |\beta(\lambda)|^2\dd\lambda + 2\pi\, \#\{ j:\, \lambda_j\in\Lambda\} \,.
    $$
    Moreover, for any $k\in\C\setminus\sigma(\CMop)$,
    \begin{equation}\label{eqn:relationtrace}
    \int_{-\infty}^{\infty} \ov{q(x)} m_0(x,k) \dd x = \frac1{2\pi} \int_0^\infty \frac{|\beta(\lambda')|^2}{\lambda'-k}\,\dd\lambda' + 2\pi \sum_{j=1}^N \frac{1}{\lambda_j - k} \,.
    \end{equation}
    Finally, for a.e.~$\lambda\in\R$,
    \begin{equation*}
        \int_{-\infty}^\infty \overline{q(x)}m_0(x,\lambda-0\ii)\dd x = - \ii\, C_{-} \big(\abs{\beta}^2\big)(\lambda) + 2\pi  \sum_{j=1}^N \frac{1}{\lambda_j -\lambda}
    \end{equation*}
    where $C_{-}=1-C_+$ and with $\beta$ extended by zero to $(-\infty,0)$. 
\end{lemma}

Given this lemma, the proof of the first trace formula is immediate.

\begin{proof}[Proof of Theorem \ref{thm:traceformula1}]
    It suffices to take $\Lambda=\R$ in the first equality in Lemma \ref{lem:riemannhilbtraceform}.   
\end{proof}

\begin{proof}[Proof of Lemma \ref{lem:riemannhilbtraceform}]
    Let $\nu$ be the spectral measure of $q$ with respect to $\CMop$, that is, $\nu(\Lambda):=\langle \1_\Lambda(L_q) q,q\rangle$ for any Borel set $\Lambda\subset\R$. It follows from Corollary \ref{cor:nosc} and Theorem \ref{thm:meFT} that, for any Borel set $\Lambda\subset\R$,
    \begin{align}
        \label{eq:nu}
        \nu(\Lambda) & = \int_{\Lambda\cap\R_+} |(\Phi q)(\lambda)|^2\,\dd\lambda + \sum_{\lambda_j\in\Lambda} \frac{|\langle q,\phi_j\rangle|^2}{\|\phi_j\|^2} \notag \\
        & = \frac1{2\pi} \int_{\Lambda\cap\R_+} |\beta(\lambda)|^2\,\dd\lambda + 2\pi \, \#\{ j:\ \lambda_j\in\Lambda\} \,,
    \end{align}
    where we used the definition of $\beta$, as well as Lemma \ref{lem:qeigenfuncinprod}. This proves the first assertion in the lemma. Note that this formula means that the absolutely continuous part of $\nu$ has density $(2\pi)^{-1}|\beta|^2$ and its singular part consists in $2\pi$ times a delta function at the eigenvalues.

    Let $k\in\C\setminus\sigma(\CMop)$ and recall that $m_0(k)= R(k) q$. This implies that
    \begin{equation}
        \label{eq:cauchymeasure}
            \int_{-\infty}^{\infty} \ov{q(x)} m_0(x,k)\,\dd x  = \langle R(k) q,q \rangle = \int_\R \frac{\dd \nu(\lambda)}{\lambda-k} \,,
    \end{equation}
    where we used the spectral theorem in the last step.

    According to our formula for $\nu$, we can rewrite this as
    $$
    \int_{-\infty}^{\infty} \ov{q(x)} m_0(x,k) \dd x = \frac1{2\pi} \int_0^\infty \frac{|\beta(\lambda')|^2}{\lambda'-k}\,\dd\lambda' + 2\pi \sum_{j=1}^N \frac{1}{\lambda_j - k} \,,
    $$
    which proves the second assertion in the lemma.

    To prove the third assertion we want to take $k=\lambda-\ii\epsilon$ with $\lambda\in\R\setminus\sigma_\mathrm p(\CMop)$ and $\epsilon>0$ and let $\epsilon\to 0$. The left side has a limit according to the continuity statement made right after the definition of $m_0(k)$, and the second term on the right side clearly has a limit. Thus, the first term on the right side has a limit as well and we have
    $$
    \int_{-\infty}^{\infty} \ov{q(x)} m_0(x,\lambda-0\ii) \dd x = \lim_{\epsilon\to 0^+} \frac1{2\pi} \int_0^\infty \frac{|\beta(\lambda')|^2}{\lambda'-\lambda + \ii\epsilon}\,\dd\lambda' + 2\pi \sum_{j=1}^N \frac{1}{\lambda_j - \lambda} \,.
    $$
    By the Sokhotski--Plemelj formula (see, e.g., \cite[Theorem 1.2.5]{yafaev_mathematical_1992}) it follows that for almost every $\lambda\in\R$ we have
    $$
    \lim_{\epsilon\to 0^+} \int_0^\infty \frac{|\beta(\lambda')|^2}{\lambda'-\lambda + \ii\epsilon}\,\dd\lambda'
    = -\ii\pi\,|\beta(\lambda)|^2 + \mathrm{p.v.} \int_0^\infty \frac{|\beta(\lambda')|^2}{\lambda'-\lambda}\,\dd\lambda' = -2\pi\ii \, C_-(|\beta|^2)(\lambda) \,.
    $$
    It follows, in particular, that the principal value integral exists for a.e.~$\lambda\in\R$. (This follows also standard harmonic analysis results concerning the Hilbert transform of an $L^1$-function.) This shows that the formula in the lemma holds for almost every $\lambda\in\R$.
\end{proof}


The following is the second ingredient in the proof of Theorem \ref{thm:birmankrein}.

\begin{lemma}\label{lem:resdifftrace}
    Let $q\in L^2_+(\R)$. Then for any $k\in\C\setminus\sigma(L_q)$
    $$
    \Tr((L_q-k)^{-1} - (L_0-k)^{-1}) = \frac1{2\pi}\, \langle (\CMop - k)^{-1} q,q \rangle \,.
    $$
\end{lemma}

The proof of this lemma uses some ideas from \cite{talbut_dissertation_2021}, where the perturbation determinant in the Benjamin--Ono setting was computed.

\begin{proof}
    We will show the formula for  $k=\kappa<\inf\sigma (\CMop)$ sufficiently negative, the extension to all $k\in \C\backslash\sigma(\CMop)$ will follow by analytic continuation.

    From the resolvent formula it follows that $\kappa\in\R\setminus\sigma(\CMop)$ if and only if the operator $1-C_+\ov{q}R_0(\kappa)qC_+$ is invertible and, in this case,
    \begin{align*}
        C_+\ov{q}R(\kappa)=(1-C_+\ov{q}R_0(\kappa)q C_+)^{-1}C_+\ov{q}R_0(\kappa) \,.
    \end{align*}
    It follows that
    $$
    R(\kappa)-R_0(\kappa) = R_0(\kappa)qC_+\ov{q}R(\kappa) = R_0(\kappa)q C_+ (1-C_+\ov{q}R_0(\kappa)q C_+)^{-1}C_+\ov{q}R_0(\kappa) \,.
    $$
    According to \cite[Lemma 2.1]{killip_scaling_2025} we have $\| C_+\ov{q}R_0(\kappa)q C_+\|<1$ for all sufficiently negative $\kappa$ and then we can expand $(1-C_+\ov{q}R_0(\kappa)q C_+)^{-1}$ in a norm-convergent Neumann series. From now on we assume that $\kappa$ is sufficiently negative. Taking into account that $C_+\overline q R_0(\kappa)$ is Hilbert--Schmidt \cite[Lemma 2.5]{killip_scaling_2025}, we obtain
    \begin{align*}
        \Tr(R(\kappa)-R_0(\kappa))&=\Tr(R_0(\kappa)q C_+(1-C_+\ov{q}R_0(\kappa)qC_+)^{-1}C_+\ov{q}R_0(\kappa))\\
        &=\Tr((1-C_+\ov{q}R_0(\kappa)q)^{-1}C_+\ov{q}R_0(\kappa)^2qC_+)\\
        & = \sum_{n=0}^\infty \Tr((C_+\ov{q}R_0(\kappa)q C_+)^n C_+\ov{q}R_0(\kappa)^2qC_+).
    \end{align*}
    For each $n\in\N_0$ we shall show that
    \begin{equation}
        \label{eq:detgoal}
        \Tr((C_+\ov{q}R_0(\kappa)q C_+)^n C_+\ov{q}R_0(\kappa)^2qC_+)
        = \frac1{2\pi} \, \langle (R_0(\kappa) q C_+ \ov{q})^n R_0(\kappa)q,q \rangle \,.
    \end{equation}
    Once we have shown this, we can sum with respect to $n\in\N_0$ and use the Neumann series $\sum_{n=0}^\infty (R_0(\kappa) q C_+ \ov{q})^n R_0(\kappa) = R(\kappa)$ to obtain the claimed bound.

    Let us turn to the proof of \eqref{eq:detgoal}. The case $n=0$ is easy. Indeed, we have
    $$
    \Tr(R_0(\kappa)\pot R_0(\kappa))=\frac{1}{2\pi}\int_0^\infty \frac{|\widehat{q}(\xi)|^2}{\xi+\kappa} \dd\xi = \frac{1}{2\pi}\sclp{R_0(\kappa)q,q} \,,
    $$
    where the left equality follows from \cite[Lemma 2.5]{killip_scaling_2025} and the right equality follows by Plancherel's theorem. For $n\geq 1$ we write
    $$
    \Tr((C_+\ov{q}R_0(\kappa)q C_+)^n C_+\ov{q}R_0(\kappa)^2qC_+)
    = \frac{1}{n+1}\partial_\kappa(\Tr((C_+\ov{q}R_0(\kappa)q C_+)^{n+1})) \,.
    $$
    We abbreviate $m=n+1$ and compute
    \begin{align*}
        &\frac{1}{m}\Tr((C_+\ov{q}R_0(\kappa)q C_+)^{m})
        \\&=\frac{1}{m(2\pi)^{m}}\int_{[0,\infty)^{m}}\int_{[0,\infty)^{m}}\frac{\widehat{q}(\xi_1-\eta_1)\widehat{\overline{q}}(\eta_1-\xi_2)}{\xi_1-\kappa}\cdots \frac{\widehat{q}(\xi_{m}-\eta_{m})\widehat{\overline{q}}(\eta_{m}-\xi_1)}{\xi_{m}-\kappa}\dd\eta\dd \xi\\ 
        &=\frac{1}{(2\pi)^{m}}\int_{[0,\infty)^{m}}\int_{\omega_{m}}\frac{\widehat{q}(\xi_1-\eta_1)\widehat{\overline{q}}(\eta_1-\xi_2)}{\xi_1-\kappa}\cdots \frac{\widehat{q}(\xi_{m}-\eta_{m})\widehat{\overline{q}}(\eta_{m}-\xi_1)}{\xi_{m}-\kappa}\dd \eta\dd\xi \,,
    \end{align*}
    where $\omega_{m} :=\{\eta\in [0,\infty)^{m}\colon \eta_{m} \leq \min\{\eta_1,\ldots,\eta_{m-1}\}\}$, and the last step follows by cyclicity. Taking $\nu_j\coloneqq\xi_j-\eta_{m}$ for $j=1,
    \ldots,m$, we have 
    \begin{align*}
        \ldots &=\frac{1}{(2\pi)^{m}}\int_{\omega_{m}}\int_{[-\eta_{m},\infty)^{m}}\frac{\widehat{q}(\nu_1+\eta_{m}-\eta_1)\widehat{\overline{q}}(\eta_1-\eta_{m}-\nu_2)}{\nu_1+\eta_{m}-\kappa}\cdots \frac{\widehat{q}(\nu_{m})\widehat{\overline{q}}(-\nu_1)}{\nu_{m}+\eta_{m}-\kappa} \dd\nu \dd\eta \,. 
    \end{align*}
    Finally, with $\zeta_j\coloneqq\eta_j-\eta_m\geq0$ for $j=1,\ldots,m-1$ and $\lambda:=\kappa -\eta_m\leq \kappa$ and noting that $\widehat{q}(\nu)=0$ for a.e.~$\nu<0$, we have 
    \begin{align*}
        \ldots &=\frac{1}{(2\pi)^m}\int_{-\infty}^\kappa \int_{[0,\infty)^{m-1}} \int_{[0,\infty)^{m}} \overline{\widehat{q}(\nu_1)}\frac{\widehat{q}(\nu_1-\zeta_1)\widehat{\overline{q}}(\zeta_1-\nu_2)}{\nu_1-\lambda} \cdots \frac{\widehat{q}(\nu_m)}{\nu_m-\lambda}\dd\nu\dd \zeta\dd\lambda \\ 
        &=\frac{1}{2\pi}\int_{-\infty}^\kappa \int_0^\infty \overline{\widehat{q}(\nu_1)} \F\left(\left[ R_0(\lambda)qC_+\overline{q}\right]^{m-1} R_0(\lambda)q\right)(\nu_1)\dd \nu_1\dd\lambda \\
        &=\frac{1}{2\pi}\int_{-\infty}^{\kappa} 
        \langle \left[ R_0(\lambda)qC_+\overline{q}\right]^{m-1}R_0(\lambda)q, q \rangle \dd\lambda,
    \end{align*}
    where we used Plancherel in the last step. Taking the derivative in $\kappa$, we arrive at the claimed equality \eqref{eq:detgoal}. This completes the proof.
\end{proof}

\begin{remark}
    In terms of the regularised perturbation determinant \cite{yafaev_mathematical_1992} the identity in Lemma \ref{lem:resdifftrace} can be formulated as
    $$
    \ln\mathrm{det}_2(1-C_+ \ov{q}(\CMop[0]-k)^{-1}q C_+) = \frac1{2\pi} \int_{-\infty}^k \langle ((\CMop[0]-\lambda)^{-1}- (\CMop-\lambda)^{-1})q,q \rangle \dd\lambda  
    $$
    for all sufficiently negative $k$ (and hence, by analyticity for all $k<\inf\sigma(\CMop)$).
\end{remark}

\begin{proof}[Proof of Theorem \ref{thm:birmankrein}]
    According to \cite{killip_scaling_2025}, the difference $(\CMop-k)^{-1}-(\CMop[0]-k)^{-1}$ is trace class all sufficiently negative $k\in\R$ and consequently for all $k\in\C\setminus\sigma(\CMop)$. Therefore by \cite[Theorems 0.9.4 and 0.9.7]{yafaev_mathematical_2010}, there is a function $\xi\in L^1(\R;(1+|\lambda|)^{-2})$, vanishing on $(-\infty,\inf\sigma(\CMop))$, such that for any $f\in\mathcal C$ the operator $f(\CMop)-f(\CMop[0])$ is trace class and
    \begin{equation}
        \label{eq:krein}
            \Tr \left( f(\CMop)-f(\CMop[0]) \right) = \int_\R \xi(\lambda) f'(\lambda)\dd\lambda \,.
    \end{equation}
    In particular, this gives
    $$
    \Tr \left( (\CMop-k)^{-1} - (\CMop[0]-k)^{-1} \right) = - \int_\R \frac{\xi(\lambda)}{(\lambda-k)^2}\dd\lambda \,.
    $$
    Meanwhile, combining Lemma \ref{lem:resdifftrace} with \eqref{eq:cauchymeasure}, we find
    $$
    \Tr \left( (\CMop-k)^{-1} - (\CMop[0]-k)^{-1} \right) = \int_\R \frac{d\nu(\lambda)}{\lambda-k}\dd\lambda
    $$
    with the measure $\nu$ from \eqref{eq:nu}. By uniqueness of the Cauchy--Stieltjes transform and the fact that $\xi$ vanishes near $-\infty$, a comparison of the previous two relations shows that $\xi(\lambda) = \nu((-\infty,\lambda))$ for a.e.~$\lambda\in\R$, that is,
    $$
    \xi(\lambda) = - \frac1{(2\pi)^2} \int_0^\lambda |\beta(\lambda')|^2\dd\lambda' - \#\{ j:\ \lambda_j<\lambda\} \,.
    $$
    Inserting this formula into \eqref{eq:krein} and integrating by parts we obtain the assertion of the theorem. Let us justify this integration by parts. From the definition of the class $\mathcal C$ we see that $f(\lambda) = c\lambda^{-1} + O(\lambda^{-1-\epsilon})$ for some constant $c$. Thus the integration by parts requires $M^{-1} \int_0^M |\beta(\lambda')|^2\dd\lambda'\to 0$ as $M\to\infty$, which follows from the square integrability of $\beta$, see Remark \ref{rem:b}.
\end{proof}


\section{Higher order trace formulas}\label{sec:tracehigher}

In this section we prove an infinite family of trace formulas, reminiscent of those proved by Zaharov and Faddeev \cite{zaharov_faddeev} for the Schr\"odinger operator.

To obtain the full family of trace formulas, we restrict ourselves to Schwartz class $q$ for the sake of simplicity. It is clear from the proof that this assumption can be weakened. We recall the sequence $\{\mn_n\}^\infty_{n=1}$, defined in Theorem \ref{thm:m0fullexpan}. 

\begin{theorem}\label{thm:schwartztraceform}
    Let $q\in\mathcal{S}_+(\R)$, and let $\{\lambda_j\}_{j=1}^N$ denote the eigenvalues of $\CMop$. Then, for each $n\in\N_0$, 
    \begin{align}\label{eqn:infinfamtraceform}
        - \int_{-\infty}^\infty\overline{q(x)}\mn_{n+1}(x)\dd x = \frac{1}{2\pi}\int_0^\infty \abs{\beta(\lambda)}^2\lambda^n\dd\lambda + 2\pi\sum_{j=1}^N \lambda_j^n \,.
    \end{align}
\end{theorem}

More explicitly, after \eqref{eqn:traceident1} the second and third trace formulas are given by 
\begin{align*}
    \int_{-\infty}^\infty\left(\overline{q}(-\ii\partial_x q)-\frac12|q|^4\right)\dd x
    & = \frac{1}{2\pi}\int_0^\infty \abs{\beta(\lambda)}^2\lambda\dd\lambda + 2\pi\sum_{j=1} \lambda_j \,, 
    \intertext{and }
    \int_{-\infty}^\infty\abs{-\ii\partial_x q-qC_+\abs{q}^2}^2\dd x
    & = \frac{1}{2\pi}\int_0^\infty \abs{\beta(\lambda)}^2\lambda^2\dd\lambda + 2\pi\sum_{j=1} \lambda_j^2 \,.
\end{align*}
Each of the terms in the trace identity \eqref{eqn:infinfamtraceform} corresponds to a conserved quantity of the continuum Calogero--Moser equation. In fact, combining these identities with the formal analysis in Section \ref{sec:timeevolscat} implies that the quantities 
\begin{align*}
    \int_{-\infty}^\infty \ov{q(x)}\mn_{n+1}(x) \dd x=\sclp{\CMop^{n}q,q}
\end{align*}
define conservation laws for \eqref{eqn:CMeqn} for all $n$, as proved rigorously in \cite[Lemma 2.4]{gerard_calogero--moser_2023}.

The proof of Theorem \ref{thm:schwartztraceform} relies on the identity in Lemma \ref{lem:riemannhilbtraceform} together with the full asymptotic series for $m_0$ and the decay estimate for $\beta$ from Theorem \ref{thm:m0fullexpan}. In passing, we note that Kaup and Matsuno derived similar identities for the Benjamin--Ono equation in \cite{kaup_inverse_1998}, although their analysis was not rigorous.

\begin{proof}
    Let us fix $M\in \N$. We want to take the limit as $s\rightarrow \infty$ with $k=\ii s$ in the identity \eqref{eqn:relationtrace} from Lemma \ref{lem:riemannhilbtraceform}. We consider the three terms separately.

    \medskip
    
    \emph{Contribution from $q$.}
    It follows from Theorem \ref{thm:m0fullexpan} that 
    \begin{align*}
        \int_{-\infty}^\infty \ov{q(x)} m_0(x,\ii s)\dd x
        = \sum_{n=1}^{M+1} \frac1{(\ii s)^n} \int_{-\infty}^\infty \ov{q(x)} \, \mn_{n}(x) \dd x + o_{s\rightarrow\infty}(s^{-M-1}).
    \end{align*}

    \medskip

    \emph{Contribution from the eigenvalues.}
    Since $N\coloneqq N(\CMop)<\infty$, we clearly have
    \begin{equation*} 
        \sum_{j=1}^N \frac{1}{\lambda_j-\ii s} = - \sum_{j=1}^N \frac{1}{\ii s}\frac{1}{1-\lambda_j/\ii s} = - \sum_{n=0}^M \frac{1}{(\ii s)^{n+1}}\sum_{j=1}^N\lambda_j^{n}+o_{s\rightarrow\infty}(s^{-M-1}).
    \end{equation*}
    
    \medskip

    \emph{Contribution from $\beta$.}
    Similarly, for the term with $\beta$, we find
    \begin{align*}
        \int_0^\infty \frac{|\beta(\lambda)|^2}{\lambda-\ii s}\dd\lambda=-\sum_{n=0}^M\frac{1}{(\ii s)^{n+1}}\int_0^\infty|\beta(\lambda)|^2 \lambda^n\dd \lambda+o_{s\rightarrow\infty}(s^{-M-1}),
    \end{align*}
    where we have used the series expansion above together with the rapid decay of $\beta$ from Theorem \ref{thm:m0fullexpan}. More precisely, we use $(1-z)^{-1} - \sum_{n=0}^M z^n = (1-z)^{-1} z^{M+1}$ with $z=\ii s/\lambda$ and control the remainder using dominated convergence.
    
    \medskip

    \emph{Putting everything together.}
    Inserting the contributions from $q$, from the eigenvalues and from $\beta$ into \eqref{eqn:relationtrace}, we obtain
    \begin{align*}
        \sum_{n=1}^{M+1} \frac1{(\ii s)^n} \int_{-\infty}^\infty \ov{q(x)} \, \mn_{n}(x) \dd x & = - \frac{1}{2\pi}\sum_{n=0}^{M}\frac{1}{(\ii s)^{n+1}}\int_0^\infty \abs{\beta(\lambda)}^2 \lambda^n\dd \lambda - 2\pi
        \sum_{n=0}^M \frac{1}{(\ii s)^{n+1}}\sum_{j=1}^N\lambda_j^{n} \\
        & \quad +o_{s\rightarrow\infty}(s^{-M-1}).
    \end{align*}
    The trace formulas are derived through an iterative process, repeatedly multiplying through by $\ii s$ and taking the limit $s\rightarrow \infty$. Since $M$ can be chosen arbitrarily large, we obtain the claimed infinite family of trace formulas.
\end{proof}


\newpage

\part{Inverse scattering theory}
\section{Time evolution of scattering data} \label{sec:timeevolscat}

In this section we argue non-rigorously that the scattering coefficients evolve in a simple manner under the continuum Calogero--Moser flow \eqref{eqn:CMeqn}. We do not aim for rigour here since, as noted in \cite{wu_jost_2017}, from the perspective of applying the inverse scattering transform, the formulas are only useful insofar as they yield solutions to \eqref{eqn:CMeqn}, which can be verified directly in applications.

To render these arguments rigorous, the crucial step is to show that the potentials remain in the space $L^2_+\cap L^1_+$ in which we have constructed the Jost solutions and scattering coefficients; without this preliminary result, there is no point in pursuing a fully rigorous treatment of the time evolution.

We recall that the Lax structure of \eqref{eqn:CMeqn} states that the continuum Calogero--Moser equation can be formulated equivalently as
\begin{align}\label{eqn:Laxeqn}
    \partial_t\CMop=[B_q,\CMop] \,,
\end{align}
where 
\begin{align*}
    L_q=-\ii \partial_x-\pot \quad \text{and}\quad B_q=\ii\partial_x^2+2 q\partial_xC_+\overline{q}.
\end{align*}
We take the skew-adjoint operator $B_q$ as in \cite{killip_sharp_2023}, as opposed to \cite{gerard_calogero--moser_2023}, since the Calogero--Moser equation \eqref{eqn:CMeqn} can be written equivalently as
\begin{align}\label{eqn:Laxeqn2}
    (\partial_t-B_q)q=0 \,.
\end{align}

Formally, the spectral data constructed in Sections \ref{sec:lap} and \ref{sec:jostinhom} satisfy
\begin{align}
    \partial_t\lambda_j&=0 \,, \label{eqn:lambdajevol} \\
    \partial_t\gamma_j&=-2\lambda_j \,, \label{eqn:gammajevol}\\
    \partial_t\Gamma(\lambda)&=0 \,, \label{eqn:Gammaevol}\\
    \partial_t\beta(\lambda)&=-\ii\lambda^2 \beta(\lambda) \,, \label{eqn:betaevol}
\intertext{and the eigenfunctions and Jost solutions evolve with}
    \partial_t\phi_j&=B_q\phi_j \,, \label{eqn:eigenfncevol}\\
    \partial_tm_0(k)&=B_qm_0(k) \,, \label{eqn:m0evol}\\
    \partial_tm_\e(\lambda\pm0\ii)&=B_q m_\e(\lambda\pm0\ii)+\ii\lambda^2m_\e(\lambda\pm0\ii) \,. \label{eqn:meevol}
\end{align}

    To show this, we will use the Lax structure \eqref{eqn:Laxeqn}. We will also use the fact that if $q$ is a solution to the CM equation then \eqref{eqn:Laxeqn2} holds.

\medskip
    
    We start by considering the eigenvalues $\lambda_j$ and eigenfunctions $\phi_j$. It follows by a standard argument that $\partial_t\lambda_j=0$, that is \eqref{eqn:lambdajevol}. Next, we take $\phi_j$ to solve 
    \begin{align*}
        \CMop\phi_j=\lambda_j\phi_j, \quad \text{with}\quad  \sclp{\phi_j,q}=2\pi\ii. 
    \end{align*}
    Taking the time derivative we get 
    \begin{align*}
        (\partial_t\CMop)\phi_j+\CMop(\partial_t\phi_j)=(\partial_t\lambda_j)\phi_j+\lambda_j(\partial_t\phi_j)
    \end{align*}
    which gives
    \begin{align*}
    0 = [B_q,\CMop]\phi_j+(\CMop-\lambda_j)\partial_t\phi_j=(\CMop-\lambda_j)(\partial_t\phi_j-B_q\phi_j) \,.
    \end{align*}
    From the simplicity of eigenvalues, determined in \cite{gerard_calogero--moser_2023} (see also the discussion after Lemma \ref{lem:qeigenfuncinprod}), we see that $(\partial_t -B_q)\phi_j=\alpha\phi_j$ for some $\alpha\in\C$. Then, using $(\partial_t-B_q)q=0$ and our normalisation, we have 
    \begin{align*}
        \sclp{(\partial_t -B_q)\phi_j,q}=\sclp{\partial_t\phi_j,q}+\sclp{\phi_j,B_qq}=\sclp{\partial_t\phi_j,q}+\sclp{\phi_j,\partial_t q}=\partial_t\sclp{\phi_j,q}=0
    \end{align*}
    which from Lemma \ref{lem:qeigenfuncinprod} implies that $\alpha =0$. Thus we have the evolution \eqref{eqn:eigenfncevol}. 

\medskip
    
    Next, we determine the evolution of $m_0$. For $k\in\R_+^\diamond$ we apply the operators $L_q$ and $B_q$ to $m_0(k)$, even if these functions are not in $L^2_+(\R)$. Following a similar approach to the above, we have  
    \begin{align*}
        (\CMop-k)(\partial_tm_0-B_qm_0)=(\partial_t-B_q)q=0.
    \end{align*}
    Then for $k\in \C\backslash[0,\infty)^\diamond$ we deduce that $(\partial_t-B_q)m_0=0$. For $k=\lambda\pm0\ii$ one needs to argue that $\ov{\e(\lambda)}(\partial_t m_0 - B_qm_0)\to 0$ as $x\to\mp\infty$. Therefore the uniqueness of Jost solutions $m_\e(\lambda\pm0\ii)$ implies that even in this case we have $(\partial_t-B_q)m_0=0$. Thus, for all $k\in\Chat$ we have the evolution \eqref{eqn:m0evol}.

    \medskip

    Similarly, for the Jost solutions $m_\e(\lambda\pm0\ii)$ (with the same caveat as for $m_0(\lambda\pm0\ii)$) we obtain
    $$
    (L_q - k) (\partial_tm_\e-B_qm_\e) = 0 \,.
    $$    
    Assuming we can show that $\ov{\e(\lambda)} (\partial_tm_\e-B_qm_\e)\to\alpha$ as $x\to\mp\infty$ for some constant $\alpha\in\C$, we obtain from the uniqueness of Jost solutions $m_\e(\lambda\pm0\ii)$ that
    \begin{align*}
        \partial_tm_\e-B_qm_\e =\alpha m_\e(\lambda\pm 0\ii).
    \end{align*}
    Then, since $\ov{\e(\lambda)} m_\e(\lambda\pm0\ii) \to 1$ as $x\to\mp\infty$ and
    \begin{align*}
        B_q(\e(\lambda))=-\ii\lambda^2 \e(\lambda)+2 q\partial_xC_+\overline{q}\ee(\lambda) \,,
    \end{align*}
    we should have $\alpha=\ii \lambda^2$. In this way we arrive at \eqref{eqn:meevol}.

    \medskip

    Now if we use the Laurent expansion for $m_0$ about $\lambda_j$, found in Theorem \ref{eqn:improvlaurentexp1},
    \begin{align*}
         m_0(k)=-\frac{\ii \phi_j}{k-\lambda_j}+(\gamma_j+x)\phi_j+o_{k\rightarrow\lambda_j}(1) \,,
    \end{align*}
    then applying $\partial_t-B_q$ to both sides, using \eqref{eqn:m0evol} and \eqref{eqn:eigenfncevol}, and taking $k=\lambda_j$ gives
    \begin{align*}
        0=(\partial_t\gamma_j)\phi_j+[(\gamma_j+x),B_q]\phi_j.
    \end{align*}
    We calculate 
    \begin{align*}
        [(\gamma_j+x),B_q]\phi_j&=- 2\ii \partial_x\phi_j + 2 xq\partial_xC_+\ov{q}\phi_j- 2 q\partial_xC_+\ov{q}x\phi_j\\
        &=-2\ii \partial_x\phi_j- 2 \pot \phi_j+ 2q\left(x C_+\partial_x(\ov{q}\phi_j) - C_+x\partial_x(\ov{q}\phi_j)\right)\\
        &=2 \CMop\phi_j=2 \lambda_j\phi_j,
    \end{align*}
    and find
    \begin{align*}
        (\partial_t\gamma_j+2\lambda_j)\phi_j=0
    \end{align*}
    which is the statement \eqref{eqn:gammajevol}.

    \medskip
        
    To study \eqref{eqn:Gammaevol} apply $\partial_t-B_q$ to \eqref{eqn:megammarel} and recall \eqref{eqn:meevol}, to obtain 
    \begin{align*}
        \ii\lambda^2m_\e(\lambda+0\ii)
        &=(\partial_t\Gamma(\lambda))m_\e(\lambda-0\ii)+\Gamma(\lambda)(\partial_t-B_q)m_\e(\lambda-0\ii)\\
        &=(\partial_t\Gamma(\lambda) +\ii \lambda^2\Gamma(\lambda))m_\e(\lambda-0\ii)
    \end{align*}
    which, after using \eqref{eqn:megammarel} again, gives \eqref{eqn:Gammaevol}. 

    \medskip
    
    Finally, we do the same with \eqref{eqn:m0mebetarel} and recall \eqref{eqn:m0evol} and \eqref{eqn:meevol} to obtain
    \begin{align*}
        0=(\partial_t-B_q)(\beta(\lambda)m_\e(\lambda-0\ii))&=(\partial_t\beta(\lambda))m_\e(\lambda-0\ii)+\beta(\lambda)(\partial_t-B_q)m_\e(\lambda-0\ii)\\
        &=(\partial_t\beta(\lambda)+\ii \lambda^2\beta(\lambda))m_\e(\lambda-0\ii)
    \end{align*}
    which gives \eqref{eqn:betaevol}.

\subsection*{Acknowledgement}
    The first author is grateful for stimulating discussions with Enno Lenzmann that motivated this project.
    
\subsection*{Funding}
    Partial support was provided through US National Science Foundation grant DMS-1954995 (R.L.F.), the German Research Foundation grants EXC-2111-390814868 and TRR 352-Project-ID 470903074 (R.L.F. \& L.R.), and by the European Research Council (ERC) under the European Union’s Horizon 2020 research and innovation programme (grant agreement No. 101097172 – GEOEDP, L.R.). 


\end{document}